\documentclass[11pt]{amsart}
\usepackage[inner=3cm, outer=3cm, top=3cm, bottom=3cm]{geometry}
\usepackage{latexsym}
\usepackage{fullpage}
\usepackage{amscd}
\usepackage{float}
\usepackage{amsfonts}
\usepackage{pb-diagram}
 \usepackage{amsmath,amscd}
\usepackage{color,graphicx,times}
\usepackage{xcolor}
\usepackage[export]{adjustbox}
\usepackage{amsmath}
\usepackage{amssymb}
\usepackage{todonotes}
\usepackage{amsthm}
\usepackage{mathtools}
\usepackage{breqn}
\usepackage{physics}
\usepackage[T1]{fontenc}
\usepackage{pst-node}
\usepackage{tikz-cd}
\usepackage{subcaption}
\usepackage[utf8]{inputenx}
\usepackage{indentfirst}
\usepackage{microtype}
\usepackage{graphicx} 
\usepackage{enumerate}
\usepackage{mdframed}
\usepackage{breqn}

\usepackage{coffeestains}

\theoremstyle{plain}
\newtheorem{theorem}{Theorem}[section]
\newtheorem{corollary}[theorem]{Corollary}
\newtheorem{proposition}[theorem]{Proposition}

\theoremstyle{definition}
\newtheorem{definition}[theorem]{Definition}

\newtheorem{remark}[theorem]{Remark}
\newtheorem{remarks}[theorem]{Remarks}

\newcommand{\cI}{\mathcal{I}}
\newcommand{\cR}{\mathcal{R}}

\begin{document}

\title[Pure braids and the Bonded braid monoid]
  {On pure braids and bonded braids}

\author{Sofia Lambropoulou}
\address{School of Applied Mathematical and Physical Sciences, National Technical University of Athens, Zografou campus, GR-15780 Athens, Greece.}
\urladdr{http://www.math.ntua.gr/~sofia}

\author{Lucrezia Beatrice Lorenzi}
\address{Scuola Internazionale Superiore di Studi Avanzati, Via Bonomea 265, 34136 Trieste, Italy, \& School of Applied Mathematical and Physical Sciences, National Technical University of Athens, Zografou campus, GR-15780 Athens, Greece.}

\date{September 2026}
\keywords{braids, pure braids, bonded braid monoid, singular braid monoid}

\subjclass[2020]{57K10, 57K12, 57K14, 20F36, 20F38, 57K99, 92C40}

\setcounter{section}{-1}

\begin{abstract}

In this paper we investigate the relation between the bonded braid monoid $BB_n$ and the group of pure braids $P_n$, realised by an algebraic analogue of the topological tangle insertion, namely by replacing a bond by (a power of) a pure braid generator. To do so, we provide appropriate presentations for $P_n$ using cyclic relations and clasp generators, and extend them to presentations of $B_n$. We construct the monoid epimorphism $\Phi \colon BB_n \to B_n$, mapping bonds to pure braid generators, which is the main result of the paper. Using the isomorphism from the singular braid monoid $SB_n$ to the tight bonded braid monoid $BB_n$, we extend $\Phi$ to a new type of epimorphism $s\Phi \colon SB_n \to B_n$. The epimorphisms $\Phi$ and $s\Phi$ are extended to families of epimorphisms $\Phi_p$, $s\Phi_p$, $p \in \mathbb{Z}$.
\end{abstract}

\maketitle

\section*{Introduction} 

First-born in their family, \textit{bonded knotoids} were introduced in \cite{GGLDSK} as suitable topological models for folded proteins. Roughly, bonded knotoids are open knotted curves (\textit{knotoids}) endowed with non-intersecting arcs embedded in their complement, whose endpoints are attached to the knotoid. These arcs are named \textit{bonds}. \textit{Bonded knots} are classical knots endowed with bonds, and are included in the class of bonded knotoids in the form of bonded knot type knotoids, which have endpoints lying in the same connected region. Bonded knots are further discussed in \cite{G}. 

A natural direction would be to inquire the existence of an algebraic structure to support analogues of the Alexander and Markov theorems for bonded knots. \textit{Bonded braids} were introduced in \cite{DKL}, where the authors also discussed the topology of bonded knots, polynomial invariants, as well as statements and proofs for suitable Theorems of Alexander and Markov. Bonded braids are classical braids with extra structure, the \textit{bonds}: horizontal arcs embedded in the complementary of the braid, with endpoints lying on distinct strands. See an example in Figure~\ref{bondedbraidfig}. Bonded braids on the same number $n$ of strands have a monoid structure under the classical composition operation of vertical stacking of braids. The \textit{bonded braid monoid} $BB_n$ has a presentation with generators the classical braid crossings, their inverses (Figure~\ref{braidgensfig}) and all bonds, with relations deriving from all possible isotopy interactions among bonds and of bonds with crossings.

Singular braids on the same number $n$ of strands also form a monoid, the \textit{singular braid monoid} $SB_n$ \cite{Baez}, \cite{Birman}. In fact in \cite{DKL} it is proved that there exists an isomorphism from $SB_n \to BB_n$.

\begin{figure}[H]
\begin{center} 
    \includegraphics[width=1.5cm]{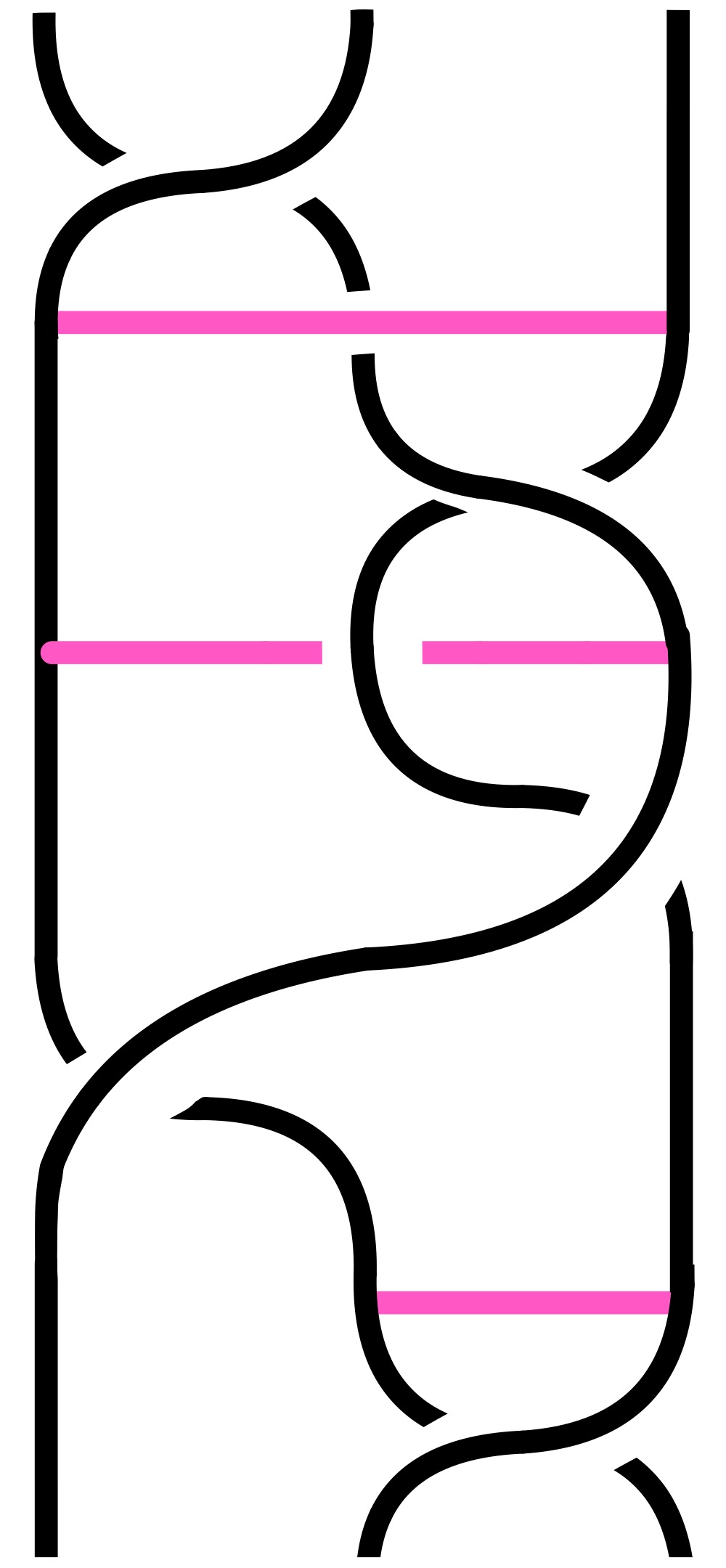}
\end{center}
\caption{ A bonded braid. }
\label{bondedbraidfig}
\end{figure}  

In this work we offer further insight on the bonded braid monoid by understanding its relation with a well-known mathematical object, namely the group of pure braids on $n$ strands, $P_n$. Our main result is the existence of an epimorphism $\Phi \colon BB_n \to B_n$ mapping bonds to pure braid generators. The motivation is to pave the road to studying bonded (and singular) knots by means of a tangle insertion. Topologically, in a tangle insertion one replaces a knot portion contained in a disc with a designated tangle that attaches compatibly on the boundary: if objects are oriented, orientation is also assumed to be compatible. For bonded knots, one can replace a disc neighbourhood of a bond with a designated tangle. Algebraically, in this work the bond generator $b_{i j}$  in the bonded braid monoid is replaced with Artin's pure braid generator $A_{i j}$ (or any power of it $A_{ij}^p$, $p \in \mathbb{Z}$):
\begin{equation*}
A_{ij} \coloneqq \sigma_i^{-1} \sigma_{i+1}^{-1} \cdots \sigma_{j-2}^{-1} \sigma_{j-1}^2 \sigma_{j-2} \cdots \sigma_i
\end{equation*}
This elementary pure braid takes the $i$-th strand, passes it above all those preceding the $j$-th, makes it turn around the $j$-th passing it under at first and then above it, and then returns the $i$-th strand back in its position by passing once again above all others on the way back (see Figure~\ref{Agen}). By conjugation, pure braid generators can also thread under intermediate strands (see Figure~\ref{conj}).
This choice of braid insertion stems from a fundamental property of bonds, that they preserve connectivity: bonds do not switch the indices of the bonded strands. Braids, however, carry with them a possibly non-trivial permutation datum of the strand indices. 

Alike $A_{ij}$, bonds in $BB_n$  can also pass under intermediate strands to obtain conjugate expressions. If a bond $b$ joining the $i$-th and $j$-th strands passes under a collection of intermediate strands with indices an ordered subset $\mathcal{I} = (l_1, \dots , l_k)$ of $\left\{i+1, \, \dots, j-1\right\}$, it will be denoted as $b_{i j}^{\mathcal{I}}$. Using the notation in \cite{DKL}, where to a bond is associated an ordered sequence of labels $u$ and $o$, denoting underpassings and overpassings of the bond, $\mathcal{I}$ is the set of indices of the $u$'s. 
With the above convention, the notation $b_{ij}$ means that the bond passes above all intermediate strands. The set of all bonds $b_{ij}^{\cI}$ form, together with the elementary braids, a generating set for $BB_n$ in \cite{DKL}. 

By an appropriate \textit{threading algorithm}  on the bonds $b_{ij}^{\cI}$ (Definition~\ref{threadingdef}), we show in this paper that the overpassing bonds $b_{ij}$ together with the elementary braids suffice for a presentation of $BB_n$ (Proposition~\ref{oversuff} and Theorem~\ref{oBB}). We proceed with showing in Theorem~\ref{mainthm} that there exists an epimorphism $\Phi$ of the bonded braid monoid $BB_n$ onto the braid group $B_n$, that sends the bonds $b_{ij}$ to pure braid generators $A_{ij}$ and extends naturally to mapping the general bonds $b_{i j}^{\mathcal{I}}$ to appropriate conjugate elements of $A_{ij}$:
$$\Phi(\sigma_k^{\pm 1}) = \sigma_k^{\pm 1}, \quad \Phi(b_{ij}) = A_{ij}, \quad \Phi(b_{ij}^{(l_1\cdots l_k)}) = A_{i,l_1} \, \cdots \, A_{i,l_k} \, A_{ij} \, A_{i,l_k}^{-1} \, \cdots \, A_{i,l_1}^{-1}$$
The surjection of $\Phi$ is straightforward by the identical mappings of crossings. Furthermore $\Phi$ is far from being injective, since for example the braid $A_{ij}$ and the bond $b_{ij}$ are both mapped to $A_{ij}$.

To prove that $\Phi$ is a homomorphism we use for $BB_n$ the presentation in the overpassing bonds, while for $B_n$ we derive in Section~\ref{sectionalternative} a suitable presentation from a presentation of $P_n$, which we call the \textit{cyclic presentation}. In this presentation, Artin's original relations \cite{Art} are modified to focus on how generators commute with each other, rather than on formulating a combing algorithm. So, in this presentation for $B_n$, pure braid generators behave similarly to how bonds do in $BB_n$, thus clarifying how $\Phi$ acts. 

We further derive in Theorem~\ref{mainthm2} an infinitum of epimorphisms $\Phi_p$, $p \in \mathbb{Z}$, of $BB_n$ onto $B_n$ by mapping the bond $b_{ij}$ to the element $A_{ij}^p$, obtaining thus a multitude of algebraic tangle insertions for bonds, with $\Phi_0$ being the forgetful map.

\smallbreak
In \cite{DKL} the original presentation of $BB_n$ is reduced to having as bonded generators the \textit{tight bonds}, $b_i := b_{i,i+1}$, which are bonds attached to adjacent strands, with  a reduced set of relations. The epimorphism $\Phi$ (resp. $\Phi_p$) restricted to the tight bonds, maps $b_i$
to the \textit{clasp} $p_i:=\sigma_i^2$ (resp. to $\sigma_i^{2p} = p_i^{p}$). For understanding how these epimorphisms work, we provide a new, redundant, presentation of $B_n$, with generators elementary braidings and clasps. By noting that Artin's generators are conjugations of clasps, in Theorem~\ref{Pclaspres} we conclude that these provide a set of normal generators for $P_n$, which we complete to a full presentation of $B_n$ in Theorem~\ref{braidpreswclasp}. 

On the other hand, singular braids are classical braids with \textit{singular crossings}, where the two strands actually intersect transversally, such that a neighbourhood of the crossing lies flat in a disc centered at the intersection. Thanks to the isomorphism $s$ between $SB_n$ and $BB_n$, the epimorphisms $\Phi$ and $\Phi_p$ give rise to an abundance of new epimorphisms of the singular braid monoid $s\Phi \colon SB_n \to BB_n$ and $s \Phi_p \colon SB_n \to BB_n$ that, unlike the singular crossings, preserve the connectivity of strands.

Our results also extend to irredundant presentations of the tight bonded braid monoid, respectively the singular braid monoid, having only one tight bond $b_1$ (see \cite{DKL}), respectively one singular crossing $\tau_1$. The epimorphisms $\Phi, \Phi_p$ and $s\Phi$, $s\Phi_p$ can then connect $BB_n$, respectively $SB_n$, to a presentation for $B_n$ with only one clasp generator $p_1$, which mirrors the behaviour of the single tight bond or the single singular crossing generator.

\smallbreak
Here follows a summary of this work. In Section 1, we recall the basics for the braid group $B_n$ and the pure braid group $P_n$, Artin's classical presentation of the latter, and some useful properties, such as the combing of a pure braid and how conjugation behaves. 
In Section 2, we recover a known alternative presentation for $P_n$ which will link to bonds, and offer a presentation for $B_n$ constructed from the alternative presentation for $P_n$.
In Section 3 we introduce bonded braids and the bonded braid monoid $BB_n$ as from \cite{DKL}. Then, we give an alternative presentation for $BB_n$ in the overpassing bond generators and the classical crossings.
In Section 4 we state and prove our main result, namely the existence and well-posedness of the epimorphism $\Phi \colon BB_n \to B_n$ mapping (overpassing) bonds to pure braid generators, and state our results on $\Phi$, and we extend our results to the epimorphisms $\Phi_p$.
In Section 5, we work on a set of normal generators for $P_n$ in $B_n$, and study how the relations for the pure braid group are reduced when \textit{contracted}, so when the conjugation blocks are simplified. This leads to another alternative yet redundant presentation for $B_n$ in the clasp generators.
In Section 6, we get back to $BB_n$ and recall its presentation with only tight bonds and the \textit{irredundant} presentation with only one bond generator. For both presentations, we analyse the relation with the singular braid monoid $SB_n$.
In conclusion, in Section 7 we extend the results of Section 4 to the case of the tight bonded braid monoid and the singular braid monoid via the restrictions of $\Phi \colon BB_n \to B_n$ and $s\Phi \colon SB_n \to B_n$ (respectively, $\Phi_p$ and $s\Phi_p$).


\setcounter{section}{0}
\section{The braid group and the pure braid group}

The classical {\it braid group on $n$ strands}, $B_n$, is the Artin braid group generated by $\sigma_1, \ldots, \sigma_{n-1}$ and relations:
\begin{equation} \label{braidgrp} \begin{gathered}
    \sigma_i \, \sigma_j \, = \, \sigma_j \, \sigma_i \, \hspace{2mm} \text{for any} \hspace{2mm} |i-j|>1 \\
    \sigma_i \, \sigma_{i+1} \, \sigma_i \, = \, \sigma_{i+1} \, \sigma_i \, \sigma_{i+1} \hspace{2mm} \text{for} \hspace{2mm} 1 \leq i \leq n-2
\end{gathered}
\end{equation}
An \textit{algebraic braid} is a word in these generators or their inverses. Algebraic braids are in correspondence with \textit{geometric braids}, which  are regular projections of monotonic configurations of $n$ strands in a thickened disc.
Regular means that braids are projected transversally on a suitable plane such that they are Morse regular with respect to the up-down direction, with a finite amount of transversal double points, and points of order at most two.  The correspondence is given by sending the generator $\sigma_i \in B_n$ in the following elementary braid:

\begin{figure}[H]
\begin{center} 
    \includegraphics[width=4.5cm]{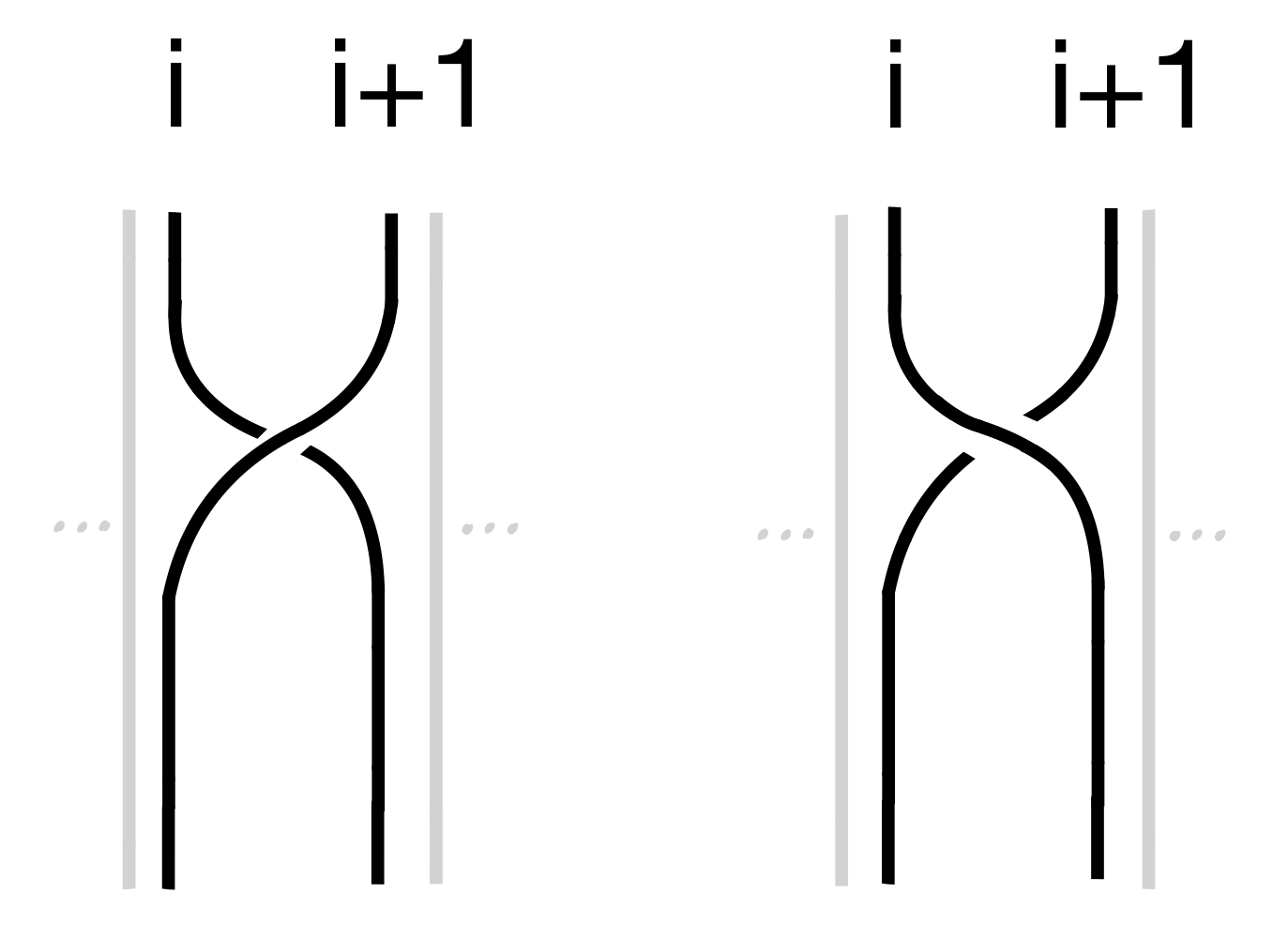}
\end{center}
\caption{ The i-th braid generator $\sigma_i$ (left) and its inverse $\sigma_i^{-1}$ (right).}
\label{braidgensfig}
\end{figure}  

The identity braid is given by $n$ straight strands. Composition is diagrammatically read from top to bottom. 
This type of two-dimensional pictures is often referred to as a braid diagram. The two relations in the braid group presentation graphically correspond to the ones depicted in Figure~\ref{braidrels}. The second one, with its derived relations, can also be referred to as the \textit{third Reidemeister move}.

\begin{figure}[H]
\begin{center} 
    \includegraphics[width=10cm]{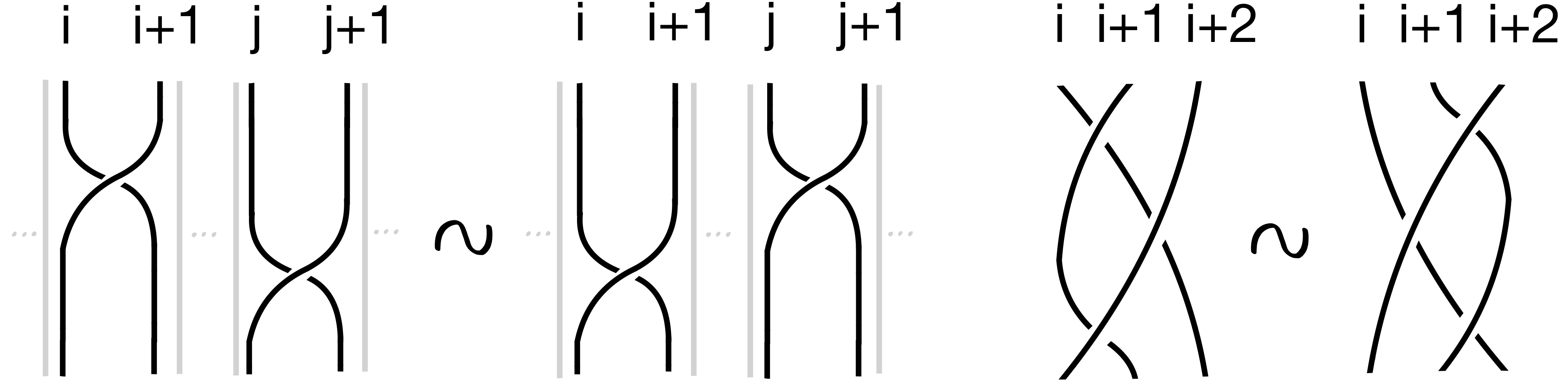}
\end{center}
\caption{ The braid relations. }
\label{braidrels}
\end{figure} 

A braid in the braid group on $n$ strands, $B_n$, carries the datum of a permutation in the symmetric group on $n$ objects, $S_n$, which sends $i$ to the position of the endpoint of the $i$-th strand (strands are numbered from left to right). Algebraically, the generator $\sigma_i$ is sent to the elementary transposition $s_i$ between $i$ and $i+1$.  In fact, this mapping is a group epimorphism, and a presentation for $S_n$ can be obtained by that of $B_n$ by imposing the additional relation $\sigma_i^2=1$, that is, $S_n = B_n/<\sigma_i^2=1>$. Namely: 
\begin{equation*} \label{Snpres}
    S_n = \langle \, s_1, \ldots, s_{n-1} \, | \, s_i s_j = s_j s_i \ \text{for} \ |i-j|>1; \ s_i s_{i+1} s_i = s_{i+1} s_i s_{i+1}, \ s_i^2 = 1 \ \forall \, i \, \rangle 
\end{equation*}
There are infinitely many different braids corresponding to the same permutation. Projecting a braid on its induced permutation gives rise to a surjection from $B_n $ to $ S_n$, which fits in a short exact sequence:
\begin{equation} \label{ses}
    1 \rightarrow P_n \hookrightarrow B_n \twoheadrightarrow S_n \rightarrow 1
\end{equation}
The normal subgroup $P_n$, the surjection's kernel, is the so-called \textit{pure braid group on n strands},  which consists in all braids in $B_n$ which surject to the trivial permutation. 

\subsection{Generators for $P_n$}

As a subgroup of $B_n$, there are several choices for sets of generators for $P_n$. The choice taken in \cite{Art}  is the following: 

\begin{equation} \label{AG}
    \left\{ \ A_{ij} = \sigma_i^{-1} \cdots \sigma_{j-2}^{-1} \, \sigma_{j-1}^2 \, \sigma_{j-2} \cdots \sigma_i \ \right\}_{1 \leq i<j \leq n}
\end{equation}

\noindent with inverses $A_{ij}^{-1} = \sigma_i^{-1} \cdots \sigma_{j-2}^{-1} \, \sigma_{j-1}^{-2} \, \sigma_{j-2} \cdots \sigma_i$. These are the {\it right pure braid generators} directed from left to right, as depicted in Figure~\ref{Agen}, and we shall also refer to them as  \textit{right loop generators}. 

\begin{figure}[h]
\begin{subfigure}{0.4\textwidth}
\includegraphics[width=0.69\linewidth]{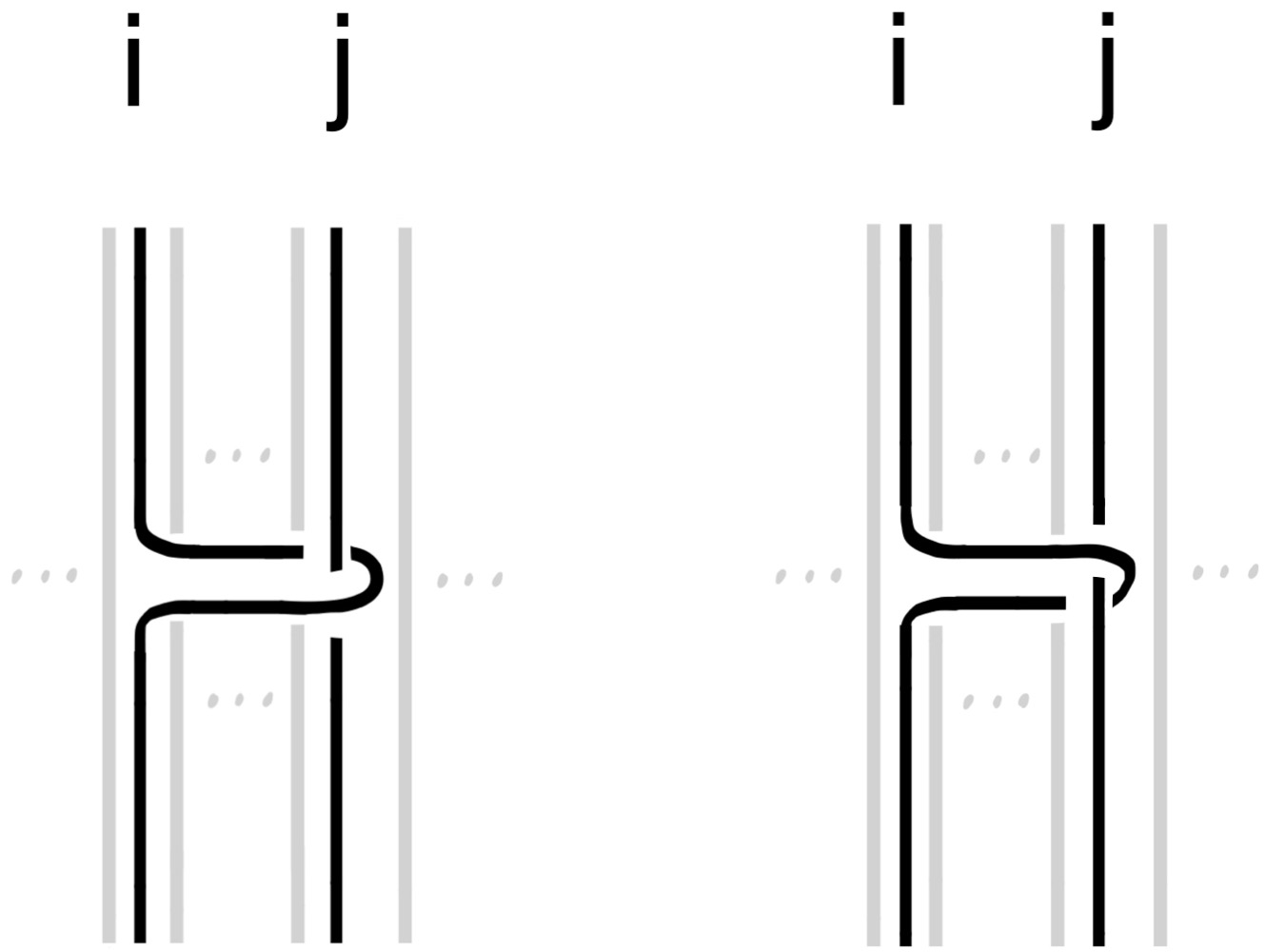}
    \caption{Right loop generators and their inverses.}
\label{Agen}
\end{subfigure}
\begin{subfigure}{0.4\textwidth}
\includegraphics[width=0.7\linewidth]{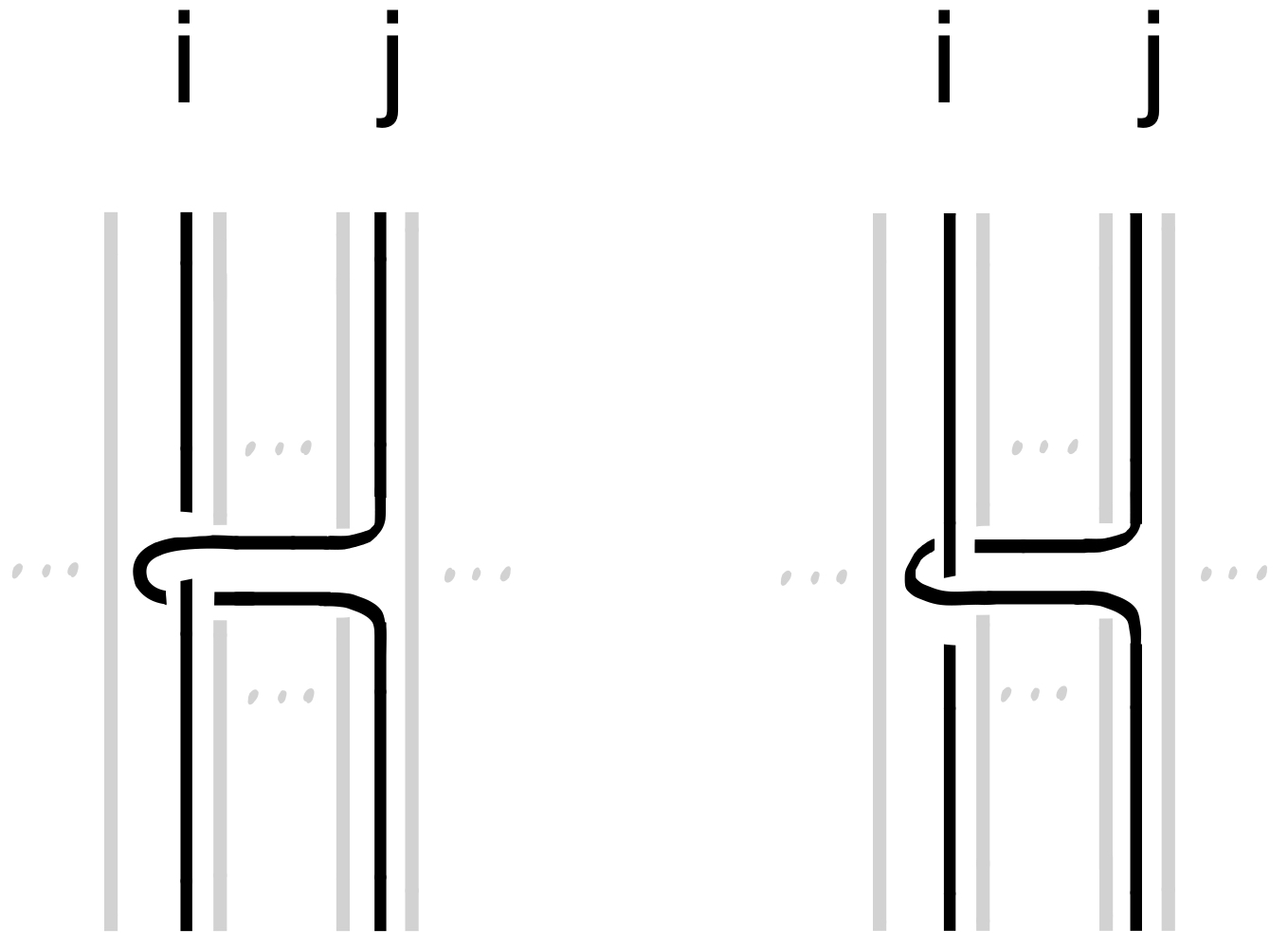}
    \caption{Left loop generators and their inverses.}
\label{Lgen}
\end{subfigure} 
\end{figure} 

\begin{remark} In the braid group, one can deduce the relations \begin{equation} \label{braidsqrels}
    \sigma_i \sigma_{i+1}^{\pm 2} \sigma_i^{-1} = \sigma_{i+1}^{-1} \sigma_i^{\pm 2} \sigma_{i+1} \quad \text{and} \quad \sigma_i^{-1} \sigma_{i+1}^{\pm 2} \sigma_i = \sigma_{i+1} \sigma_i^{\pm 2} \sigma_{i+1}^{-1}.
\end{equation}
Indeed, the first relation with $\sigma_{i+1}^2$ is equivalent to
\begin{equation*}
    \sigma_{i+1} \sigma_i \sigma_{i+1} \sigma_{i+1} = \sigma_i \sigma_i \sigma_{i+1} \sigma_i
\end{equation*}
Using the second of relations Eq.~\ref{braidgrp} for braids on both hand sides we obtain equality:

\begin{equation*}
\sigma_i\sigma_{i+1}\sigma_i\sigma_{i+1} = \sigma_i\sigma_{i+1}\sigma_i\sigma_{i+1}.
\end{equation*}

The proof for the negative exponent and both instances of the second equation in Eq.~\ref{braidsqrels} are analogues.
\end{remark}

Using these relations, we can prove that an equivalent set of generators for the pure braid group consists in the \textit{left pure braid generators} (or \textit{left loop generators}), depicted in Figure~\ref{Lgen}:

\begin{equation} \label{LG}
\left\{ A_{ij} \coloneqq \sigma_{j-1} \cdots \sigma_{i+1} \, \sigma_i^{2} \, \sigma_{i+1}^{-1} \cdots \sigma_{j-1}^{-1} \right\}_{1 \leq i<j \leq n}
\end{equation}

with inverses $A_{ij}^{-1} = \sigma_{j-1} \cdots \sigma_{i+1} \, \sigma_i^{-2} \, \sigma_{i+1}^{-1} \cdots \sigma_{j-1}^{-1}$, used in \cite{PhDLam}.

\begin{remark} Inverting a pure braid generator consists in inverting only the core double twist; the two blocks of overpassings remain the same. This is because $A_{ij}$ is a conjugate of the double twist. 
\end{remark}

\begin{remark} Note that imposing $\sigma_i^2=1$ for all indices trivialises all of the $A_{ij}$'s and their inverses: it neutralises the central $\sigma_{j-1}^{ \pm 2}$, and the blocks on its right and left are inverses to each other. More generally, any expression of the style $\sigma_i^{\varepsilon_i} \cdots \sigma_{j-2}^{\varepsilon_{j-2}} \sigma_{j-1}^{2(\varepsilon_{j-1})} \sigma_{j-2}^{\varepsilon'_{j-2}} \cdots \sigma_i^{\varepsilon'_i}$, with $\{\varepsilon_k, \, \varepsilon'_k\} = \{+,-\}$, will get cancelled by imposing $\sigma_i^2=1$. Remind that $B_n / \langle \, \sigma_1^2 \, \rangle \cong S_n$
\end{remark}

Any placement of the double twist in between the strands $i,j$ and above the intermediate ones is also equivalent to the pure braid generator $A_{ij}$, as one can see by geometric or algebraic means. See Figure~\ref{Eqgen}. The middle illustration of the figure presents a symmetric geometric representation for $A_{ij}$, which can be deduced algebraically by the iterations in Eq.~\ref{braidsqrels}. 

\begin{figure}[H]
\begin{center} 
    \includegraphics[width=8cm]{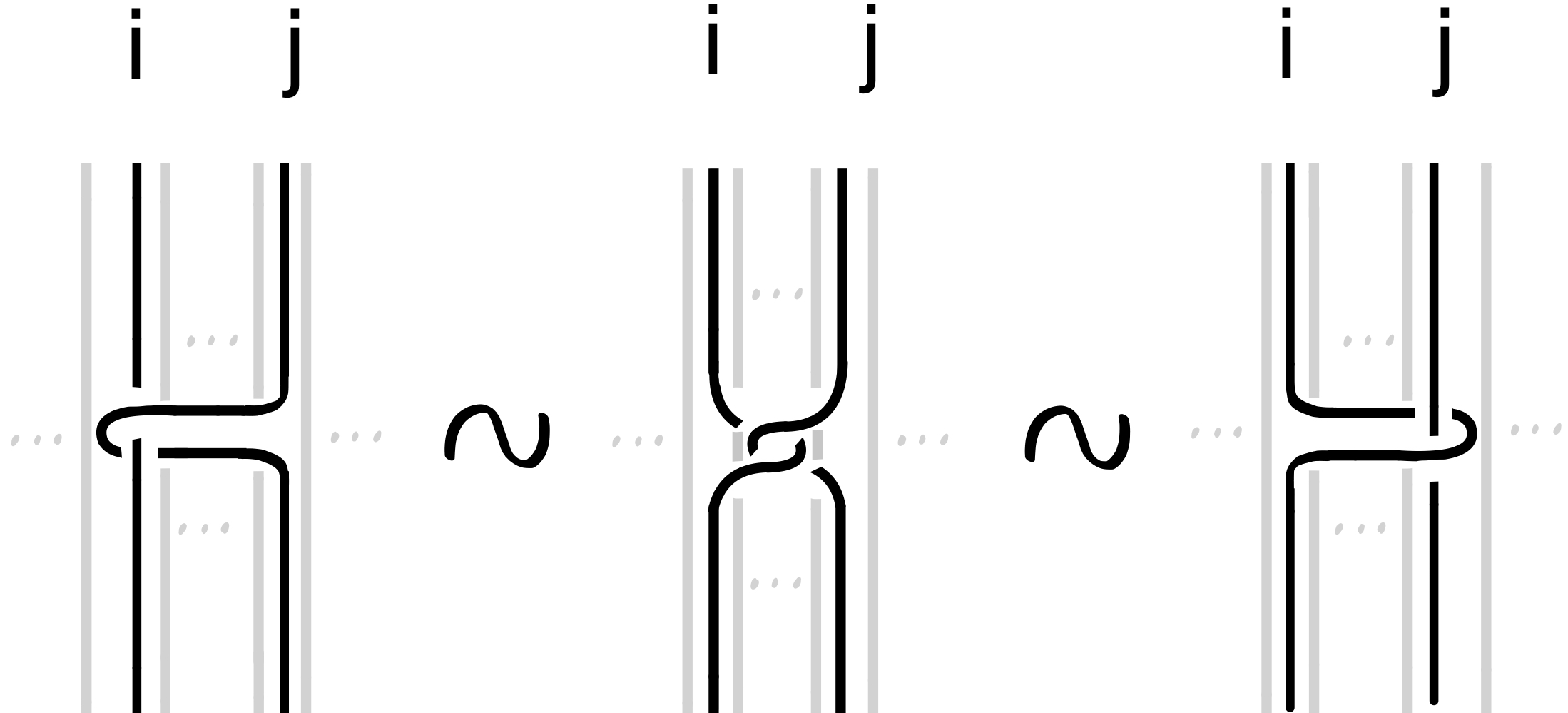}
\end{center}
\caption{Different isotopic placements of the double twist give rise to the left, symmetric and right generators.}
\label{Eqgen}
\end{figure}

\subsection{Presentations for $P_n$} 

A presentation for the pure braid group was found by Artin in the seminal article \cite{Art} using the right pure braid generators of Eq.~\ref{AG}. 
 
\begin{theorem}[Artin]\label{thm:Ar} The pure braid group on $n$ strands, $P_n$, is generated by the $\frac{n(n-1)}{2}$ elements $A_{lm}, \, 1 \leq l<m \leq n$, in Eq.~\ref{AG}, and their inverses, with relations 
\[ \begin{array}{lrcll}
      1)& \, A_{rs}A_{ij}A_{rs}^{-1} & = &A_{ij}   & \hspace{2mm} \text{for} \hspace{2mm} i<r<s<j \hspace{2mm} \text{,} \hspace{2mm} i<j<r<s; \\
     2)& \, A_{rs}A_{ir}A_{rs}^{-1} & = &  A_{is}^{-1}A_{ir}A_{is} & \hspace{2mm} \text{for} \hspace{2mm} i<r<s \\
     3)& \, A_{rs}A_{is}A_{rs}^{-1} & = &  A_{is}^{-1}A_{ir}^{-1}A_{is}A_{ir}A_{is} & \hspace{2mm} \text{for} \hspace{2mm} i<r<s \\
     4)& \, A_{rs}A_{ij}A_{rs}^{-1} & = &  A_{is}^{-1}A_{ir}^{-1}A_{is}A_{ir}A_{ij}A_{ir}^{-1}A_{is}^{-1}A_{ir}A_{is} & \hspace{2mm} \text{for}\hspace{2mm} i<r<j<s.
\end{array}
\]
\end{theorem}

The geometrical counterparts of the pure braid relations are shown in Figure~\ref{ArtRel}.

\begin{figure}[H]
\begin{center} 
    \includegraphics[width=17cm]{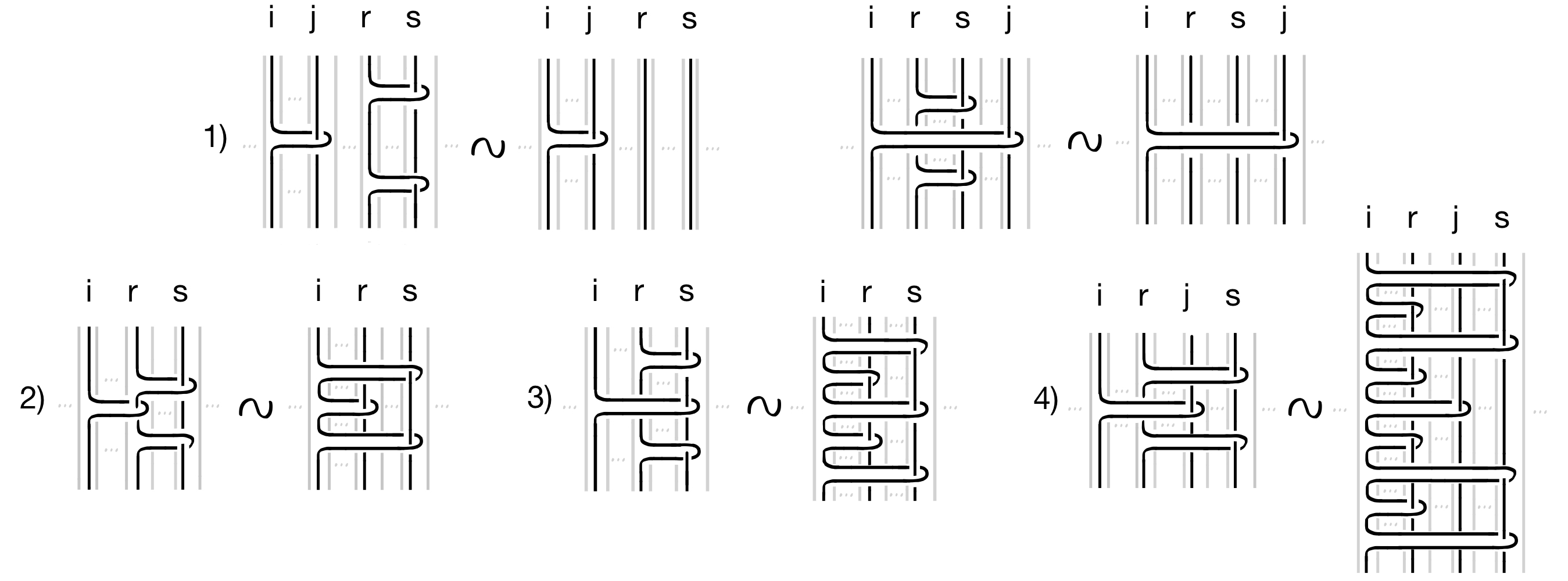}
\end{center}
\caption{Geometric interpretation of Artin's pure braid relations.}
\label{ArtRel}
\end{figure}

Note that the generators appearing in the braids on the right-hand sides have the same first index: this is because the relations show the result of the “combing algorithm” on the pure braids on the left-hand side. This combing algorithm will be explained in more detail below in Remark~\ref{combing}. 

A pure braid in $P_n$ is said to be \textit{$i$-pure} if it can be expressed as a word in the generators $A_{ij}$ for fixed $i$.
Two words are said to be {\it equivalent} if they represent the same element in the group $P_n$. Topologically, these words are said to be {\it isotopic}. Further, two (ordered) pairs of indices,  according to the same ordering relation, \textit{do not separate each other} if, imposing the order relation, at least one of the two pairs' elements are consecutive. Artin's theorem is actually more general, at least regarding the indices, and provides a wide set of relations holding in the pure braid group. More precisely: 

\begin{theorem}[Artin] \cite{Art}
    For $\varepsilon = \pm 1$, the braid $A_{rs}^{\varepsilon}A_{ij}A_{rs}^{- \varepsilon}$ in Eq.~\ref{AG}, is i-pure and can be expressed as:
    \[\begin{array}{lrl}
         1) &  A_{ij} & \text{ if all indices differ and the pairs } (i, j), \, (r, s) \text{  }  \\ & & \text{ do   not separate each other,} \\
         2) & A_{is}^{- \varepsilon} A_{ir} A_{is}^{\varepsilon} & \text{ if } j=r,  \\
         3) &  A_{is}^{-\varepsilon} A_{ir}^{-\varepsilon} A_{is} A_{ir}^{\varepsilon} A_{is}^{\varepsilon} & \text{ if } j = s,  \\
         4) & A_{is}^{-\varepsilon} A_{ir}^{-\varepsilon} A_{is}^{\varepsilon} A_{ir}^{\varepsilon} A_{ij} A_{ir}^{-\varepsilon} A_{is}^{-\varepsilon} A_{ir}^{\varepsilon} A_{is}^{\varepsilon} & \text{ if all indices differ and the pairs } (i, j), \, (r, s)  \\ & &
         \text{ separate each other}
    \end{array}\]
where in the expressions 2), 3), 4) we change, if necessary, first the indices (r, s) in such a way that
the arrangement (i, r, s) as compared with the natural arrangement of these three
indices is a permutation with the same sign as $\varepsilon$.
\end{theorem} 

\noindent In \cite{Art} it is argued that it is sufficient to pick relations for $i$ the smallest index and for $i<r<j<s$, so as to have $\varepsilon = +1$ as in the statement of Theorem~\ref{thm:Ar}.

\begin{remark} \label{combing}
\textit{Artin's canonical form} is the result of \textit{Artin's combing algorithm}, which is related to the fact that the first indices of the words in the right-hand sides of all the relations in Theorem~\ref{thm:Ar} are the same. The combing of the $i$th strand may be regarded as a loop in the complement space of the other strands, and as such is an element of a free group, since the fundamental group of a punctured disc is free. This means that the words in the left-hand sides of the relations have canonical expressions as words in free groups. 
This is the main idea for finding the given presentations for $P_n$ (see analysis further below).  

The resulting \textit{canonical form} related to Theorem~\ref{thm:Ar} consists in dividing a pure braid in $P_n$ in $n-1$ blocks, such that in each one of them there is only one strand that moves around and braids with the others, which remain straight. In the first block, the first strand braids with the other $n-1$. In the second block, the second strand braids with the $n-2$ strands to its right: it is not allowed to braid with the first one. And so on: in the $i$-th block, the $i$th strand braids with the $n-i$ strands at its right, and not with those that have already carried out the algorithm.  The $i$-th block can be seen as a word in the free group $F_{n-i}$, generated by the pure braid generators $A_{ik}$ for $k = i+1, \dots, n$ to each of the $n-i$ strands that the $i$-th strand braids with. 

The given presentation for $P_n$ is then derived as follows: the pure braidings of the $1$st strand with the previous ones are words in the free group on $n-1$ generators $F_{n-1} = \langle A_{1, 2}, \ldots, A_{1,n}\rangle$. Then  $P_{n-1}$ acts on $F_{n-1}$ by conjugation, thus yielding the relations in the presentation. Hence:
 $$ P_n =  F_{n-1} \rtimes F_{n-2} \rtimes \cdots \rtimes F_1 = F_{n-1} \rtimes P_{n-1}, $$ 
where each $F_i$ is a free group on the generators \ $A_{i,i+1}, \ldots, A_{i,n}$ \
(the loop generators between the $(i+1)$-st strand and all its consecutive ones), and where the action is by conjugation. 
Analogous reasoning holds for the relations in Theorem~\ref{thm:Ar}.
\end{remark}

\begin{remark} \label{automorphism}
One can also implement a combing algorithm in $P_n$ by beginning from the last strand and working their way to the first one. The resulting canonical form is related to the left loop generators. There is an analogous of Theorem~\ref{thm:Ar} in the left pure braid generators \cite{Lam} related to this combing algorithm.
The set of relations for the left generators can be recovered from Theorem~\ref{thm:Ar} by applying on $P_n$ the automorphism $A_{ij} \to A_{n-j,\, n-i}^{-1}$. 
\end{remark}

\subsection{On conjugation} \label{onconj} 

The original presentation of $P_n$ focuses on how generators behave under conjugation by other generators.  What does conjugation look like topologically in the right pure braid generators? Recall Eq.~\ref{AG} and write the conjugation:
\begin{equation*}\begin{split}
    A_{ir} & A_{ij}A_{ir}^{-1} = \, (\sigma_i^{-1} \sigma_{i+1}^{-1} \cdots \sigma_{r-2}^{-1} \sigma_{r-1}^{2} \sigma_{r-2} \cdots \sigma_{i+1} \sigma_{i}) \\ &(\sigma_i^{-1} \sigma_{i+1}^{-1} \cdots \sigma_{r-1}^{-1} \sigma_r^{-1} \cdots \sigma_{j-2}^{-1} \sigma_{j-1}^2 \sigma_{j-2} \cdots \sigma_r \sigma_{r-1} \cdots \sigma_{i+1} \sigma_{i})(\sigma_i^{-1} \sigma_{i+1}^{-1} \cdots \sigma_{r-2}^{-1} \sigma_{r-1}^{-2} \sigma_{r-2} \cdots \sigma_{i+1} \sigma_{i}).
\end{split}
\end{equation*}
When $r<j$ there are two inverse sub-blocks: 
\[ (\sigma_{r-1} \sigma_{r-2} \cdots \sigma_{i+1} \sigma_{i})(\sigma_i^{-1} \sigma_{i+1}^{-1} \cdots \sigma_{r-1}^{-1}) \hspace{2mm}\text{and}  \hspace{2mm} (\sigma_{r-1} \sigma_{r-2} \cdots \sigma_{i+1} \sigma_{i})(\sigma_i^{-1} \sigma_{i+1}^{-1} \cdots \sigma_{r-1}^{-1}).\] Cancelling them, the expression becomes
\begin{equation} \label{conjA}
    A_{ir}A_{ij}A_{ir}^{-1} = \sigma_i^{-1} \sigma_{i+1}^{-1} \cdots \sigma_{r-2}^{-1} \sigma_{r-1} \sigma_r^{-1} \cdots \sigma_{j-2}^{-1} \sigma_{j-1}^{2} \sigma_{j-2} \cdots \sigma_r \sigma_{r-1}^{-1} \sigma_{r-2} \cdots \sigma_{i+1} \sigma_{i}
\end{equation}
 
Diagrammatically, the $i$-th strand passes in front of all the others but for the $r$-th, which it underpasses, then turns about the $j$-th strand, and goes back to the $i$th place overpassing all strands but the $r$-th again. See left-hand side of Figure~\ref{conj}. 
In the above expression, the exponents of the generator $\sigma_{r-1}$ (not of $\sigma_r$) are inverted with respect to the other generators. This could be counter-intuitive, since diagrammatically, the pure braid generator passes under the $r$-th strand, not the $(r-1)$-st one, but it is more obvious by observing that the crossing $\sigma_{i-1}$ involves the $i+1$-th and $i$-th strands, the former passing under the latter.

\begin{figure}[H]
\begin{center} 
    \includegraphics[width=12cm]{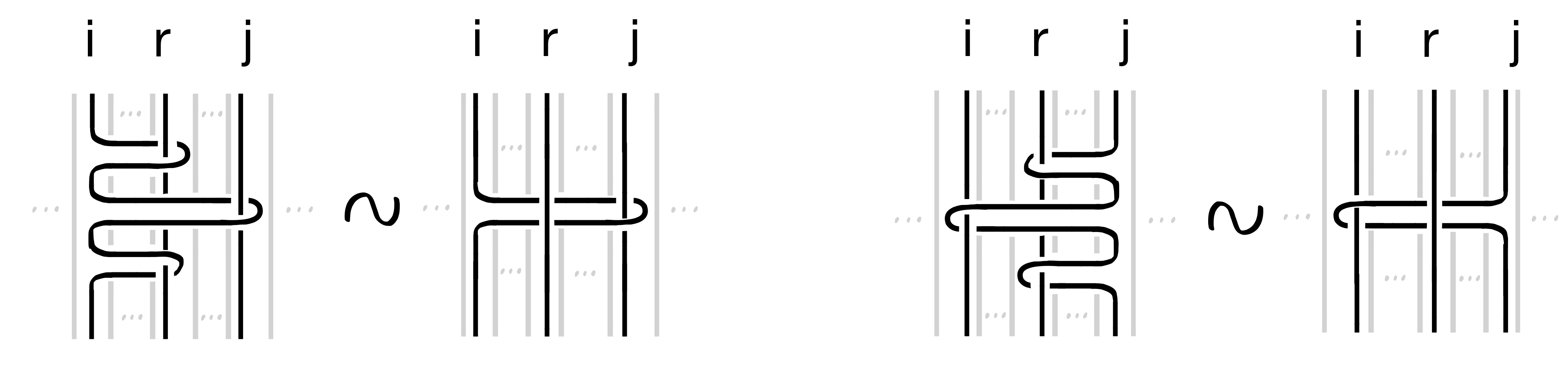}
\end{center}
\caption{ Conjugations and underpassings in the right generators and in the left ones respectively. }
\label{conj}
\end{figure}

Furthermore, if the $i$-th strand was to overpass the $r$-th on the going but to underpass it on the way back, it would be expressed as $A_{ij}A_{ir}^{-1}$, and similarly for the reverse passing, to underpass the $r$-th strand on the go and overpass it on the way back, $A_{ir} A_{ij}$. 

The conjugation $A_{ir}A_{ij}A_{ir}^{-1}$ for $ i < j < r$ does not give rise to a similar phenomenon, but it is part of Artin's relations as the right-hand side of relations 2).
In the left generators, the underpass under the strand $r$ for a generator $A_{ij}$ is expressed as the conjugation by $A_{rj}^{-1}$: $A_{rj}^{-1} A_{ij} A_{rj}$. 


\section{Alternative presentations for $P_n$ and $B_n$ with cyclic relations} \label{sectionalternative}

In this section we give an alternative presentation of the pure braid group $P_n$ in the right pure braid generators Eq.~\ref{AG}, with alternative relations that focus on how two pure braids commute, manipulating Artin's relations of Theorem~\ref{thm:Ar}. We note that this result appears also independently in \cite{Lee}. We, nevertheless, present our proof for the sake of completeness of our results.

\subsection{Manipulating Artin's relations}
Relation 1) in Theorem~\ref{thm:Ar} [ $A_{rs}A_{ij}A_{rs}^{-1} = A_{ij}$ ] is easily manipulated in the desired form $A_{ij} A_{rs} = A_{rs} A_{ij}$.

Relation 2) in Theorem~\ref{thm:Ar} $\big{[} A_{rs}A_{ir}A_{rs}^{-1}  =  A_{is}^{-1}A_{ir}A_{is}\big{]} $ is immediately equivalent to $A_{is}A_{rs}A_{ir} = A_{ir}A_{is} A_{rs}$. 

Relation 3) in Theorem~\ref{thm:Ar} $\big{[} A_{rs}A_{is}A_{rs}^{-1}  =   A_{is}^{-1}A_{ir}^{-1}A_{is}A_{ir}A_{is} \big{]} $ also reads $A_{ir}A_{is}A_{rs}A_{is} = A_{is}A_{ir}A_{is}A_{rs}$. Using now relation 2) on the left-hand side, relation 3)  becomes $A_{is}A_{rs}A_{ir}A_{is} = A_{is}A_{ir}A_{is}A_{rs}$, and by cancellation of the $A_{is}$ on the left of both words we obtain the combined relations: 
\begin{equation} \label{cyclicrels}
 A_{is}A_{rs}A_{ir}  = A_{rs}(A_{ir}A_{is}) = (A_{ir}A_{is})A_{rs}\hspace{5mm} \text{for} \hspace{5mm} i<r<s     
\end{equation}
so that these three cyclic permutations of generators are equivalent. Hence, by Tietze's theorem on finitely presented groups \cite{Tit}, relations 2) and 3) can be summed up in the above relations, which are illustrated in Figure~\ref{fig:cyclicrels} and proved topologically in Figure~\ref{sliding}:


\begin{figure}[H]
\begin{center} 
    \includegraphics[width=8cm]{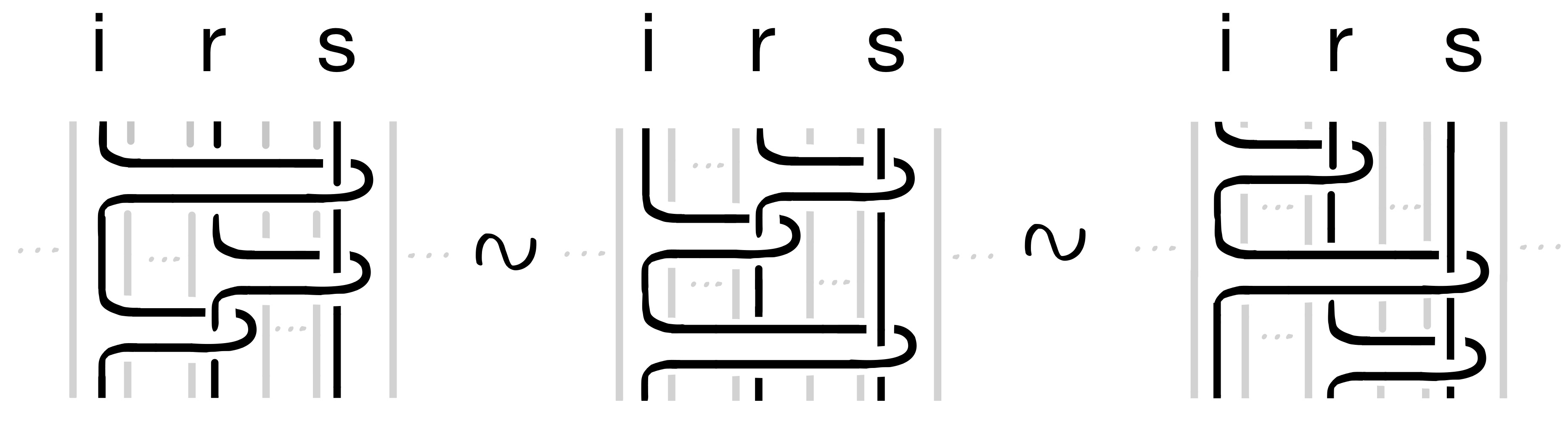}
\end{center}
\caption{The cyclic relations (\ref{cyclicrels}) for the pure braid group in the right generators.}
\label{fig:cyclicrels}
\end{figure}

\begin{figure}[H]
\begin{center} 
\includegraphics[width=17cm]{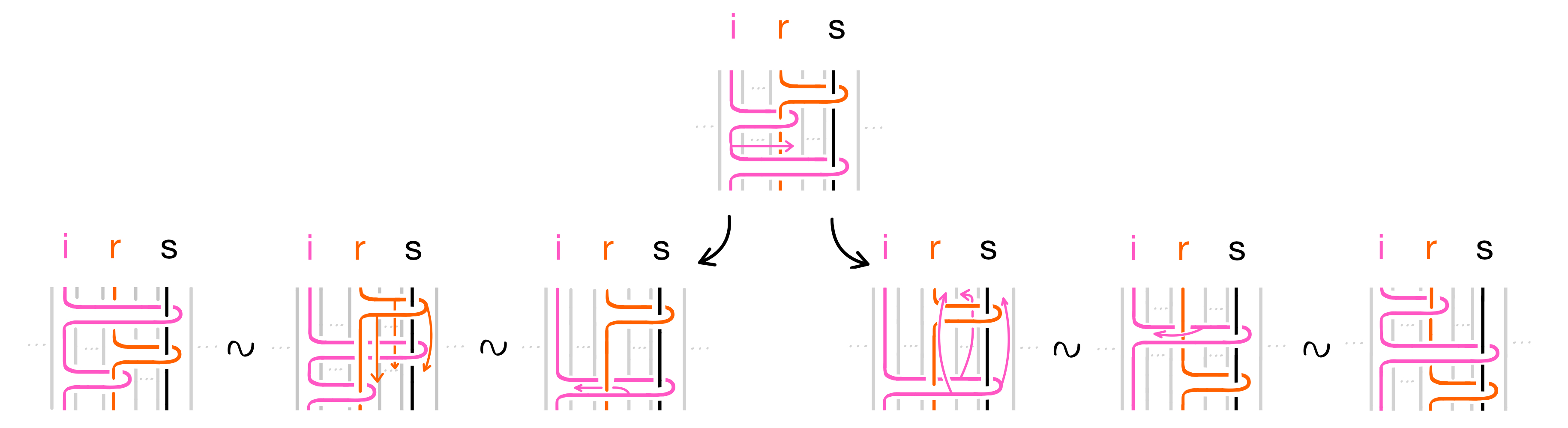}
\end{center}
\caption{Topological portrayal of the equalities in the cyclic relations}
\label{sliding}
\end{figure}

\begin{remark}
    From Figure~\ref{sliding}, the gist of the cyclic relations $(A_{ir} A_{is}) A_{rs} = A_{rs} (A_{ir} A_{is})$ can be seen to be captured in the commutativity of a standard pure generator with acomposition of two generators. The composition $(A_{ir} A_{is})$ winds around the $r$-th and $s$-th strands, and thus commutes with $A_{rs}$. 

Note that, even though the combined relations in Eq.~\ref{cyclicrels} focus on the commutation between generators, one can still extract some sort of combing algorithm. In fact, the combined relations show that, if there is a generator ($A_{rs}$) interfering with two other generators on the $i$-th strand ($A_{ir}$ and $A_{is}$, \, $i<r<s$), these latter two can be combed to the top of the pure braid or, equivalently, to the bottom.
\end{remark}

For what concerns relation 4) in Theorem~\ref{thm:Ar} $\big{[} A_{rs}A_{ij}A_{rs}^{-1}  =  A_{is}^{-1}A_{ir}^{-1}A_{is}A_{ir}A_{ij}A_{ir}^{-1}A_{is}^{-1}A_{ir}A_{is} \big{]}$, we claim that it is  equivalent to 
\begin{equation} \label{sliderelA}
(A_{rj} A_{rs} A_{rj}^{-1}) A_{ij}   = A_{ij} (A_{rj} A_{rs} A_{rj}^{-1}) \hspace{5mm} \text{for} \hspace{5mm} i<r<j<s     
\end{equation}
which regulates the commutation of a generator $A_{ij}$ overpassing/underpassing another one, $A_{rs}$, with their indices alternating, thus also referred to as a \textit{sliding} relation. Geometrically, the simplification is pictured in Figure~\ref{ArtAlt}.

\begin{figure}[H]
\begin{center} 
    \includegraphics[width=11cm]{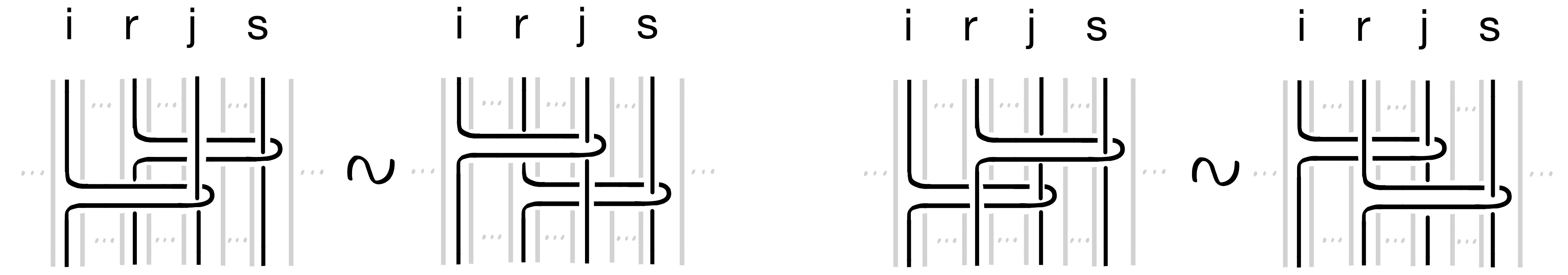}
\end{center}
\caption{The generator $A_{ij}$ slides over/under the generator $A_{rs}$ with alternating indices.}
\label{ArtAlt}
\end{figure}

\noindent Indeed, solving relation 4) with respect to $A_{ij}$ and following the steps illustrated in Figure~\ref{4Ralt1}, relation 4) can be rewritten as:
\begin{equation*}
    \begin{split}
        A_{ij} & = \underline{A_{rs}^{-1} A_{is}^{-1} A_{ir}^{-1} A_{is}}  A_{ir} A_{ij} A_{ir}^{-1} \underline{A_{is}^{-1} A_{ir} A_{is} A_{rs}}\\
        & \overset{\text{rel. 2)}}{=} A_{ir}^{-1} A_{rs}^{-1} \underline{A_{ir} A_{ij} A_{ir}^{-1}} A_{rs} A_{ir}\\
        & \overset{\text{rel. 3)}}{=} A_{ir}^{-1} \underline{A_{rs}^{-1} A_{rj}^{-1}} A_{ij} \underline{A_{rj} A_{rs}} A_{ir}\\
        & \overset{\text{rel. 3)}}{=} \underline{A_{ir}^{-1} A_{rj}^{-1}} A_{js}^{-1} A_{rs}^{-1} A_{js}  A_{ij} A_{js}^{-1} A_{rs} A_{js} \underline{A_{rj} A_{ir}}.  
    \end{split}
\end{equation*}

\begin{figure}[H]
\begin{center} 
    \includegraphics[width=11.5cm]{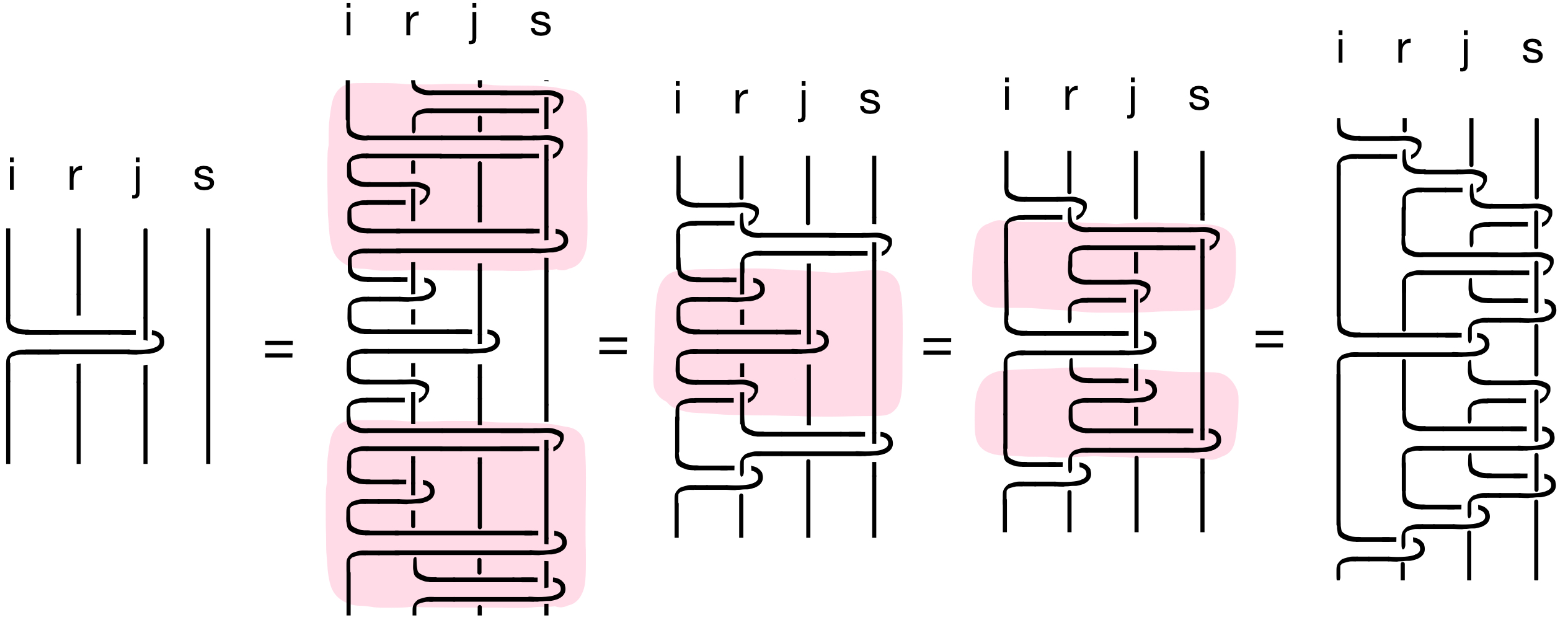}
\end{center}
\caption{Modifying Relation 4), part 1.}
\label{4Ralt1}
\end{figure}

Moving $ A_{ir}^{-1} A_{rj}^{-1}$ and $A_{rj} A_{ir}$  to the left-hand side, and proceeding as illustrated in Figure~\ref{4Ralt2}, we obtain the equivalent relation:
\begin{equation*}
    \begin{split}
       A_{rj} \underline{A_{ir} A_{ij} A_{ir}^{-1}} A_{rj}^{-1} & = A_{js}^{-1} A_{rs}^{-1} A_{js}  A_{ij} A_{js}^{-1} A_{rs} A_{js}\\
       \overset{\text{rel. 3)}}{\Longleftrightarrow}
       \underline{A_{rj} A_{rj}^{-1}} A_{ij} \underline{A_{rj} A_{rj}^{-1}} & = \underline{A_{js}^{-1} A_{rs}^{-1} A_{js}}  A_{ij} \underline{A_{js}^{-1} A_{rs} A_{js}}\\
       \overunderset{\text{cancel in l.h.s.}}{\text{apply rel. 2) on r.h.s.}}{\Longleftrightarrow} A_{ij} & = \underline{A_{rj} A_{rs}^{-1} A_{rj}^{-1}}  A_{ij} A_{rj} A_{rs} A_{rj}^{-1}\\ \overset{\text{moving to l.h.s.}}{\Longleftrightarrow} 
       A_{rj} A_{rs} A_{rj}^{-1} A_{ij} &  = A_{ij} A_{rj} A_{rs} A_{rj}^{-1}.
    \end{split}
\end{equation*}

\begin{figure}[H]
\begin{center} 
    \includegraphics[width=18cm]{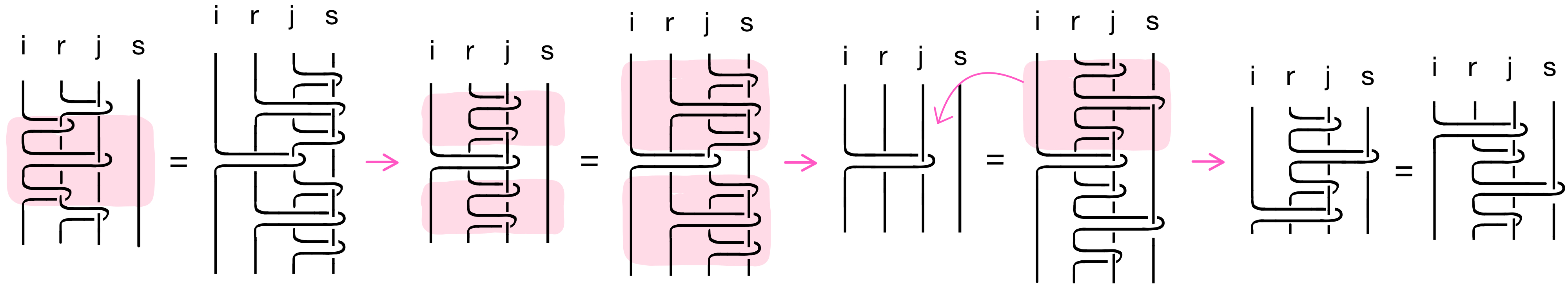}
\end{center}
\caption{Modifying Relation 4), part 2.}
\label{4Ralt2}
\end{figure}

Hence the alternative fourth relation. Therefore, we have proved the following:

\begin{proposition} \label{alternative} 
An alternative presentation with a set of cyclic relations for the pure braid group $P_n$ on $n$ strands on the right pure braid generators $A_{lm}, \, 1\leq l<m\leq n$, is given by
\[
\begin{array}{lrcll} 
         P_1) & A_{ij}A_{rs} & = & A_{rs}A_{ij} & \text{for} \hspace{2mm} i<r<s<j \hspace{2mm} \text{and} \hspace{2mm} i<j<r<s\\
         P_2) & A_{ir}(A_{is}A_{rs})  & = & (A_{is}A_{rs})A_{ir} & \text{for} \hspace{2mm} i<r<s\\
         P_3) & A_{rs}(A_{ir}A_{is}) & = & (A_{ir}A_{is})A_{rs} & \text{for} \hspace{2mm} i<r<s\\
         P_4) & (A_{rj} A_{rs} A_{rj}^{-1}) A_{ij} &  = & A_{ij} (A_{rj} A_{rs} A_{rj}^{-1}) & \text{for} \hspace{2mm} i<r<j<s
\end{array}
\]
\end{proposition}

The presentation for $P_n$ above, with the alternative second, third, and fourth relations, focuses on how neighbouring generators or blocks of generators commute. The second and third ones deal with three indices overall, whereas the first and fourth deal with four total indices: assuming $i<r$, the first considers the case where the two pairs of indices $(i, j)$ and $(r, s)$ do not separate each other (either $i<r<s<j$ or $i<j<r<s$), and the fourth the case of separating (alternating) indices, so $i<r<j<s$. 
In $B_n$, on the contrary, it suffices to inquire the two cases: that of adjacent indices (the second relation in Definition~\ref{braidgrp}), and those farther apart (commutativity, or first relation in Definition~\ref{braidgrp}). 

\begin{remark}
    Consider relation $P_4)$ of the alternative presentation in Proposition~\ref{alternative}. One can extract the equivalent relation:
    \begin{equation} \label{alternativeslideartin}
    \begin{split}
        \underline{A_{rj}} A_{rs} A_{rj}^{-1} A_{ij} = A_{ij} A_{rj} A_{rs} \underline{A_{rj}^{-1}} \stackrel{\text{ to other h.s. }}{\Longleftrightarrow} & A_{rs} \underline{A_{rj}^{-1} A_{ij} A_{rj}} = \underline{A_{rj}^{-1} A_{ij} A_{rj}} A_{rs}\\ \stackrel{Eq.~\ref{cyclicrels}}{\Longleftrightarrow} & A_{rs} (A_{ir} A_{ij} A_{ir}^{-1}) = (A_{ir} A_{ij} A_{ir}^{-1}) A_{rs}
    \end{split}
    \end{equation}
    which corresponds to the right generator $A_{rs}$ sliding above the right generator $A_{ij}$ for $i<r<j<s$ (refer to the right-hand side in Figure~\ref{ArtAlt}).
\end{remark}


\subsection{A redundant presentation for $\mathbf{B_n}$.}

Inspired by a similar result in \cite{Lam} on the left loop generators for the pure braid group $P_n$, the next proposition shows that there is an alternative, redundant yet clarifying, presentation for the braid group $B_n$, which focuses on the normal subgroup $P_n$ and makes use of Proposition~\ref{alternative}. This presentation will be surprisingly useful in linking the bonded braid monoid with the braid group. 

\begin{proposition} \label{redBn}
    The braid group $B_n$ admits the following redundant presentation in the standard braid generators, the right loop generators, and the braid and cyclic relations:
    $$ \langle \, \sigma_1, \dots , \sigma_{n-1}, \, (A_{ij})_{1 \leq i < j \leq n} \quad | \quad \Sigma_1, \, \Sigma_2, \, \Sigma_3, \, P_1, \, P_2, \, P_3, \ P_4, \, M_1, \, M_2, \, M_3, \, M_4, \, M_5 \, \rangle$$
    where the complete set of relations is:
    \[
    \begin{array}{rrcll}
        \Sigma_1 ) & \sigma_i \sigma_j & = & \sigma_j \sigma_i & |i-j| > 1 \\
         \Sigma_2 ) & \sigma_i \sigma_{i+1} \sigma_i & = & \sigma_{i+1} \sigma_i \sigma_{i+1} & 1 \leq i \leq n-2\\
         \Sigma_3 ) & \sigma_i^2 & = & A_{i, i+1} & 1 \leq i \leq n-1\\
         P_1) & A_{ij} \, A_{rs} & = & A_{rs} A_{ij} & i<j<r<s, \ \ i<r<s<j\\
         P_2 ) & A_{ir} \, A_{is} \, A_{rs} & = & A_{is} \, A_{rs} \, A_{ir} & i < r < s  \\
         P_3 ) & A_{rs} \, A_{ir} \, A_{is} & = & A_{ir} \, A_{is} \, A_{rs} & i < r < s\\
         P_4 ) & (A_{rj} \, A_{rs} \, A_{rj}^{-1}) \, A_{ij} & = & A_{ij} \, (A_{rj} \, A_{rs} \, A_{rj}^{-1}) & i < r <j < s \\
         M_1 ) & \sigma_{i-1}^{-1} \, A_{ij} \, \sigma_{i-1} & = & A_{i-1, j} & 1 < i < j \leq n \\
         M_2) & \sigma_i^{-1} \, A_{ij} \, \sigma_i & = & A_{i, i+1}^{-1} \, A_{i+1, j} \, A_{i, i+1} & 1 < i+1 < j \leq n\\
         M_3) & \sigma_{j-1}^{-1} \, A_{ij} \, \sigma_{j-1} & = & A_{i, j-1} & 1 \leq i < j-1 \leq n\\
         M_4) & \sigma_j^{-1} A_{ij} \sigma_j & = & A_{i j} \, A_{i, j+1} \, A_{ij}^{-1} & 1 \leq i < j \leq n-1 \\
         M_5) & \sigma_i^{-1} A_{rs} \sigma_i & = & A_{rs} & i < r-1, \ \ r < i < s-1, \ \ s < i
         \end{array}
    \]
\end{proposition}
\begin{proof} The right pure braid group generators $A_{ij}$  are now seen as elements of $B_n$, in which $P_n$ embeds. 
The relations can be computed from a classical result on short exact sequences, for which we refer back to \cite{Jo}: let 
$$ 1 \rightarrow X \xrightarrow{\iota} Y \xrightarrow{\theta} Z \rightarrow 1
$$
be a short exact sequence. A presentation for $Y$ can be computed from presentations for $X$ and $Z$ in the following way. If $\langle A | R \rangle$ and $\langle B | S \rangle$ are presentations for $X$ and $Z$ respectively, a presentation for $Y$ has as generators a choice of preimages $\theta^{-1}(B)$ and the image set $\iota(A)$, and as relations the union of three sets: the set of relations in $S$ rewritten in the preimage generators $\theta^{-1}(B)$, the set of relations $R$ rewritten in the image generators $\iota(A)$, and the set of all conjugates $\tilde{b}^{-1} \, \tilde{a} \, \tilde{b}$ for any choice $\tilde{a} \in \iota(A)$, $\tilde{b} \in \theta^{-1}(B)$.\\
With reference to the short exact sequence in Eq.~\ref{ses}, $X$ corresponds to $P_n$, $Y$ corresponds to $B_n$, and $Z$ to $S_n$, so the statement follows with inputs the presentations of $P_n$ in Proposition~\ref{alternative} and the classical presentation of $S_n$ in Eq.~\ref{Snpres}. In particular, the relations $\Sigma_i)$ are the preimages of the defining relations of $S_n$, the relations $P_j)$ are the images of the pure braid relations, and the mixed relations $M_k)$  are the conjugate ones.

Note that the original set of relations of type $M)$ obtained following this algorithm, would also include $M_6$):  $\sigma_i^{-1} \, A_{i, i+1} \sigma_i = A_{i, i+1}$,
but as it can be traced back to $\Sigma_3)$, it can be omitted.
\end{proof}
 

\section{The Bonded Braid Monoid}

The main goal of this work is to relate the  braid group on $n$ strands $B_n$ seen as a monoid with the \textit{bonded braid monoid on $n$ strands} $BB_n$. This latter is generated by the elementary braid generators and bonds, horizontal embedded arcs connecting distinct strands. 
The theory of bonded braids is introduced in \cite{DKL} as the algebraic counterpart of the topological theory of bonded knots and links. In Section~\ref{Representations} it will be shown that there is a epirphism of monoids from $BB_n$ to $B_n$ that sends bonds to pure braids. The present section is preparatory to this end. In this section, we present the bonded braid monoid and we derive a reduced presentation with only overpassing bonds (Proposition~\ref{oversuff} and Theorem~\ref{oBB}) making use of Artin's pure braid generators.

\subsection{Topology of bonded knots and bonded links} \label{topologyBB} 

Bonded knots  were introduced in 2017 in the wider context of bonded knotoids \cite{GGLDSK}, in order to answer the question of a suitable model for protein chains. 

A \textit{knotoid diagram} is an immersion in a surface of the unit interval (so an open-ended curve), with only a finite amount of double points, each endowed with an overcrossing/undercrossing datum [Tu]. Spherical knotoids with their endpoints in the same diagrammatic region correspond bijectively to (isotopy classes of) classical knots, cf. [Tu]. 
 A \textit{bonded link} (or \textit{bonded knot}, in case of a single component) is a pair $(L, B)$, where $L$ is a (oriented) link in $S^3$, and $B$  is a set of $k$ disjoint simple arcs, called \textit{bonds}, properly embedded in the complement of $L$ in $S^3$, such  that they present no knotting or linking among themselves, and such that their ends, the \textit{nodes}, intersect $L$ transversally in $2k$ distinct points. The neighbourhood of a node shall be referred to as a \textit{bonding site}. If $B = \emptyset$, $(L, \emptyset)$ is a classical link. From the above, a bonded link diagram contains no crossings or self-crossings of bonded arcs.
  Bonded links undergo either rigid vertex isotopies or topological vertex isotopies \cite{DKL} (see Figure~\ref{TRVI}). Each theory has its proper diagrammatic calculus, as an equivalence generated by the classical Reidemeister moves for the link arcs and local moves reflecting interactions of bonds and arcs of $L$, see Figures~\ref{planisoandvslide} to ~\ref{TRVI}. For an in-depth analysis on bonded link isotopy and allowed/forbidden diagrammatical moves, the reader is invited to consult \cite{DKL}.
 We note that what we call as bonded knot/link, in \cite{DKL} is referred to as `standard bonded knot/link', whereas bonded knots/links are considered there in a wider context. 

\begin{figure}[h]
\begin{center} 
\includegraphics[width=10cm]{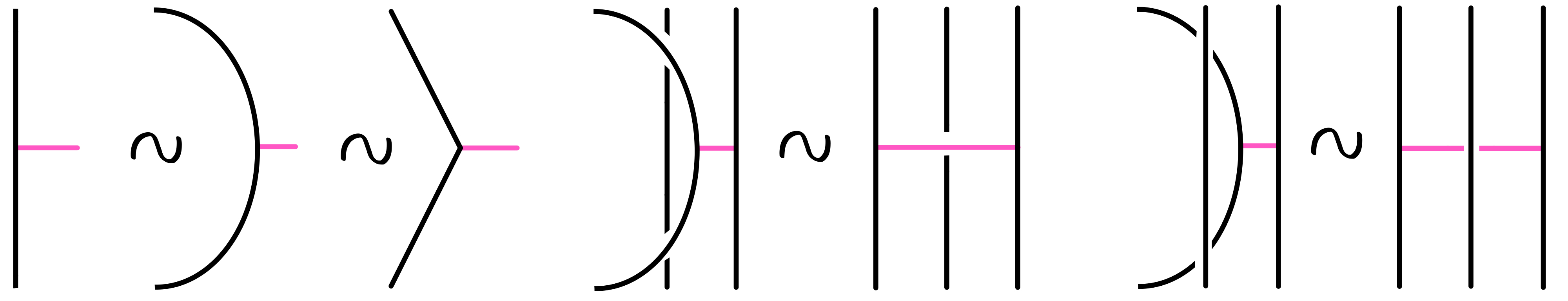}
\end{center}
\caption{Planar isotopies of bonding sites and vertex slide moves.}
\label{planisoandvslide}
\end{figure}

\begin{figure}[h]
\begin{center} 
\includegraphics[width=6.1cm]{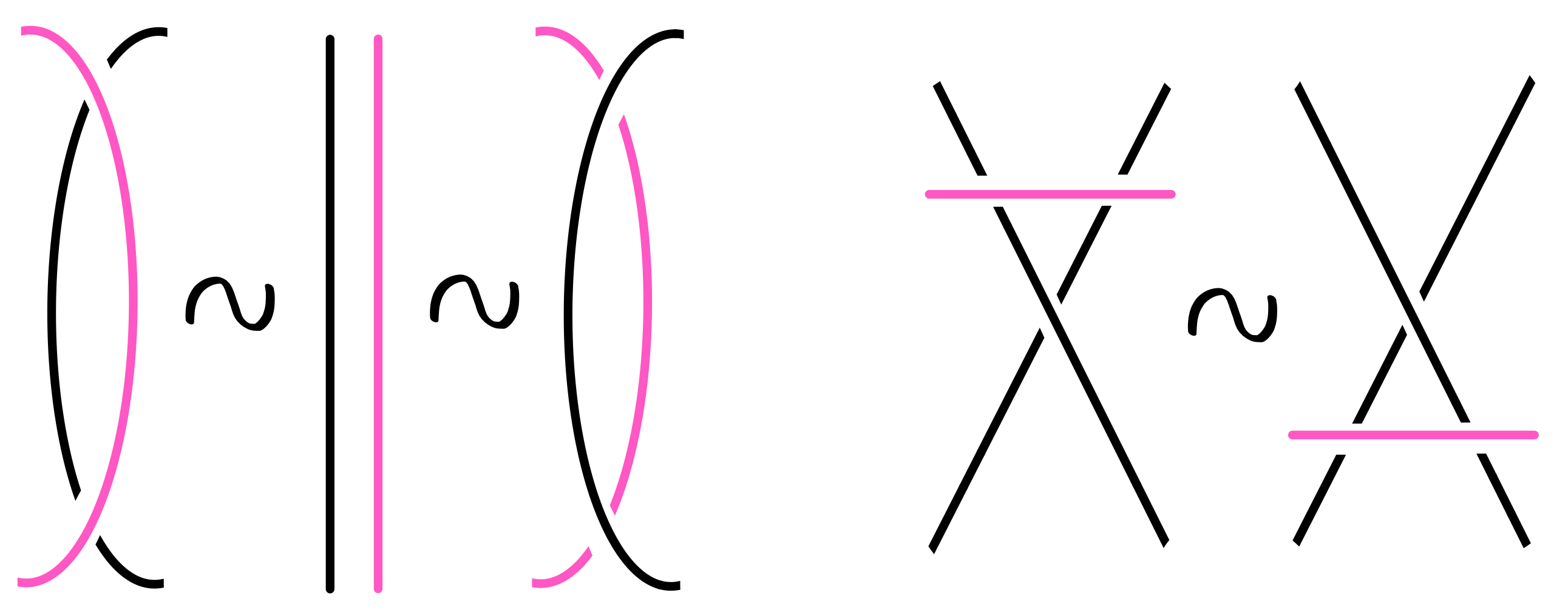}
\end{center}
\caption{Reidemeister type moves on mixed configurations.}
\label{mixedreid}
\end{figure}

\begin{figure}[h]
\begin{center} 
\includegraphics[width=11cm]{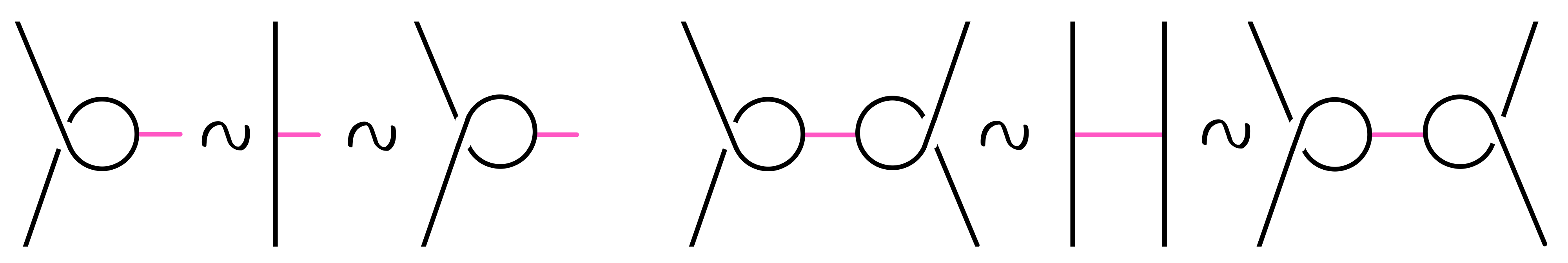}
\end{center}
\caption{Topological vertex twists and rigid vertex twists.}
\label{TRVI}
\end{figure}


\subsection{The defining presentation for the bonded braid monoid}

In \cite{DKL} the theory of bonded braids is introduced as the algebraic counterpart theory of bonded knots/links. 
In classical knot theory, there are two theorems relating knots and braids - the Alexander and Markov theorems - which respectively show that any link is isotopic to the \textit{standard closure} of a geometrical braid, and that there is a list of algebraic operations in the braid groups corresponding to knot isotopy. 
Analogous results hold for bonded links and bonded braids. The interested reader is referred to \cite{DKL}.
As in the classical case, bonded braids have a dual algebraic and geometric nature, as a monoid and as isotopy classes of braided diagrams respectively. An example of a bonded braid and its standard closure to an oriented bonded knot is illustrated in Figure~\ref{closure}.

\begin{definition}[\cite{DKL}, Definition 9] \label{definitionbondedbraid} A {\it bonded braid diagram on $n$ strands} (also referred to as \textit{bonded braid}) is a pair $(\beta, B)$ of a classical braid diagram $\beta$ on $n$ strands and a set $B$ of $k$ disjoint, horizontal immersed simple arcs, the {\it bonds}, with over/under information at every crossing with a strand of $\beta$, and such that no crossing is horizontally aligned with a bond. The boundary points of a bond, called {\it nodes},  intersect $\beta$ transversally in $2$ distinct points, which cannot coincide with endpoints of the braid, and they have local neighbourhoods that are 3-valent graphs with the attaching arcs, like $\vdash$, forming together an \texttt{H}-neighbourhood. If $B=\emptyset$, then $(\beta,\emptyset)$ is just a classical braid.

A bond joining the $i$-th and $j$-th strands with $i<j$  threads transversely around the strands of $\beta$ in between. A bond denoted $b_{ij}$ is assumed to pass \textit{over} any intermediate strand between the $i$-th and the $j$-th. If the bond passes \textit{under} the strands $l_1, \dots, l_k$, with $i < l_1 < \dots < l_k < j$, then it is denoted as $b_{ij}^{(l_1, \dots, l_k)}$ or simply $b_{ij}^{\mathcal{I}}$, where $\mathcal{I} = (l_1, \dots , l_k)$ an ordered subset  of $\left\{i+1, \, \dots  , j-1\right\}$. For $\mathcal{I} = \emptyset$ we get the bond $b_{ij}$.  
If the bond is a uniform underpass, it can also be denoted as $b_{ij}^u$ for brevity. A bond $b_{i, i+1}$ joining two consecutive strands $i$ and $i+1$ shall be called {\it elementary bond} and will be denoted as $b_i$. 
\end{definition}

\begin{figure}[H]
\begin{center} 
    \includegraphics[width=6cm]{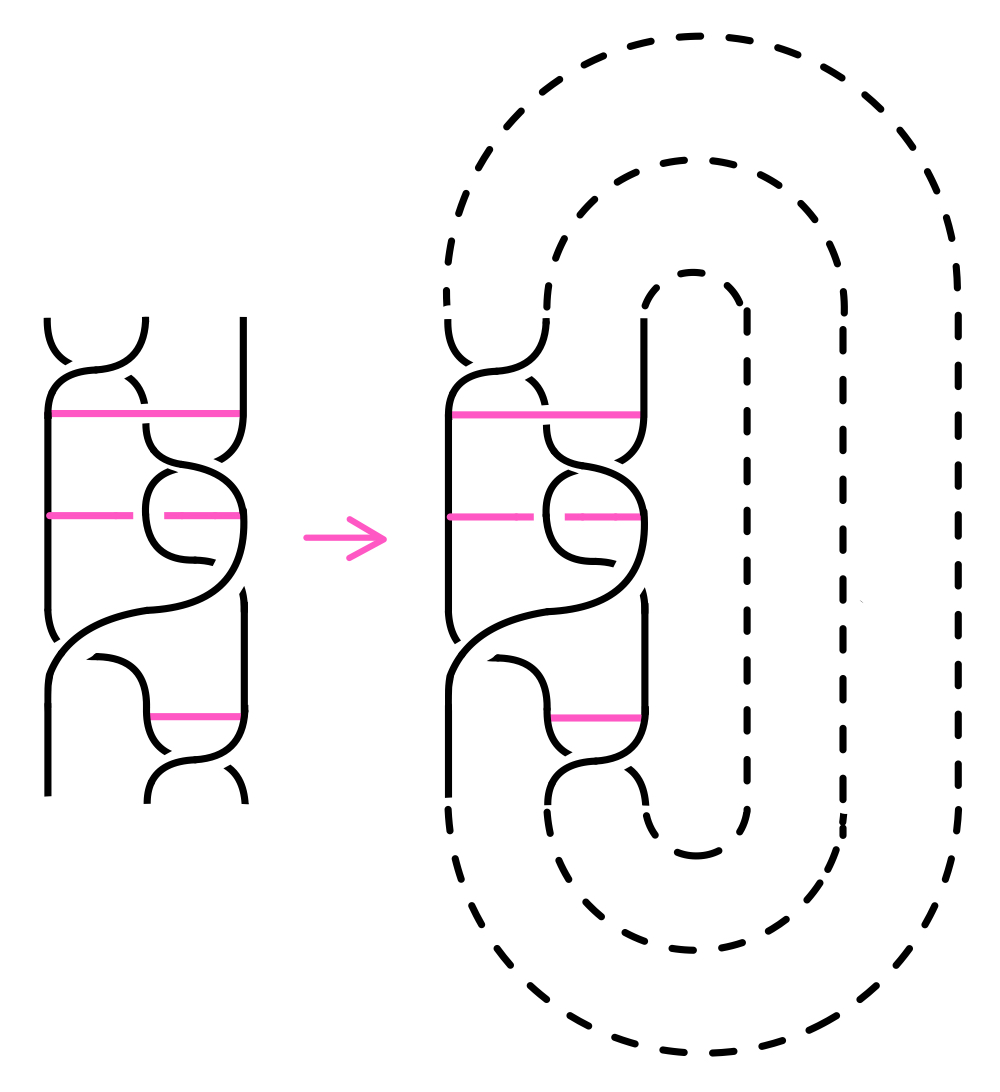}
\end{center}
\caption{An example of a bonded braid and its closure in an oriented bonded link.}
\label{closure}
\end{figure}

The set of bonded braids on $n$ strands has a monoid structure \cite{DKL}, not a group structure, since bonds do not have topological inverses. Under the classical composition, a generating set consists in the bonds $b_{ij}^{(l_1, \dots, l_k)}$ for $1 \leq i <j \leq n$ for any underpass sequence $\mathcal{I} = (l_1, \dots , l_k)$ of intermediate strands, and the elementary braids $\sigma_i$ with their inverses, for $1 \leq i \leq n-1$. The monoid is finitely presented and the relations correspond to the algebraic translation of all possible elementary isotopy moves between bonded braids.

\definition \label{BB} (\cite{DKL}) The \textit{bonded braid monoid} $BB_n$ on $n$ strands is generated by:
\[ 
\begin{array}{lcl}
    \text{- the classical braid generators and their inverses } & \sigma_i^{\pm 1} & \text{for } 0<i<n \vspace{1.5mm} \\  
     \text{- the uniform overpassing bonds } & b_{ij} & \text{for } 0<i<j<n \vspace{1.5mm} \\  
     \text{- the mixed over/underpassing bonds } & b_{ij}^{\mathcal{I}}=b_{ij}^{(l_1, \dots, l_k)} & \text{for } 0<i<l_1< \dots < l_k < j <n, \,  \vspace{1.5mm} \\  
     & &\mathcal{I} = (l_1, \dots, l_k)
\end{array}
\]
subject to the relations:
\[
\begin{array}{rrcll}
        1) & \hspace{4mm} \sigma_i \sigma_j \, & = & \, \sigma_j \sigma_i \hspace{2mm} & \text{for}\hspace{2mm} |i-j|>1 
         \vspace{1.5mm} \\  
        2) & \hspace{4mm}  \sigma_i \sigma_{i+1} \sigma_i \, & = & \, \sigma_{i+1} \sigma_i \sigma_{i+1} \hspace{2mm} & \text{for all } i \vspace{1.5mm} \\
        3) & \hspace{4mm} b_{i j}^{\mathcal{I}} \, b_{r s}^{\mathcal{R}} \, & = & \,  b_{rs}^{\mathcal{R}} \, b_{ij}^{\mathcal{I}} \hspace{2mm} & \text{for all} \hspace{2mm} i<j<r<s, \hspace{2mm} \mathcal{I} \subseteq \left\{ i+1, \dots , j-1\right\},\\
        & & & & \mathcal{R} \subseteq \left\{ r+1, \dots , s-1\right\}  \vspace{1.5mm} \\
        4) & \hspace{4mm} b_{i j}^{\mathcal{I}} \, b_{r s}^{\mathcal{R}} \, & = & \,  b_{rs}^{\mathcal{R}} \, b_{ij}^{\mathcal{I}} \hspace{2mm} & \text{for all} \hspace{2mm} i<r<s<j, \hspace{2mm} \mathcal{I} \subseteq \left\{ i+1, \dots , j-1\right\} \vspace{1.5mm}\\
        & & & & \mathcal{R} \subseteq \left\{ r+1, \dots , s-1\right\}, \text{ such that:}\vspace{1.5mm} \\
        &&&& \text{if } r, s  \in \mathcal{I} \implies \cR \subseteq \cI \cap \{ r+1, \dots , s-1\} \vspace{1.5mm} \\
        &&&& \text{if }  r, s  \notin \mathcal{I}  \vspace{1.5mm} \implies \cI \cap \{ r+1, \dots , s-1\} \subseteq \cR \vspace{1.5mm} \\
        5.1) & \hspace{4mm} \sigma_i^{-1} b_{i+1, j}^{\mathcal{I}} \, \sigma_i & = & b_{ij}^{\mathcal{I}}  \hspace{2mm} & \text{for } 1 \leq i<j \leq n-1, \, |i-j|>1, \vspace{1.5mm} \\
        & & & & \mathcal{I} \subseteq \left\{ i+2, \dots , j-1 \right\} \textit{ (vertex slide move)}\vspace{1.5mm} \\
        5.2) & \hspace{4mm} \sigma_i b_{i+1, j}^{\mathcal{I}} \, \sigma_i^{-1} & = & b_{ij}^{(i+1) \cup \mathcal{I}}  \hspace{2mm} & \text{for } 1 \leq i<j \leq n-1, \, |i-j|>1, \vspace{1.5mm} \\
        & & & & \mathcal{I} \subseteq \left\{ i+2, \dots , j-1 \right\} \textit{ (vertex slide move)} \vspace{1.5mm} \\
        5.3) & \hspace{4mm} \sigma_{j-1} b_{i, j-1}^{\mathcal{I}} \, \sigma_{j-1}^{-1} & = & b_{ij}^{\mathcal{I}}  \hspace{2mm} & \text{for } 1 \leq i<j \leq n-1, \, |i-j|>1, \vspace{1.5mm} \\
        & & & & \mathcal{I} \subseteq \left\{ i+1, \dots , j-2 \right\} \textit{ (vertex slide move)} \vspace{1.5mm} \\
        5.4) & \hspace{4mm} \sigma_{j-1}^{-1} b_{i, j-1}^{\mathcal{I}} \, \sigma_{j-1} & = & b_{ij}^{\mathcal{I} \cup (j-1)}  \hspace{2mm} & \text{for } 1 \leq i<j \leq n-1, \, |i-j|>1, \vspace{1.5mm} \\
        & & & & \mathcal{I} \subseteq \left\{ i+1, \dots , j-2 \right\} \textit{ (vertex slide move)} \vspace{1.5mm} \\
        6) & \hspace{4mm} \sigma_i b_i & = & b_i \sigma_i \hspace{2mm} & \text{for } 1 \leq i \leq n-1 \hspace{2mm} \textit{(bonded flype moves)} \vspace{1.5mm} \\
        7) & \hspace{4mm} \sigma_i \, b_{rs}^{\mathcal{R}} \, & = & \,  b_{rs}^{\mathcal{R}} \, \sigma_i \hspace{2mm} & \text{for} \hspace{2mm} i+1<r \hspace{2mm} \text{or} \hspace{2mm} i > s, \, 1 \leq i \leq n-1, \vspace{1.5mm} \\ & & & & \mathcal{R} \subseteq \left\{ r+1, \dots , s-1\right\} \vspace{1.5mm} \\
        8.1) & \hspace{4mm} \sigma_i \, b_{rs}^{\mathcal{R}} \, & = & \,  b_{rs}^{\mathcal{R}} \, \sigma_i & \text{for} \hspace{2mm} 1 \leq r<i<s-1 \leq n-1 \vspace{1.5mm} \\ &&&& \text{and either } i, i+1  \in \mathcal{R} \text{ or }  i, i+1  \notin \mathcal{R},   \vspace{1.5mm} \\
        8.2) & \hspace{2mm} \sigma_i \, b_{rs}^{\mathcal{R}} & = & b^{\mathcal{R'}}_{rs} \sigma_i & \text{for } 1 \leq r < i < s-1 \leq n-1, \  \mathcal{R},\mathcal{R^\prime} \subseteq \{ r+1, \dots , s-1 \} \vspace{1.5mm} \\ &&&& \text{equal except: }  i \in \cR , \, i +1 \notin \cR \   \& \ i \notin \cR', i+1 \in \cR' \vspace{1.5mm} \\
        9.1) & \hspace{2mm} b_i \sigma_{i+1} \sigma_i \, & = & \, \sigma_{i+1}\sigma_i b_{i+1} \hspace{2mm} & \text{for } 1 \leq i < n-1 \vspace{1.5mm} \\
        9.2) & \hspace{2mm} \sigma_i \sigma_{i+1} b_i \, & = & \, b_{i+1} \sigma_i \sigma_{i+1} \hspace{2mm} & \text{for } 1 \leq i < n-1.
   \end{array}
 \]

\bigbreak
\noindent All of the above relations come from a topological analysis of all possible local isotopy interactions between bonds, braid arcs and crossings, preserving the bonded braid structure. This presentation is the \textit{bonded braid monoid presentation}, which is by no means a reduced one. In Figures~\ref{bondedrels1} and ~\ref{bondedrels2} the above relations are shown with underpassing indices sets $\mathcal{I}$, $\mathcal{R}$, which for graphical easiness are empty when possible. 

\begin{figure}[H]
\begin{center} 
    \includegraphics[width=16cm]{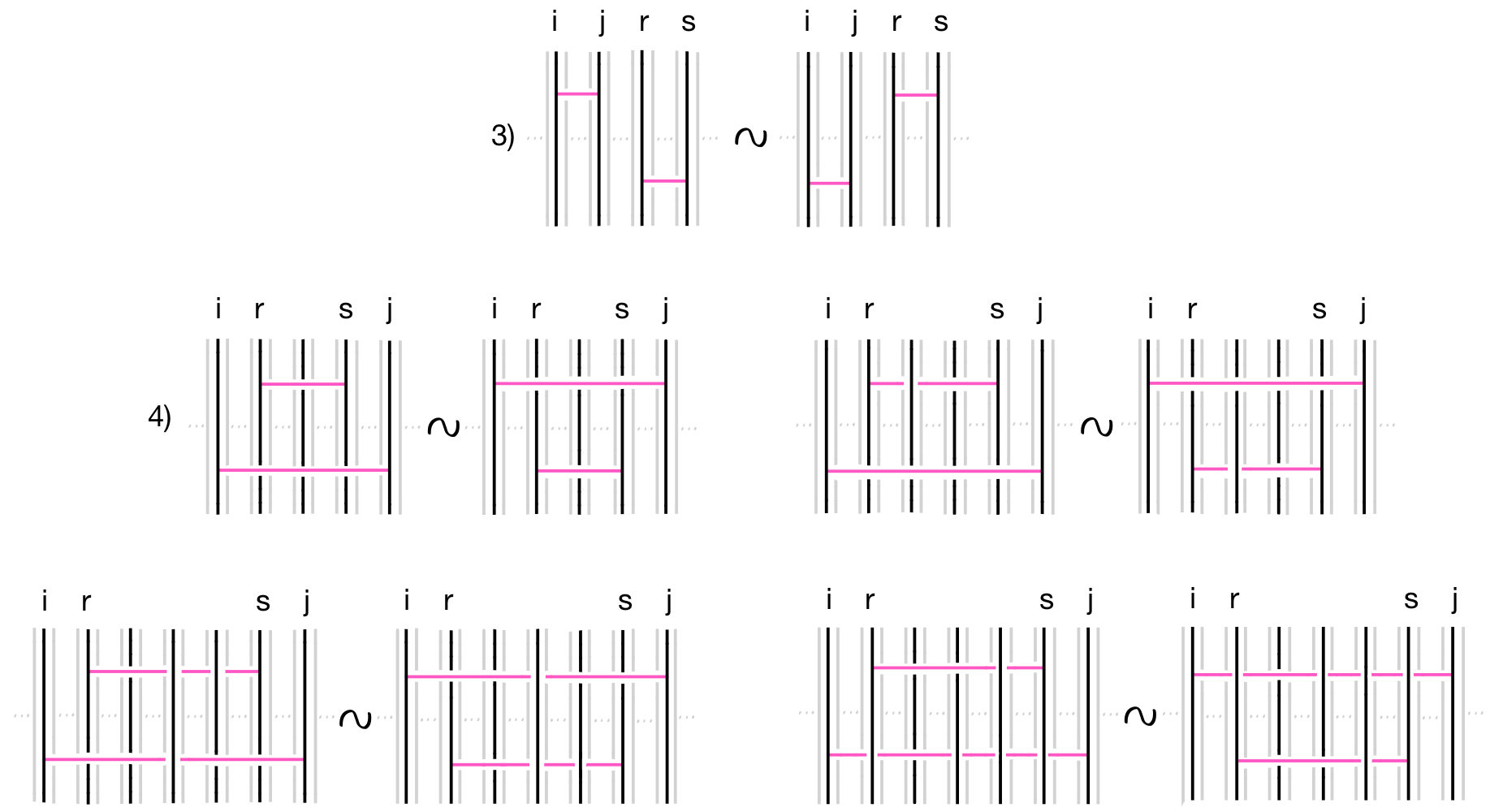}
\end{center}
\caption{Bond-bond interactions (relations 3) and 4)).}
\label{bondedrels1}
\end{figure}

\begin{figure}[H]
\centering
    \includegraphics[width=17cm]{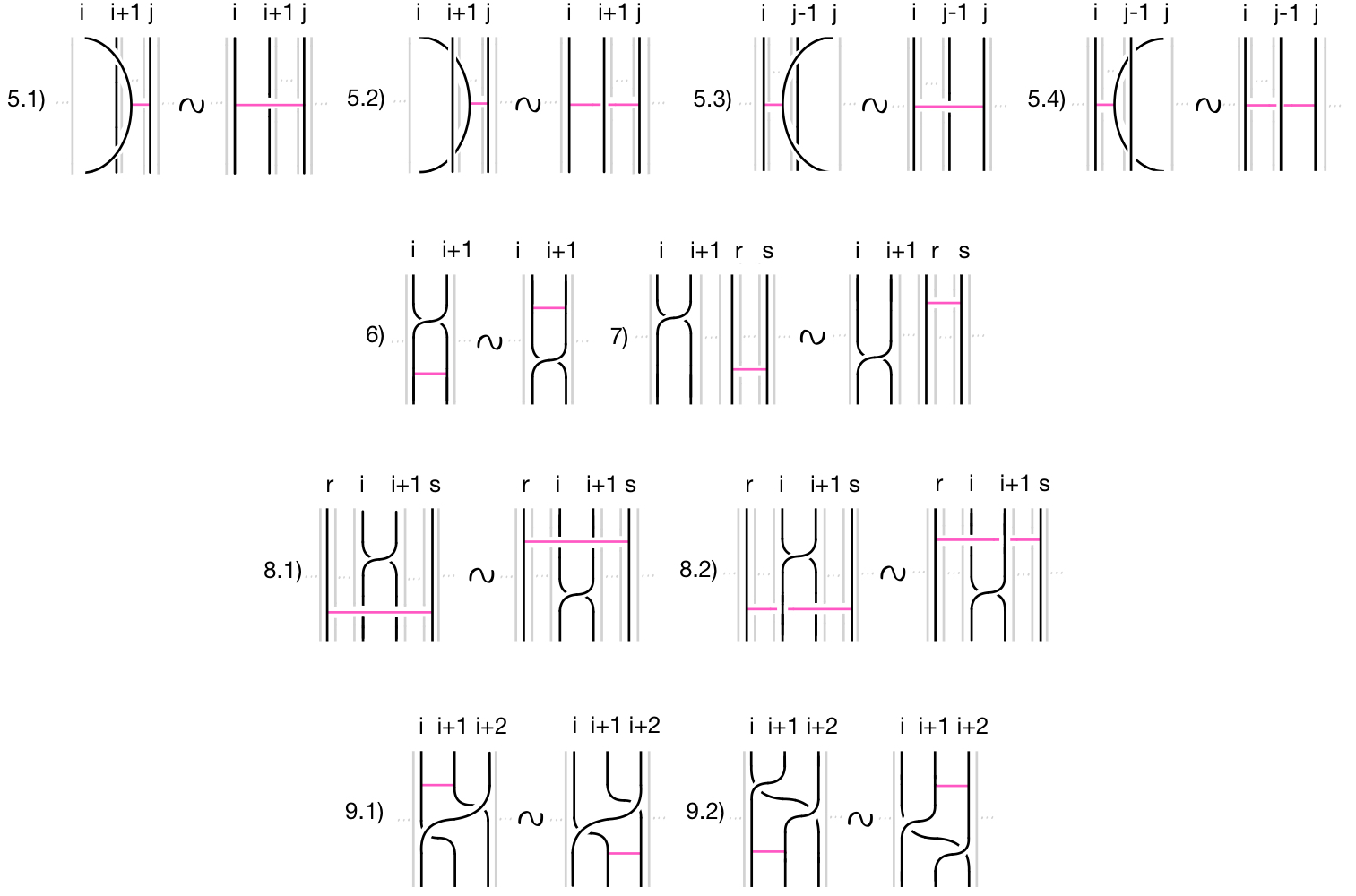}
\caption{Bond-braid interactions (relations 5) to 9.2)).}
\label{bondedrels2}
\end{figure}

\begin{remark}
The topological implication of the conditions for relations 4) in Definition~\ref{BB} is that $b_{ij}^{\cI}$ should either underpass or overpass both the strands $r, \, s$, otherwise  an obstruction  would unavoidably present itself during the isotopy for realising the topological commuting of $b_{ij}^{\cI}$ and $b_{rs}^{\cR}$. This means that either both $r,s \in \cI$ or both $r,s \notin \cI$. The condition on $r,s$ forces then the conditions on the common intermediate braid strands: 

If $r,s \in \cI$ then $b_{ij}^{\cI}$ has to pass under $b_{rs}^{\cR}$. This implies that if $t \in \cR$ then it should also be $t \in \cI$ (otherwise a topological obstruction will occur). While if $t \in \cI$ then it could be $t \in \cR$ or $t \notin \cR$ (because $b_{ij}^{\cI}$ can slide on the bottom anyway: $b_{ij}^{\cI}$ can be either the intermediate or the bottom arc of the local Reidemeister move). Equivalently, $ \cR \subseteq \cI \cap \{ r+1, \dots , s-1\}$.  

If $r,s \notin \cI$ then $b_{ij}$ has to pass over $b_{rs}$. This implies that if $t \in \cI$ then it should also be $t \in \cR$ (otherwise a topological obstruction will occur). While if $t \in \cR$ then it could be $t \in \cI$ or $t \notin \cI$ (because $b_{ij}$ can slide on the top anyway: $b_{ij}$ can be either the intermediate or the top arc of the local Reidemeister move). Equivalently, $\cI \cap \{r+1, \dots , s-1 \} \subseteq \cR$    
\end{remark}

\begin{remark}
    Relations 1) and 2) in Definition~\ref{BB} regulate interactions among crossings.
    Relations 3) and 4) regulate the interactions among bonds. In relations 3) the indices are sequential, while in relations 4) the indices $r,s$ lie within the interval $(i,j)$.  Displayed in Figure~\ref{bondslideFIG}, there is another relation between bonds, with alternating indices: 
    \begin{equation} \label{bondslide}
        b_{ij}\, b_{rs}^{(j)} = b_{rs}^{(j)}\, b_{ij}, \quad \text{for } i<r<j<s.
    \end{equation}

\begin{figure}
\begin{center} 
    \includegraphics[width=4.5cm]{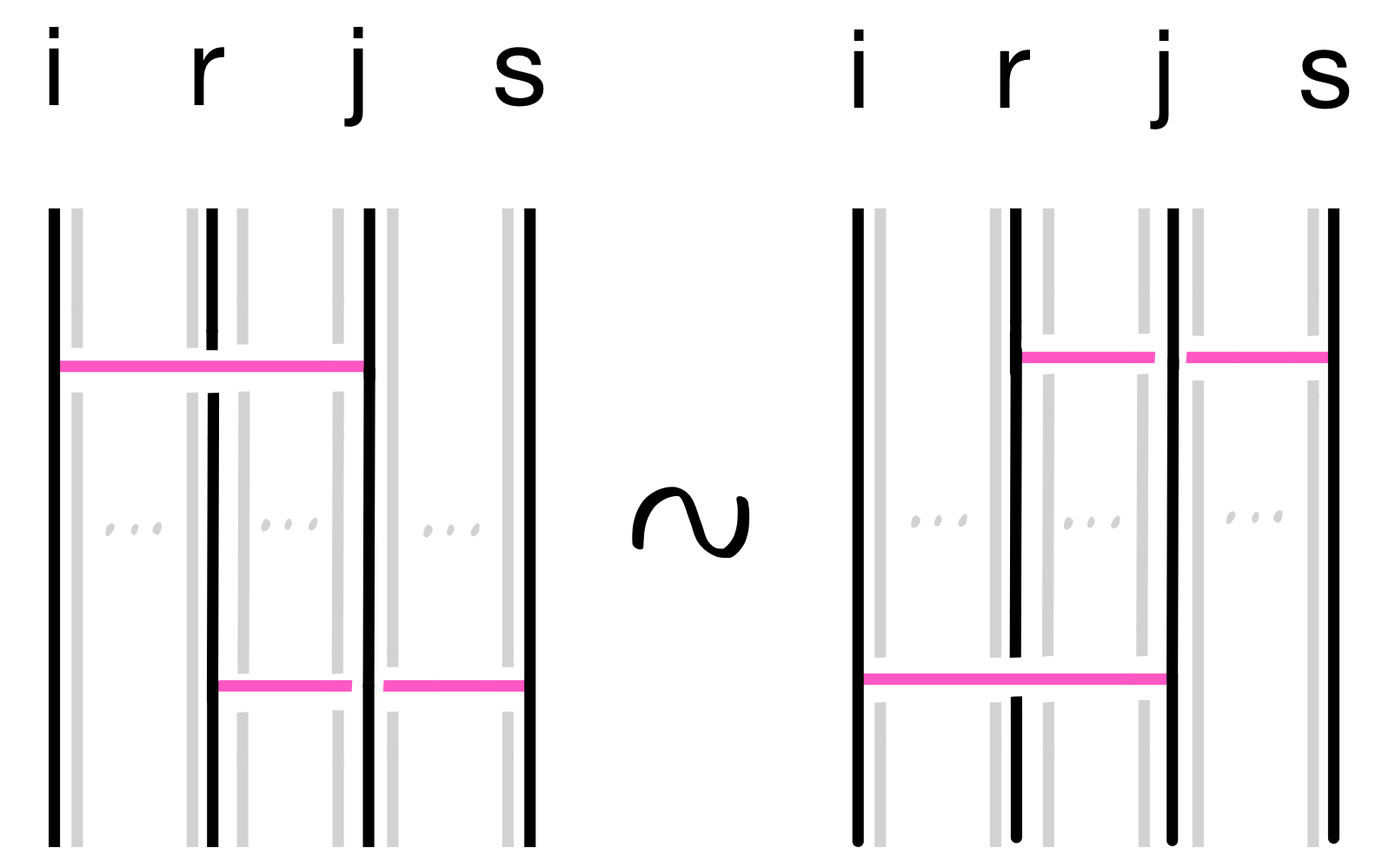}
\end{center}
\caption{A bond slide relation with alternating indices.}
\label{bondslideFIG}
\end{figure}

\noindent This relation follows from other defining relations, as also illustrated in Figure~\ref{bondslide2}. Indeed,
    \begin{equation*}
        \begin{split}
        b_{ij}  b_{rs}^{(j)} & \overset{\text{5.3)}}{=} \sigma_{j-1} \cdots \sigma_r b_{ir} \sigma_r^{-1} \cdots \underline{\sigma_{j-1}^{-1} b_{rs}^{(j)}} \\ &
        \overset{\text{8.2)}}{=} \sigma_{j-1} \cdots \sigma_r b_{ir} \sigma_r^{-1} \cdots \underline{\sigma_{j-2}^{-1} b_{rs}^{(j-1)}} \sigma_{j-1}^{-1} \\
        & \overset{\text{reiterating}}{=} \sigma_{j-1} \cdots \sigma_r b_{ir} \underline{\sigma_r^{-1} b^{(r+1)}_{rs}} \sigma_{r+1}^{-1} \cdots \sigma_{j-1}^{-1} \\ & \overset{\text{5.2)}}{=} \sigma_{j-1} \cdots \sigma_r \underline{b_{ir}  b_{r+1, s}} \sigma_r^{-1} \sigma_{r+1}^{-1} \cdots \sigma_{j-1}^{-1}\\
        & \overset{\text{3)}}{=} \sigma_{j-1} \cdots \underline{\sigma_r  b_{r+1, s}} b_{ir} \sigma_r^{-1} \sigma_{r+1}^{-1} \cdots \sigma_{j-1}^{-1}\\
        & \overset{\text{5.2)}}{=} \sigma_{j-1} \cdots \underline{\sigma_{r+1} b_{r s}^{(r+1)}} \sigma_r b_{ir} \sigma_r^{-1} \sigma_{r+1}^{-1} \cdots \sigma_{j-1}^{-1} \\ & \overset{\text{8.2)}}{=} \sigma_{j-1} \cdots \underline{\sigma_{r+2} b_{rs}^{(r+2)}} \sigma_{r+1} \sigma_r b_{ir} \sigma_r^{-1} \sigma_{r+1}^{-1} \cdots \sigma_{j-1}^{-1}\\
        & \overset{\text{reiterating}}{=} b_{rs}^{(j)} \sigma_{j-1} \cdots \sigma_r b_{ir} \sigma_r^{-1} \cdots \sigma_{j-1}^{-1} = b_{rs}^{(j)}b_{ij}. 
        \end{split}
    \end{equation*}

Where all of the used relations are from Theorem~\ref{BB}.

\begin{figure}[h]
\begin{center} 
    \includegraphics[width=14.5cm]{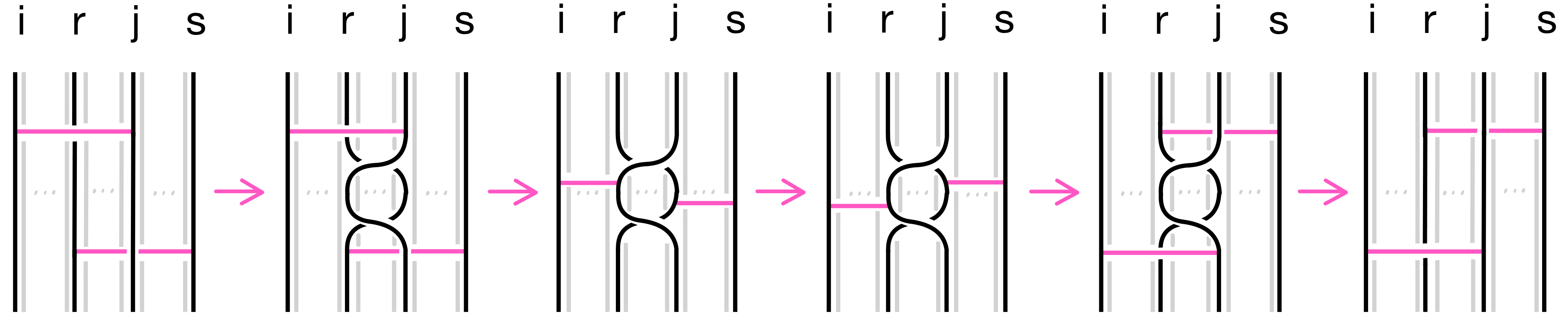}
\end{center}
\caption{Recovering the bond slide relation.}
\label{bondslide2}
\end{figure}
    
\noindent The above relation closely resembles the one in Figure~\ref{ArtAlt}. Generalising to underpasses indices sets for the bonds, this relation behaves similarly to relations 4) when $r, \, s \notin \cI$. The conditions for $\cR$ and $\cI$ are $r \notin \cI$ and $\cR \cap \{ r+1, \dots, j-1 \} \subseteq \cI \cap \{ r+1, \dots, j-1 \}$. The general version of Eq.~\ref{bondslide} is

\begin{equation} \label{bondslidegeneral}
        b_{ij}^{\cI} \, b_{rs}^{\cR} = b_{rs}^{\cR}\, b_{ij}^{\cI}, \quad \text{for } i<r<j<s, \text{ with } j \in \cR, \, r \notin \cI, \, \cI \cap \{ r+1, \dots, j-1 \} \subseteq \cR \cap \{ r+1, \dots, j-1 \}.
    \end{equation}

There is a symmetric version of Eq.~\ref{bondslidegeneral} for $b_{ij}^{\cI}$ sliding above $b_{rs}^{\cR}$, stating 
\begin{equation} \label{bondslidegeneral2}
    b_{ij}^{\cI} \, b_{rs}^{\cR} = b_{rs}^{\cR} \, b_{ij}^{\cI}, \, \text{for } i<r<j<s, \text{ with } r \in \cI, \, j \notin \cR, \, \cR \cap \{ r+1, \dots, j-1 \} \subseteq \cI \cap \{ r+1, \dots, j-1 \}.
\end{equation}

\noindent In this case the generalisation for underpasses indices sets have that $j \notin \cR$ and $\cI \cap \{ r+1, \dots, j-1 \} \subseteq \cR \cap \{ r+1, \dots, j-1 \}$.

Relations 3) and 4) in Definition~\ref{BB} and Eqs.~\ref{bondslidegeneral}, ~\ref{bondslidegeneral2} compile an exhaustive list of all allowed bond-bond interactions, since all possible orderings of indices for the nodes of the bonds have been considered.
All the other relations (from relations 5) on) of Definition~\ref{BB} deal instead with the interactions of bonds with crossings.
\end{remark}


\subsection{A reduced presentation for the bonded braid monoid}

\begin{proposition} \label{oversuff}
    The overpassing bonds $(b_{ij})_{1 \leq i < j \leq n}$ constitute a sufficient set of generators for $BB_n$.
\end{proposition}

\begin{proof}
    Consider, without loss of generality, a bond with one underpass, namely $b_{ij}^{(l)}$, where $i<l<j$. The argument extends naturally to the more general case of a bond $b_{ij}^{(l_1, \dots, l_h)}, \, i < l_1 < \dots < l_h < j$.\\
Using iteratively relation 5.1) of Definition~\ref{BB}, 
$$b_{ij}^{(l)} = \sigma_i^{-1} \, b_{i+1, j}^{(l)} \, \sigma_i = \dots = (\sigma_i^{-1} \ldots \sigma_{l-2}^{-1}) \, b_{l-1, j}^{(l)} \, (\sigma_{l-2} \ldots \sigma_i).$$ 
Using now relation 5.2) of Definition~\ref{BB} instead,
$$ (\sigma_i^{-1} \ldots \sigma_{l-2}^{-1}) \, b_{l-1, j}^{(l)} \, (\sigma_{l-2} \ldots \sigma_i) = (\sigma_i^{-1} \ldots \sigma_{l-2}^{-1}) \sigma_{l-1} \, b_{l j} \, \sigma_{l-1}^{-1} (\sigma_{l-2} \ldots \sigma_i) $$
so that the bond becomes an overpass. The bond $b_{ij}$ can be brought back to subtend on the strands $i$ and $j$ by using again relations 5.1) in  Definition~\ref{BB}:
\[
\begin{split}
    (\sigma_i^{-1} \ldots \sigma_{l-2}^{-1}) \sigma_{l-1} \, \underline{b_{l j}} \, \sigma_{l-1}^{-1} (\sigma_{l-2} \ldots \sigma_i) & = (\sigma_i^{-1} \ldots \sigma_{l-2}^{-1}) \sigma_{l-1}^2 \, \underline{b_{l-1, j}} \, \sigma_{l-1}^{-2} (\sigma_{l-2} \ldots \sigma_i) \\ & \vdots \\
    & = (\sigma_i^{-1} \ldots \sigma_{l-2}^{-1} \sigma_{l-1}^2 \sigma_{l-2} \ldots \sigma_i) \, b_{ij} \, (\sigma_i^{-1} \ldots \sigma_{l-2}^{-1} \sigma_{l-1}^{-2} \sigma_{l-2} \ldots \sigma_i)\\
    & = A_{i,l} \, b_{ij} \, A_{i,l}^{-1}
\end{split}
\]

\noindent The argument generalises to the \textit{right-threading of a generic bond}:
\begin{equation} \label{undertoover}
    b_{ij}^{(l_1, \ldots , l_h) } = A_{i,l_1} \cdots A_{i,l_h} \, b_{ij} \,  A_{i,l_h}^{-1} \cdots A_{i,l_1}^{-1}
\end{equation}
In Eq.~\ref{undertoover} notice the growing order of the indices $l_{-}$ with respect to the conjugation. Refer also to Figure~\ref{overbraid}. 
\end{proof}

\begin{definition} \label{threadingdef}
    The \emph{threading} of a generic bond $b_{ij}^{\cI}$ consists in unravelling its underpasses, thus rewriting is as a pure braid conjugate of an uniform overpassing bond (consult Figure~\ref{overbraid}). It can be done both in the right loop generators and the left loop generators.
\end{definition}

\begin{figure}[h]
\begin{center} 
    \includegraphics[width=14.5cm]{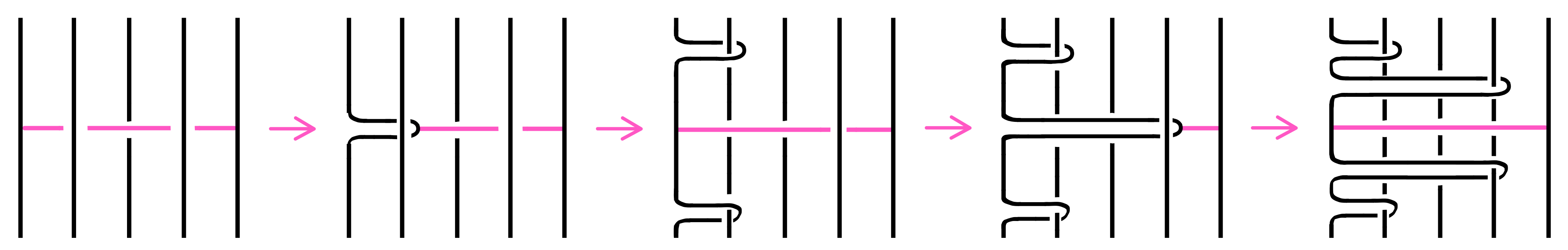}
\end{center}
\caption{ Right-threading a bond to a pure-braid conjugate of an overpassing bond.}
\label{overbraid}
\end{figure}

\noindent Note that the threading isotopy of the bond $b_{ij}^{(l)}$ could equally well take place towards the left end. In that case the left pure braid generators would be involved in the conjugating words, as opposed to the right ones involved in Eq.~\ref{undertoover} above.  Namely, the \textit{left-threading of a generic bond} is given by the formula:

\begin{equation} \label{undertooverleft}
    b_{ij}^{(l_1, \dots , l_h)} = A_{l_h, j}^{-1} \cdots A_{l_1, j}^{-1}\, b_{ij} \, A_{l_1, j} \cdots A_{l_h, j}
\end{equation}

\bigbreak
In view of the reduced set of generators, the relations in Definition~\ref{BB} reduce as well, since all of the underpassings will not appear any more. We shall examine them one-by-one applying Eq.\ref{undertoover} (or Eq.~\ref{undertooverleft} equally well). Our goal is to obtain a uniform overpass presentation. Indeed:

\begin{theorem}
\label{oBB}
    The bonded braid monoid on $n$ strands $BB_n$ admits a reduced presentation, also referred to as the \emph{uniform overpass  presentation}, with generators the elementary braid generators $\sigma_i$ \ for $1 \leq i \leq n-1$, their inverses $\sigma_i^{-1}$ and the uniform overpass bonds $b_{ij}$ \ for $1 \leq i < j \leq n$, with relations:

\[
\begin{array}{lrcll}
    1) & \sigma_i \sigma_j & = & \sigma_j \sigma_i & \text{for }  \ |i-j|>2\\[3pt]
    2) & \sigma_i \sigma_{i+1} \sigma_i & = & \sigma_{i+1} \sigma_i \sigma_{i+1} & \text{for } \ 1 \leq i \leq n-2 \\[3pt]
    3) & b_{ij} \, b_{rs} & = & b_{rs} \, b_{ij} & \text{for } \  i<j<r<s \quad \text{ or } \quad  i<r<s<j \\[3pt]
    4) & \sigma_i^{-1} \, b_{i+1, j} \, \sigma_{i} \, & = & b_{ij} & \text{for }  \ 1 \leq i \leq n-1, \, 1 \leq j \leq n \\[3pt]
    5) & \sigma_{j-1} \, b_{i, j-1} \, \sigma_{j-1}^{-1} & = & b_{ij}  & \text{for } \  1 \leq i \leq n-1, \, 1 \leq j \leq n \\[3pt]
    6) & \sigma_i \, b_i & = & b_i \, \sigma_i & \text{for } \  1 \leq i \leq n-1 \\[3pt]
    7) & \sigma_i \, b_{rs} & = & b_{rs} \, \sigma_i & \text{for } \  i \notin \{r-1, r, s-1, s\}, \, 1 \leq r < s \leq n\\[3pt]
    8) & \hspace{2mm} b_i \sigma_{i+1} \sigma_i \, & = & \, \sigma_{i+1}\sigma_i b_{i+1} \hspace{2mm} & \text{for } \  1 \leq i < n-1 \\[3pt]
    9) & \hspace{2mm} \sigma_i \sigma_{i+1} b_i \, & = & \, b_{i+1} \sigma_i \sigma_{i+1} \hspace{2mm} & \text{for }  \ 1 \leq i < n-1.
\end{array}
\] 
\end{theorem} 

\smallbreak

\begin{proof}
 We notice that relations 1), 2), 6), 9.1) and 9.2) of Definition~\ref{BB} do not involve bonds or underpasses, so these relations are not affected. We begin with examining relations 3) of Definition~\ref{BB}. 
\bigbreak
\noindent {\it Relations 3):} $b_{i j}^{\mathcal{I}} \, b_{r s}^{\mathcal{R}} \, = \,  b_{rs}^{\mathcal{R}} \, b_{ij}^{\mathcal{I}}$, \quad $i<j<r<s$, \ $\mathcal{I} \subseteq \left\{ i+1, \dots , j-1\right\}$, \ $\mathcal{R} \subseteq \left\{ r+1, \dots , s-1\right\}$. Applying Eq.~\ref{undertoover} (or Eq.~\ref{undertooverleft} equally well) the modified relations read
\[ 
\begin{split}(A_{i, i_1} \ldots A_{i,i_h}) \,  b_{ij} \, \underline{(A_{i,i_h}^{-1} \ldots A_{i,i_1}^{-1})} & \, \underline{(A_{r,r_1} \cdots A_{r,r_l})} \, b_{rs} \, (A_{r,r_l}^{-1} \ldots A_{r,r_1}^{-1}) \\ = & (A_{r,r_1} \ldots A_{r,r_l}) \, b_{rs} \, \underline{(A_{r,r_l}^{-1} \ldots A_{r,r_1}^{-1})} \, \underline{(A_{i, i_1} \ldots A_{i,i_h})} \,  b_{ij} \, (A_{i,i_h}^{-1} \ldots A_{i,i_1}^{-1})  
\end{split}\] where $\mathcal{I} = (i_1, \dots , i_h)$, $\mathcal{R} = (r_1, \dots , r_l)$. Because $i_1 < \dots < i_h < r_1 < \dots < r_l$, the underlined blocks commute.
Since $b_{ij}$ and $b_{rs}$ are sufficiently far apart, using relations 7) of Definition~\ref{BB} for $\cR = \emptyset$ the bonds commute with generators $A_{lm}^{\pm 1}$, and the relation reduces to
\[
\begin{split}
    (A_{i, i_1} \cdots A_{i,i_h}) \, (A_{r,r_1} \cdots A_{r,r_l}) \,  b_{ij} & \, b_{rs} \, (A_{r,r_l}^{-1} \cdots A_{r,r_1}^{-1}) \,  (A_{i,i_h}^{-1} \cdots A_{i,i_1}^{-1}) \\ & = (A_{i, i_1} \cdots A_{i,i_h}) \, (A_{r,r_1} \cdots A_{r,r_l}) \, b_{rs} \,  b_{ij} \, (A_{r,r_l}^{-1} \cdots A_{r,r_1}^{-1}) \,  (A_{i,i_h}^{-1} \cdots A_{i,i_1}^{-1}) \\
    & \implies b_{ij} \, b_{rs} \, = \, b_{rs} \, b_{ij}
\end{split}
\]
So, of the relations 3) can be kept only
\begin{equation} \label{over3}
    b_{ij} \, b_{rs} \, = \, b_{rs} \, b_{ij}, \quad i<j<r<s.
\end{equation}

\bigbreak
\noindent {\it Relations 4):} $b_{i j}^{\mathcal{I}} \, b_{r s}^{\mathcal{R}} \, = \,  b_{rs}^{\mathcal{R}} \, b_{ij}^{\mathcal{I}}$, \quad $i<r<s<j$, \ $\mathcal{I} \subseteq \left\{ i+1, \dots , j-1\right\}, \, \mathcal{R} \subseteq \left\{ r+1, \dots , s-1\right\}, \, \text{ if } r, s \in \cI \implies \mathcal{R} \subseteq \mathcal{I} \cap \{ r+1, \ldots s-1\}, \, \text{ if } r, s \notin \cI \implies  \mathcal{I} \subseteq \mathcal{R} \cap \{ r+1, \ldots s-1\}$:
\smallbreak
We begin with $r, s \notin \cI$.
Consider first the case of $\cI = \emptyset$, so $b_{ij} \, b_{rs}^{\mathcal{R}} \, = \,  b_{rs}^{\mathcal{R}} \, b_{ij}$. The modified relations read
$$
b_{ij} \, (A_{r,r_1} \cdots A_{r,r_l}) \, b_{rs} \, (A_{r,r_l}^{-1} \cdots A_{r,r_1}^{-1}) = (A_{r,r_1} \cdots A_{r,r_l}) \, b_{rs} \, (A_{r,r_l}^{-1} \cdots A_{r,r_1}^{-1}) \, b_{ij}
$$
Using relations 8.1) of Definition~\ref{BB} in the case of $\cR = \emptyset$ for $b_{ij}$, one gets
$$
(A_{r,r_1} \cdots A_{r,r_l}) \, b_{ij} \, b_{rs} \,  (A_{r,r_l}^{-1} \cdots A_{r,r_1}^{-1}) = (A_{r,r_1} \cdots A_{r,r_l}) \, b_{rs} \, b_{ij} \, (A_{r,r_l}^{-1} \cdots A_{r,r_1}^{-1})
$$
So, of this case it can be kept only
\begin{equation} \label{over4}
    b_{ij} \, b_{rs} \, = \, b_{rs} \, b_{ij}, \quad i<r<s<j.
\end{equation}

\smallbreak
Consider now the more general subcase 
 $\cR = \mathcal{I} = (l) $ with $l \neq r, s$. \\ 
 Then relations 4) become $b_{ij}^{(l)} \, b_{rs}^{(l)} = b_{rs}^{(l)} \, b_{ij}^{(l)}$, which are rewritten to 
\[
A_{il} \, b_{ij} \, A_{il}^{-1} \, A_{rl} \, b_{rs} \, A_{rl}^{-1} = A_{rl} \, b_{rs} \, A_{rl}^{-1} \, A_{il} \, b_{ij} \, A_{il}^{-1}.
\] Diagrammatically, the equivalence is quite clear. We proceed with caution in two steps.\\

\smallbreak

\noindent \textit{Step I), Figure~\ref{rel5}: Simplify the word between $b_{ij}$ and $b_{rs}$.} Note that, for $i<r<l<s$, from Relations 2) and 3) of Proposition~\ref{alternative} it can be derived that
\begin{equation} \label{fromstep1}
    A_{il}^{-1}A_{rl} = A_{rl} (A_{ir} A_{il}^{-1} A_{ir}^{-1})
\end{equation}

Then, because $i<r<l<j$ Relations 4) are rewritten in \begin{equation*}
    \begin{split}
        A_{il} \, b_{ij} \, \underline{A_{rl}} (A_{ir} A_{il}^{-1} A_{ir}^{-1}) \, b_{rs} \, A_{rl}^{-1} & = A_{rl} \, b_{rs}\, (A_{ir} A_{il} A_{ir}^{-1}) \underline{A_{rl} ^{-1}} \, b_{ij} \, A_{il}^{-1} \\
        \iff A_{il} \, A_{rl} \, b_{ij}  (A_{ir} A_{il}^{-1} A_{ir}^{-1}) \, b_{rs} \, A_{rl}^{-1} & = A_{rl} \, b_{rs}\, (A_{ir} A_{il} A_{ir}^{-1})  b_{ij} \, A_{rl} ^{-1} \, A_{il}^{-1}
    \end{split}
\end{equation*}
where $A_{rl}$ and its inverse slide under $b_{ij}$ due to relations 8.1) of Definition~\ref{BB}.

\smallbreak

\noindent \textit{Step II): Slide $b_{rs}$ above the simplified word.} Shift the focus on $(A_{ir} A_{il}^{-1} A_{ir}^{-1}) \, b_{rs}$: writing the pure braids in terms of elementary braid generators, the already simplified expression is equal to
\begin{equation*}
    \begin{split}
        \sigma_i^{-1} \cdots  \sigma_{r-2}^{-1} & \sigma_{r-1} \sigma_r^{-1} \cdots \sigma_{l-1}^{-2} \cdots \sigma_r \sigma_{r-1}^{-1} \sigma_{r-2} \cdots \sigma_i \, \underline{b_{rs}} = \\ & \stackrel{\text{rel. 7)}}{=} \sigma_i^{-1} \cdots \sigma_{r-2}^{-1} \sigma_{r-1} \sigma_r^{-1} \cdots \sigma_{l-1}^{-2} \cdots \sigma_r \underline{\sigma_{r-1}^{-1} \, b_{rs}} \, \sigma_{r-2} \cdots \sigma_i \\
        & \stackrel{\text{rel. 5.1)}}{=} \sigma_i^{-1} \cdots \sigma_{r-2}^{-1} \sigma_{r-1} \sigma_r^{-1} \cdots \sigma_{l-1}^{-2} \cdots \sigma_r \, \underline{b_{r-1, s}} \, \sigma_{r-1}^{-1} \sigma_{r-2} \cdots \sigma_i \\
        & \stackrel{\text{rel. 8.1)}}{=} \sigma_i^{-1} \cdots \sigma_{r-2}^{-1} \underline{\sigma_{r-1} \, b_{r-1, s}} \, \sigma_r^{-1} \cdots \sigma_{l-1}^{-2} \cdots \sigma_r \sigma_{r-1}^{-1} \sigma_{r-2} \cdots \sigma_i \\
        & \stackrel{\text{rel. 5.1)}}{=} \sigma_i^{-1} \cdots \sigma_{r-2}^{-1} \, \underline{b_{rs}} \, \sigma_{r-1} \sigma_r^{-1} \cdots \sigma_{l-1}^{-2} \cdots \sigma_r \sigma_{r-1}^{-1} \sigma_{r-2} \cdots \sigma_i \\
        & \stackrel{\text{7)}}{=} b_{rs} \, \sigma_i^{-1} \cdots \sigma_{r-2}^{-1} \sigma_{r-1} \sigma_r^{-1} \cdots \sigma_{l-1}^{-2} \cdots \sigma_r \sigma_{r-1}^{-1} \sigma_{r-2} \cdots \sigma_i \\
    \end{split}
\end{equation*}
which means that $(A_{ir} A_{il}^{-1} A_{ir}^{-1}) \, b_{rs} = b_{rs} (A_{ir} A_{il}^{-1} A_{ir}^{-1})$. Since the exponent of $A_{il}$ plays no role in the proof, it is also true that $(A_{ir} A_{il} A_{ir}^{-1}) \, b_{rs} = b_{rs} (A_{ir} A_{il} A_{ir}^{-1})$. Relation 4) then becomes
\[
\begin{split}
    A_{il} \, A_{rl} \, b_{ij} \, b_{rs} \, \underline{(A_{ir} A_{il}^{-1} A_{ir}^{-1}) \, A_{rl}^{-1}} & = \underline{A_{rl} \,  (A_{ir} A_{il} A_{ir}^{-1})} \, b_{rs}\, b_{ij} \, A_{rl} ^{-1} \, A_{il}^{-1} \\ \iff A_{il} A_{rl} \, b_{ij} \, b_{rs} \, A_{rl}^{-1} A_{il}^{-1} & = A_{il} A_{rl} \, b_{rs} \, b_{ij} \, A_{rl}^{-1} A_{il}^{-1}
\end{split}
\]

\begin{figure}[h]
\begin{center} 
\includegraphics[width=15.5cm]{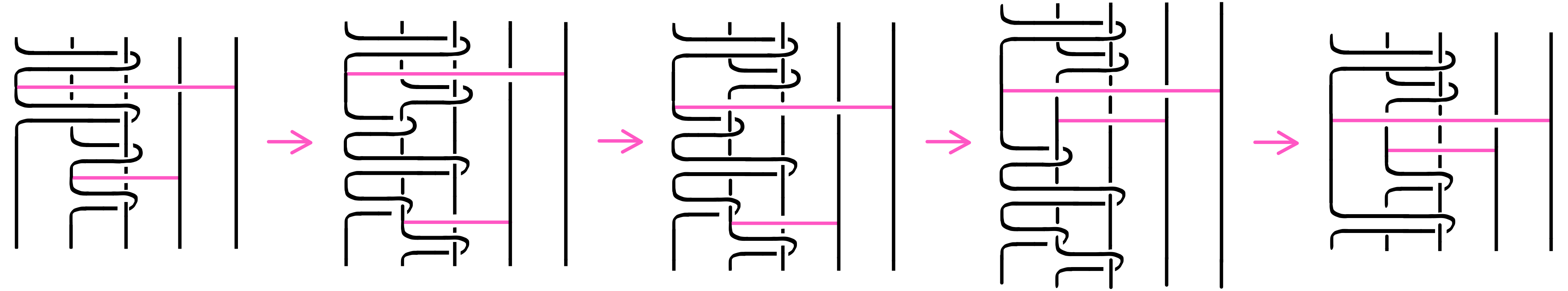}
\end{center}
\caption{Equivalence process on the right-hand side of relation 4).}
\label{rel5}
\end{figure}

\noindent Cancelling the conjugation by $A_{il} A_{rl}$ on both hand sides, it follows that it is sufficient to keep of this case of relations 4) the ones as in Eq.~\ref{over4}:
    $$b_{ij} \, b_{rs} = b_{rs} \, b_{ij}, \quad i<r<j<s.$$

\bigbreak
In the general $\mathcal{I}$ case, Relations 4) read as follows

\[ 
\begin{split}(A_{i, i_1} \ldots A_{i,i_h}) \,  b_{ij} \, (A_{i,i_h}^{-1} \ldots A_{i,i_1}^{-1}) & \, (A_{r,r_1} \cdots A_{r,r_l}) \, b_{rs} \, (A_{r,r_l}^{-1} \ldots A_{r,r_1}^{-1}) \\ = & (A_{r,r_1} \ldots A_{r,r_l}) \, b_{rs} \, (A_{r,r_l}^{-1} \ldots A_{r,r_1}^{-1}) \, (A_{i, i_1} \ldots A_{i,i_h}) \,  b_{ij} \, (A_{i,i_h}^{-1} \ldots A_{i,i_1}^{-1})  
\end{split}\]

Step II) is more or less the same, with the difference that there will be more negative exponents in the extended writing of the braid block with which $b_{rs}$ commutes. Step I) is a tad more stubborn, but not tragic, it is carried out within the block $(A_{i,i_h}^{-1} \cdots A_{i,i_1}^{-1}) \, (A_{r,r_1} \cdots A_{r,r_l})$ and in its inverse in the two hand sides of relations 4), and heavily depends on $\mathcal{I}$ and $\mathcal{R}$. 
If the set $\mathcal{K}$ of the underpasses of $b_{ij}$ happening outside of the strands involved in the bond $b_{rs}$ is non-empty, these underpasses do not obstruct the relations. 

\smallbreak
From the above, we consider the general case of $\mathcal{I} \cap \{r+1, \ldots , s-1 \} \subseteq \mathcal{R}$. Without loss of generality we can impose that $\cI \subseteq \cR$, since any other underpasses of $b_{ij}$ happening outside of $\{r+1, \ldots , s-1 \}$ can be brought below or above $b_{rs}$ and eliminated. 

\smallbreak
Step I) is carried out in the block $(A_{i,i_h}^{-1} \cdots A_{i,i_1}^{-1}) \, (A_{r,r_1} \cdots A_{r,r_l})$ and its inverse in both hand sides of relations 4). Since $\mathcal{I} \subseteq \mathcal{R}$, either $i_1  > r_1$ or $i_1 = r_1$. Indeed, if one had that $i_1 < r_1$, since $r_1$ is the smallest index in $\mathcal{R}$, this latter set would not contain $\mathcal{I}$. 

\smallbreak
If $r_1 < i_1$ then $r_1$ is smaller than any other index $i_p$: then, because $i<r<r_1<i_p<j$, in the left-hand side of relations 4) $A_{r,r_1}$ slides to the left of $b_{ij}$. Remind that $i<r<r_1<i_1<i_p$ for any $i_p$. Using pure braid relations,
\[
\begin{split}
A_{i, i_1} \cdots A_{i,i_h} \,  b_{ij} \, A_{i,i_h}^{-1} \cdots A_{i,i_1}^{-1} \, & \underline{A_{r,r_1}} \cdots A_{r,r_l} \, b_{rs} \, A_{r,r_l}^{-1} \cdots A_{r,r_1}^{-1} \\ & = A_{i, i_1} \cdots A_{i,i_h} A_{r,r_1} \,  b_{ij} \, A_{i,i_h}^{-1} \cdots A_{i,i_1}^{-1} \, A_{r,r_2}\cdots A_{r,r_l} \, b_{rs} \, A_{r,r_l}^{-1} \cdots A_{r,r_1}^{-1}
\end{split}
\]

Proceed analogously on the right-hand side: $A_{r,r_1}^{-1}$ slides to the right of $b_{ij}$:

\[
\begin{split}
A_{r,r_1} \cdots A_{r,r_l} \, b_{rs} \, A_{r,r_l}^{-1} \cdots \underline{A_{r,r_1}^{-1}} & \, A_{i, i_1} \cdots A_{i,i_h} \,  b_{ij} \, A_{i,i_h}^{-1} \cdots A_{i,i_1}^{-1} \\ & = A_{r,r_1} \cdots A_{r,r_l} \, b_{rs} \, A_{r,r_l}^{-1} \cdots A_{r,r_2}^{-1} \, A_{i, i_1} \cdots A_{i,i_h} \,  b_{ij} \, A_{r,r_1}^{-1} A_{i,i_h}^{-1} \cdots A_{i,i_1}^{-1} 
\end{split}
\]
This is repeated until we reach an $r'$ such that $r' = i_1$.

\smallbreak
Else, for $i_1 = r_1$ in the left-hand side $A_{r, r_1}$ can slide up unitl a certain point, and the block of interest looks like 
\[
\begin{split}
     A_{i,i_h}^{-1} \cdots A_{i,i_2}^{-1} \underline{A_{i,r_1}^{-1} \, A_{r,r_1}} A_{r,r_2} \cdots A_{r,r_l}
    \stackrel{Eq.~\ref{fromstep1}}{=} & A_{i,i_h}^{-1} \cdots A_{i,i_2}^{-1} A_{r,r_1} \underline{(A_{ir} A_{i,r_1}^{-1} A_{ir}^{-1}) A_{r,r_2}} \cdots A_{r,r_l}
\end{split}
\] 

Now, $A_{r, r_1}$ can overcome any $A_{i,i'}^{-1}$ on its left by standard (pure) braid commutativity, but to check that $(A_{ir} A_{i,r_1}^{-1} A_{ir}^{-1})$ is free to do the same with any $A_{r,r'}$ on its right we use relations 4) of Theorem~\ref{alternative} and that, for $i<r<l<s$,
\begin{equation} \label{usefulinrel4}
    A_{rl}^{-1}A_{il}^{-1}A_{rl} = A_{ir} A_{il}^{-1} A_{ir}^{-1}
\end{equation}
We proceed modifying relations 4):
\begin{equation*} 
    \begin{split}
        & A_{r,r_1} A_{r,r'} \underline{A_{r,r_1}^{-1} A_{i,r_1}} = \underline{A_{i,r_1} A_{r,r_1}} A_{r,r'} A_{r,r_1}^{-1}\\
        & \stackrel{\text{Invert}}{\iff} \underline{A_{r,r_1}^{-1} A_{i,r_1}^{-1} A_{r,r_1}} A_{r,r'} = A_{r,r'} \underline{A_{r,r_1}^{-1} A_{i,r_1}^{-1} A_{r,r_1}} \\ & \stackrel{\text{Eq.~\ref{usefulinrel4}}}{\iff} (A_{ir} A_{i,r_1}^{-1} A_{ir}^{-1}) A_{r,r'} = A_{r,r'} (A_{ir} A_{i,r_1}^{-1} A_{ir}^{-1})
    \end{split}
\end{equation*}
Then,
\[
A_{i,i_h}^{-1} \cdots A_{i,i_2}^{-1} \underline{A_{r,r_1}} \underline{(A_{ir} A_{i,r_1}^{-1} A_{ir}^{-1}) A_{r,r_2}} \cdots A_{r,r_l} = A_{r,r_1} A_{i,i_h}^{-1} \cdots A_{i,i_2}^{-1} A_{r,r_2} \cdots A_{r,r_l} (A_{ir} A_{i,r_1}^{-1} A_{ir}^{-1})
\]
The analogous process is carried out on the right-hand side.
The relation now looks like
\[
\begin{split}
    A_{i, r_1} \cdots A_{i,i_h} & \underline{b_{ij} \, A_{r,r_1}} \,  A_{i,i_h}^{-1} \cdots A_{i,i_2}^{-1} A_{r,r_2} \cdots A_{r,r_l} (\underline{A_{ir} A_{i,r_1}^{-1} A_{ir}^{-1}}) \, b_{rs} \, A_{r,r_l}^{-1} \cdots A_{r,r_1}^{-1} \\ & = A_{r,r_1} \cdots A_{r,r_l} \, b_{rs} \, (\underline{A_{ir} A_{i, r_1} A_{ir}^{-1}}) A_{r,r_l}^{-1} \cdots A_{r,r_2}^{-1} \, A_{i, r_1} \cdots A_{i,i_h} \, \underline{ A_{r,r_1}^{-1} \, b_{ij}} \, A_{i,i_h}^{-1} \cdots A_{i,r_1}^{-1} 
\end{split}
\]

For Step II) of the general case  we argue as follows: Step II) of the previous case ($\cR = \cI = (l)$) ensures that $(A_{ir} A_{i,r_1}^{-1} A_{ir}^{-1})$ commutes with $b_{rs}$, and clearly $A_{r,r_1}$ slides under $b_{ij}$. Apply the analogous process on the other hand side, and then use the inverse processes of Step I) on both of them to bring $(A_{ir} A_{i,r_1}^{-1} A_{ir}^{-1})$ and $(A_{ir} A_{i,r_1} A_{ir}^{-1})$ at the top and bottom respectively of their relative hand sides (we can however leave $A_{r,r_1}^{\pm 1}$ next to $b_{rs}$). In the left-hand side, $(A_{ir} A_{i,r_1}^{-1} A_{ir}^{-1})$ meets $A_{r, r_1}^{-1}$, and by pure braid relations their product is equal to $A_{r, r_1}^{-1}A_{i, r_1}^{-1}$. Analogously on the right-hand side. Relations 4) becomes
\[
\begin{split}
    A_{i, i_1} \cdots A_{i,i_h} A_{r,r_1} \,  b_{ij} \, A_{i,i_h}^{-1} \cdots & \underline{A_{i,i_2}^{-1} \, A_{r,r_2}} \cdots A_{r,r_l} \, b_{rs} \, A_{r,r_l}^{-1} \cdots A_{r,r_1}^{-1} A_{i,i_1}^{-1} \\ & = A_{i, i_1} A_{r,r_1} \cdots A_{r,r_l} \, b_{rs} \, A_{r,r_l}^{-1} \cdots \underline{A_{r,r_2}^{-1} A_{i, i_2}} \cdots A_{i,i_h} \,  b_{ij} A_{r,r_1}^{-1} \, A_{i,i_h}^{-1} \cdots A_{i,i_1}^{-1}
\end{split}
\]

The same algorithm Step I) + Step II) applied on the pairs $(r,r_1), (i,i_1)$ is now applied on $(r,r_2), (i, i_2)$ and so on. Since $\cI \subseteq \cR$, it can happen that there are more pairs $(r, r')$ than $(i, i')$. In this case, the pairs of $\cR$ left over at the end of the algorithm are covered by the “if $r_1 < i_1$” case of relations 4). In the end, we get to

\[
\begin{split}
    A_{i, i_1} \cdots A_{i,i_h} A_{r,r_1} \, A_{r,r_2} \cdots A_{r,r_l} \, b_{ij} &  \, b_{rs} \, A_{r,r_l}^{-1} \cdots A_{r,r_1}^{-1} A_{i,i_h}^{-1} \cdots A_{i,i_2}^{-1} A_{i,i_1}^{-1} \\ & = A_{i, i_1} \cdots A_{i,i_h} A_{r,r_1} \, A_{r,r_2} \cdots A_{r,r_l} \, b_{rs} \, b_{ij} \,A_{r,r_l}^{-1} \cdots A_{r,r_1}^{-1} A_{i,i_h}^{-1} \cdots A_{i,i_2}^{-1} A_{i,i_1}^{-1}
\end{split}
\]

\noindent Removing the conjugation terms from both hand sides, we finally have that we only need to keep from relations 4) with $r, s \notin \cI$ the relations $b_{ij} \, b_{rs} \, = \, b_{rs} \, b_{ij}$, as in Eq.~\ref{over4}.\\

\smallbreak

For the case $r, s \in \cI$, since we have already gotten acquainted with the machinery, we move directly to the general case of the intervals $\cI$ and $\cR$. The condition on the intervals for this case has that $\cR \subseteq \cI \cap \{r+1, \dots , s-1 \}$. Without loss of generality, as we did for the general case of $r, s \notin \cI$, we can assume that $\cI$ sits entirely in the interval $[ r, s ]$, since what happens outside of it won't affect the relation. 
\smallbreak
Fix $\cR = (r_1, \dots , r_l) \subseteq \cI = (r, i_1, \dots i_h, s)$. The modified relations look like:
\[
\begin{split}
    A_{ir} & A_{i, i_1} A_{i, i_2} \cdots A_{i, r_1} \cdots A_{i,i_h} A_{is} \, b_{ij} \, A_{is}^{-1} A_{i,i_h}^{-1} \cdots A_{i, r_1}^{-1} \cdots A_{i, i_2}^{-1} A_{i,i_1}^{-1} A_{ir}^{-1} \, A_{r,r_1} \cdots A_{r,r_l} \, b_{rs} \, A_{r,r_l}^{-1} \cdots A_{r,r_1}^{-1} \\ & = A_{r,r_1} \cdots A_{r,r_l} \, b_{rs} \, A_{r,r_l}^{-1} \cdots \underline{A_{r,r_1}^{-1} \, A_{ir} A_{i, i_1}} A_{i, i_2} \cdots A_{i, r_1} \cdots A_{i,i_h} A_{is} \, b_{ij} \, A_{is}^{-1} A_{i,i_h}^{-1} \cdots A_{i, r_1}^{-1} \cdots A_{i, i_2}^{-1} A_{i,i_1}^{-1} A_{ir}^{-1}.
\end{split}
\]

We focus computations on the right-hand side, those on the left-hand side will follow by symmetry. 
\smallbreak
Step I): There are two possible cases: either $i_1 = r_1$ \ or \ $i_1 < r_1$.
\smallbreak
If $i_1 = r_1$ the underlined block is $A_{r,r_1}^{-1} \, A_{ir} A_{i, r_1} = A_{ir} A_{i, r_1} \, \underline{A_{r, r_1}^{-1}}$ from the cyclic relations of Eq.~\ref{cyclicrels} with $i < r < r_1$. Then, $A_{r, r_1}^{-1}$ slides all the way to $b_{ij}$, commutes with it, and goes to the bottom of the expression. There it will form the block $A_{r, r_1}^{-1} A_{i,r_1}^{-1} A_{ir}^{-1} = A_{i,r_1}^{-1} A_{ir}^{-1} A_{r, r_1}^{-1}$ again from the cyclic relations. Now the relation's right-hand side looks like
\begin{equation} \label{rsnotinIfirstpart}
    A_{r,r_1} \cdots A_{r,r_l} \, b_{rs} \, A_{r,r_l}^{-1} \cdots A_{r,r_2}^{-1} \, A_{ir} A_{i, i_1} \cdots A_{i,i_h} A_{is} \, b_{ij} \, A_{is}^{-1} A_{i,i_h}^{-1} \cdots A_{i,i_1}^{-1} A_{ir}^{-1} \underline{A_{r,r_1}^{-1}}.
\end{equation}
We have achieved that $A_{r, r_1}^{-1}$ sits at the end of the expression.
\smallbreak
Otherwise, if $i_1 < r_1$ there are a few steps to take before reaching an index $i_g = r_1$. Note that this index must exist, since $\cR \subseteq \cI$.
 We apply the following algorithm to the underlined block of the modified relation above (insertion of $(A_{ir}^{-1} A_{ir})$ together with Eq.~\ref{alternativeslideartin}): 

\begin{equation} \label{step1rsinI.2}
    A_{r,r_1}^{-1} \, A_{ir} A_{i, i_1} = \underline{A_{r,r_1}^{-1} \, A_{ir} A_{i, i_1} (A_{ir}^{-1}} A_{ir}) = A_{ir} A_{i, i_1} \underline{A_{ir}^{-1} A_{r, r_1}^{-1} A_{ir}}
\end{equation}

Note that $A_{i, i_1}$ comes now before $A_{r, r_1}^{-1}$.
If after $i_1$ there is another index $i_2 < r_1$, this algorithm can be reiterated on the underlined word in the right-hand side of Eq.~\ref{step1rsinI.2}:
\[
\begin{split}
    (A_{ir} A_{i, i_1}) A_{ir}^{-1} A_{r, r_1}^{-1} A_{ir} A_{i, i_2} & = (A_{ir} A_{i, i_1}) A_{ir}^{-1} \underline{A_{r, r_1}^{-1} A_{ir} A_{i, i_2} (A_{ir}^{-1}} A_{ir}) \\ &= (A_{ir} A_{i, i_1}) \underline{A_{ir}^{-1} A_{ir}} A_{i, i_2} A_{ir}^{-1} A_{r, r_1}^{-1} A_{ir} = A_{ir} A_{i, i_1} A_{i, i_2} \underline{A_{ir}^{-1} A_{r, r_1}^{-1} A_{ir}}
\end{split}
\]
Note that $A_{i, i_2}$ comes now before $A_{r, r_1}^{-1}$. The above procedure is illustrated in Figure~\ref{rsinIalgo}.

Repeat this for all intermediate indices in order to approach $A_{i, i_g = r_1}$. Then, operating on the underlined part,
$$A_{ir}^{-1} A_{r, r_1}^{-1} A_{ir} A_{i, r_1} = A_{i, r_1} A_{r, r_1}^{-1} $$
Plugging this in the modified relation's right-hand side we obtain:
\[
\begin{split}
    A_{r,r_1} \cdots A_{r,r_l} \, b_{rs} & \, A_{r,r_l}^{-1} \cdots  A_{r, r_2}^{-1} A_{ir} A_{i, i_1} A_{i, i_2} \cdots A_{i, i_{g-1}} \cdot \\ & \cdot A_{i, r_1} \underline{A_{r, r_1}^{-1}} A_{i, i_{g+1}} \cdots A_{i,i_h} A_{is} \, b_{ij} \, A_{is}^{-1} A_{i,i_h}^{-1} \cdots \underline{A_{i, r_1}^{-1} \cdots A_{i,i_1}^{-1} A_{ir}^{-1}}
\end{split}
\]
\begin{figure}[H]
\begin{center} 
    \includegraphics[width=16.5cm]{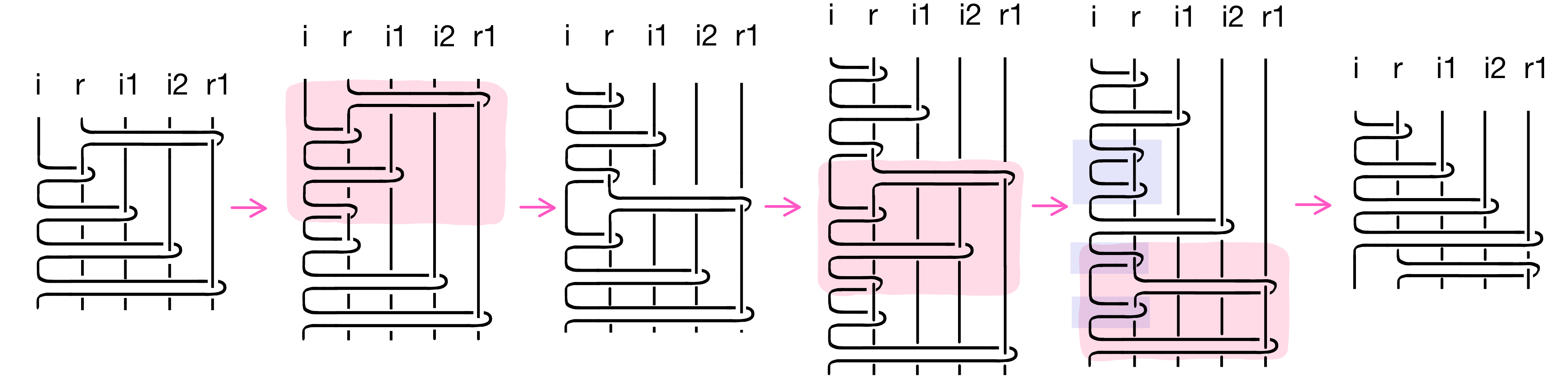}
\end{center}
\caption{Algorithm for Step I) when $r, s \in \cI$ and $i < r < i_1 < i_2 < r_1$. The blue sections cancel each other.}
\label{rsinIalgo}
\end{figure}

Since $i < r < r_1 = i_g < i_{g+1} < \dots < i_h$, \ $A_{r, r_1}^{-1}$ can slide under $b_{ij}$ to the bottom of the expression, where the inverse of the proposed algorithm is going to be applied on the block $A_{i, r_1}^{-1} \cdots A_{i,i_1}^{-1} A_{ir}^{-1}$ to ensure the commutation. In the end, the expression will look like Eq.~\ref{rsnotinIfirstpart}, that is $A_{r, r_1}^{-1}$ will sit at the end of the expression.
\smallbreak
Apply the analogous process on the left-hand side by symmetry, and repeat all steps on $A_{r, r_2}^{-1}$. In the end the modified relations will look like
\[
\begin{split}
    A_{r,r_1} \cdots A_{r,r_l} \, A_{ir} A_{i, i_1} \cdots & A_{i,i_h} A_{is} \, b_{ij} \, A_{is}^{-1} A_{i,i_h}^{-1} \cdots A_{i,i_1}^{-1} A_{ir}^{-1} \, b_{rs} \, A_{r,r_l}^{-1} \cdots A_{r,r_1}^{-1} \\ & = A_{r,r_1} \cdots A_{r,r_l} \, b_{rs} \, A_{ir} A_{i, i_1} \cdots A_{i,i_h} A_{is} \, b_{ij} \, A_{is}^{-1} A_{i,i_h}^{-1} \cdots A_{i,i_1}^{-1} A_{ir}^{-1} \, A_{r,r_l}^{-1} \cdots A_{r,r_1}^{-1}
\end{split}
\]
Equivalently,
\[
   A_{ir} A_{i, i_1} \cdots A_{i,i_h} A_{is} \, b_{ij} \, A_{is}^{-1} A_{i,i_h}^{-1} \cdots A_{i,i_1}^{-1} A_{ir}^{-1} \, b_{rs} \, = \, b_{rs} \, A_{ir} A_{i, i_1} \cdots A_{i,i_h} A_{is} \, b_{ij} \, A_{is}^{-1} A_{i,i_h}^{-1} \cdots A_{i,i_1}^{-1} A_{ir}^{-1}
\]

\smallbreak
Step II) follows immediately since $b_{rs}$ commutes with $A_{ir} A_{i, i_1} \cdots A_{i,i_h} A_{is}$ and its inverse. This can be seen by applying a few vertex slide moves on the expanded writing of $A_{ir} A_{i, i_1} \cdots A_{i,i_h} A_{is}$ and its inverse in terms of elementary braid generators. Then, by removing conjugations by the same words on the hand sides, the modified relations 4) for $r, s \in \cI$ are reduced to
\[ b_{ij} \, b_{rs} \, = \, b_{rs} \, b_{ij}
\]

which boil down to the cases already studied for relations 4).

\bigbreak
\noindent {\it Relations 5.1):}   $\sigma_i^{-1} b_{i+1, j}^{\mathcal{I}} \sigma_i = b_{ij}^{\mathcal{I}}$ for $ 1 \leq i<j \leq n-1, \, |i-j|>1$, where $\mathcal{I} = ( i_1, \dots , i_h) \subseteq \left\{ i+2, \dots , j-1 \right\}$ is an ordered set. The modifications read
$$
\sigma_i^{- 1} \, (A_{i+1, i_1} \cdots A_{i+1, i_h}) \, b_{i+1, j} \, (A_{i+1, i_h}^{-1} \cdots A_{i+1, i_1}^{-1}) \sigma_i = (A_{i, i_1} \cdots A_{i,i_h}) \, b_{i j} \, (A_{i,i_h}^{-1} \cdots A_{i,i_1}^{-1})
$$
Placing in between all terms of the left-hand side $\sigma_i \sigma_i^{-1}$ it becomes:
\[
\begin{split}
 & \underline{\sigma_i^{- 1} \, (A_{i+1, i_1} \sigma_i }\sigma_i^{-1} \cdots \sigma_i \underline{\sigma_i^{-1} A_{i+1, i_h}) \sigma_i} \sigma_i^{-1} \, b_{i+1, j} \, \sigma_i \underline{\sigma_i^{-1} (A_{i+1, i_h}^{-1} \sigma_i} \sigma_i^{-1} \cdots \sigma_i \underline{\sigma_i^{-1} A_{i+1, i_1}^{-1})\,  \sigma_i} \\
 & \stackrel{M_1) \text{ of Prop. } \ref{redBn}}{=} \, (A_{i, i_1} \cdots A_{i, i_h}) \, \sigma_i^{-1} \, b_{i+1, j} \, \sigma_i \, (A_{i, i_h}^{-1} \cdots A_{i, i_1}^{-1}) 
\end{split}
\]
Comparing this last expression with the modified right-hand side 
\[(A_{i, i_1} \cdots A_{i, i_h}) \, \sigma_i^{-1} b_{i+1, j} \sigma_i \, (A_{i, i_h}^{-1} \cdots A_{i, i_1}^{-1}) = (A_{i, i_1} \cdots A_{i, i_h}) \, b_{ij} \, (A_{i, i_h}^{-1} \cdots A_{i, i_1}^{-1})
\]
and simplifying the conjugates, we obtain
\begin{equation} \label{over51}
    \sigma_i^{-1} \, b_{i+1, j} \, \sigma_i = b_{ij}.
\end{equation}

\bigbreak
\noindent {\it Relations 5.2):}  $\sigma_i b_{i+1, j}^{\mathcal{I}} \, \sigma_i^{-1} = b_{ij}^{(i+1) \cup \mathcal{I}}$ for $1 \leq i<j \leq n-1, \, |i-j|>1$ and $ \mathcal{I} = (i_1, \dots , i_h) \subseteq \left\{ i+2, \dots , j-1 \right\}$ an ordered subset. The modifications read
$$
\sigma_i (A_{i+1, i_1} \cdots A_{i+1, i_h}) \, b_{i+1, j} \, (A_{i+1, i_h}^{-1} \cdots A_{i+1, i_1}^{-1}) \sigma_i^{-1} = (A_{i, i+1} A_{i, i_1} \cdots A_{i, i_h}) \, b_{ij} \, (A_{i, i_h}^{-1} \cdots A_{i, i_1}^{-1} A_{i, i+1}^{-1})
$$
Placing $\sigma_i \sigma_i^{-1}$ between all terms in the left-hand side is equal to conjugating all pure braid generators and the $b_{i+1, j}$ by $\sigma_i$. Since conjugates of pure braid generators by elementary braidings are words of pure braid generators and conjugates of bonds are bonds, using $M_1)$ and $\Sigma_3)$ of Proposition.~\ref{redBn} we have: 
\[
\begin{split}
& \sigma_i (A_{i+1, i_1}  \cdots A_{i+1, i_h}) \, b_{i+1, j} \, (A_{i+1, i_h}^{-1} \cdots A_{i+1, i_1}^{-1}) \sigma_i^{-1} \\  & = \underline{\sigma_i \sigma_i} (\underline{\sigma_i^{-1} A_{i+1, i_1} \sigma_i} \sigma_i^{-1} \cdots \sigma_i \underline{\sigma_i^{-1} A_{i+1, i_h} \sigma_i} ) \sigma_i^{-1} \, b_{i+1, j} \, \sigma_i (\underline{\sigma_i^{-1} A_{i+1, i_h}^{-1} \sigma_i} \sigma_i^{-1} \cdots \sigma_i \underline{ \sigma_i^{-1} A_{i+1, i_1}^{-1} \sigma_i}) \underline{\sigma_i^{-1} \sigma_i^{-1}} \\
& = A_{i, i+1} (A_{i, i_1} \cdots A_{i, i_h}) \, \sigma_i^{-1} \, b_{i+1, j} \, \sigma_i \, (A_{i, i_h}^{-1} \cdots A_{i, i_1}^{-1}) A_{i, i+1}^{-1}.
\end{split}
\]

Comparing with the right-hand side and removing the same conjugating terms on both hand sides, the relation reduces to 
$$ \sigma_i^{-1} \, b_{i+1, j} \, \sigma_i = b_{ij}$$
and is absorbed in Eq.~\ref{over51}.

\bigbreak
\noindent {\it Relations 5.3):}   $\sigma_{j-1} b_{i, j-1}^{\mathcal{I}} \sigma_{j-1}^{-1} = b_{ij}^{\mathcal{I}}$ for $ 1 \leq i<j \leq n-1, \, |i-j|>1$, where $\mathcal{I} = ( i_1, \dots , i_h) \subseteq \left\{ i+1, \dots , j-2 \right\}$ is an ordered set. The modifications read
\[
(A_{i, i_1} \cdots A_{i, i_h}) \, b_{ij} \, (A_{i, i_h}^{-1} \cdots A_{i, i_1}^{-1}) = \sigma_{j-1} \, (A_{i, i_1} \cdots A_{i, i_h}) \, b_{i, j-1} \, (A_{i, i_h}^{-1} \cdots A_{i, i_1}^{-1}) \, \sigma_{j-1}^{-1}
\]
Since $i_h$ is at most $j-2$, $\sigma_{j-1}$ is free to slide inside of the conjugating blocks: we obtain
\begin{equation} \label{over53}
     \sigma_{j-1} \, b_{i, j-1} \, \sigma_{j-1}^{-1} \, = \, b_{ij}.
\end{equation}

\bigbreak
\noindent {\it Relations 5.4):} 
$\sigma_{j-1}^{-1} \, b_{i, j-1}^{ \mathcal{I}} \, \sigma_{j-1} = b_{i j}^{\mathcal{I} \cup (j-1)}$ for $ 1 \leq i<j \leq n-1, \, |i-j|>1$, where $\mathcal{I} =  (i_1, \dots , i_h) \subseteq \{ i+1, \dots , j-2 \}$ is an ordered set. The modifications read
$$
\sigma_{j-1}^{-1} (A_{i, i_1} \cdots A_{i, i_h}) \, b_{i, j-1} \, (A_{i, i_h}^{-1} \cdots A_{i, i_1}^{-1})\sigma_{j-1} = (A_{i, i_1} \cdots A_{i, i_h}  A_{i, j-1}) \, b_{i j} \, ( A_{i, j-1}^{-1} \, A_{i, i_h}^{-1} \cdots A_{i, i_1}^{-1})
$$
Remember that $i_h$ is at most $j-2$. In the left-hand side, move $\sigma_{j-1}$ and its inverse closer to the $b_{i, j-1}$ and cancel the conjugation terms in common on both sides to remain with:
$$ \sigma_{j-1}^{-1} \, b_{i, j-1} \sigma_{j-1} \, = \, A_{i, j-1} \, b_{i j} \, A_{i, j-1}^{-1} \iff b_{i, j-1} = \sigma_{j-1} A_{i, j-1} \, b_{i j} \, A_{i, j-1}^{-1} \sigma_{j-1}^{-1}$$

Focus on the right-hand side. It can be expanded in terms of the elementary braid generators

\[
\begin{split}
    \sigma_{j-1} \underline{A_{i, j-1}} \, b_{i j} \, & \underline{A_{i, j-1}^{-1}} \sigma_{j-1}^{-1}  = \sigma_{j-1} (\sigma_i^{-1} \cdots \sigma_{j-3}^{-1} \sigma_{j-2}^2 \sigma_{j-3} \cdots \underline{\sigma_i) \, b_{i j} \, (\sigma_i^{-1}} \cdots \sigma_{j-3}^{-1} \sigma_{j-2}^{-2} \sigma_{j-3} \cdots \sigma_i) \sigma_{j-1}^{-1} \\
    & \stackrel{Rel. \, 5.1)}{=} \sigma_{j-1} (\sigma_i^{-1} \cdots \sigma_{j-3}^{-1} \sigma_{j-2}^2 \sigma_{j-3} \cdots \underline{\sigma_{i+1}) \, b_{i+1, j} \, (\sigma_{i+1}^{-1}} \cdots \sigma_{j-3}^{-1} \sigma_{j-2}^{-2} \sigma_{j-3} \cdots \sigma_i) \sigma_{j-1}^{-1} \\
    & \hspace{6mm} \vdots
    \\
    & \stackrel{Rel. \, 5.1)}{=} \sigma_{j-1} (\sigma_i^{-1} \cdots \sigma_{j-3}^{-1} \underline{\sigma_{j-2}^2) \, b_{j-2, j} \, (\sigma_{j-2}^{-2}} \sigma_{j-3} \cdots \sigma_i) \sigma_{j-1}^{-1} \\
    & \stackrel{Rel. \, 5.1)}{=} \underline{\sigma_{j-1}} (\sigma_i^{-1} \cdots \sigma_{j-3}^{-1} \sigma_{j-2}) \, b_{j-1, j} \, (\sigma_{j-2}^{-1} \sigma_{j-3} \cdots \sigma_i) \underline{\sigma_{j-1}^{-1}} \\
    & \stackrel{commutation}{=}  (\sigma_i^{-1}\cdots \sigma_{j-3}^{-1} \underline{\sigma_{j-1} \sigma_{j-2}) \, b_{j-1, j}} \, (\sigma_{j-2}^{-1} \sigma_{j-1}^{-1} \sigma_{j-3} \cdots \sigma_i)  \\
    \\
    & \stackrel{Rel. \, 9.1)}{=}  \sigma_i^{-1} \cdots \sigma_{j-3}^{-1} \, b_{j-2, j-1} \, \underline{\sigma_{j-1} \sigma_{j-2}} \underline{\sigma_{j-2}^{-1} \sigma_{j-1}^{-1}} \sigma_{j-3} \cdots \sigma_i  \\
    & \stackrel{cancel}{=} \sigma_i^{-1} \cdots \underline{\sigma_{j-3}^{-1}  \, b_{j-2, j-1} \, \sigma_{j-3}} \cdots \sigma_i \\
    & \hspace{6mm} \vdots
    \\ & \stackrel{Rel. \, 5.1)}{=} b_{i, j-1}
\end{split}
\]
giving a tautology with the left-hand side.

\bigbreak
\noindent {\it Relations 7):}  $\sigma_i b_{rs}^{\mathcal{R}} \, = \,  b_{rs}^{\mathcal{R}} \sigma_i$ for $i<r-1 $ or $ i > s$, $1 \leq i \leq n-1$, $\mathcal{R} = (r_1, \dots, r_h) \subseteq \left\{ r+1, \dots , s-1\right\}$: the modified relations read
$$\sigma_i \, A_{r, r_1} \cdots A_{r, r_h} b_{rs} A_{r, r_h}^{-1} \cdots A_{r, r_1}^{-1} = A_{r, r_1} \cdots A_{r, r_h} b_{rs} A_{r, r_h}^{-1} \cdots A_{r, r_1}^{-1} \, \sigma_i$$

Given that $i$ sits outside of the interval $[ r, s ]$, $\sigma_i$ can slide next to $b_{rs}$, and simplifying the two hand sides the residual relations read
\begin{equation} \label{over7}
    \sigma_i \, b_{rs} = b_{rs} \, \sigma_i.
\end{equation}

\bigbreak
\noindent {\it Relations 8.1):}   $\sigma_i b_{rs}^{\mathcal{R}} \, = \, b_{rs}^{\mathcal{R}} \sigma_i$ for $1 \leq r<i<s-1 \leq n$, with either $i, i+1 \notin \mathcal{R} = (r_1, \dots, r_h) \subseteq \left\{ r+1, \dots , s-1\right\}$ or $i, i+1 \in \cR$. The relations are modified into:
$$
\sigma_i (A_{r, r_1} \cdots A_{r, r_h}) \, b_{rs} \, (A_{r, r_h}^{-1} \cdots A_{r, r_1}^{-1}) = (A_{r, r_1} \cdots A_{r, r_h}) \, b_{rs} \, (A_{r, r_h}^{-1} \cdots A_{r, r_1}^{-1}) \sigma_i. $$

If $i, i+1 \in \cR$, then in the conjugations appears the pair $A_{ri} A_{r, i+1}$, with whom $\sigma_i$ commutes. The generator $\sigma_i$ can slide on both hand sides adjacent to the bond.

Otherwise, for $i, i+1 \notin \mathcal{R}$, either $A_{r, r_j}^{\pm 1}$ slides at the right or at the left of $\sigma_i$ without obstructions, so that the latter can be brought closer to the bond. In both cases the relation then reduces to:
\begin{equation} \label{over81}
    \sigma_i \, b_{rs} = b_{rs} \, \sigma_i.
\end{equation}

\bigbreak
\noindent {\it Relations 8.2):}  $\sigma_i \, b_{rs}^{\mathcal{R}} = b^{\mathcal{R'}}_{rs} \sigma_i$ for $1 \leq r < i < s-1 < n, \  \mathcal{R},\mathcal{R^\prime} \subseteq \{ r+1, \dots , s-1 \}$ equal except $i \in \cR , \, i +1 \notin \cR \   \& \ i \notin \cR', i+1 \in \cR'$. 
The relations are modified into:
\[
\begin{split}
    \underline{\sigma_i} (A_{r, r_1}  \cdots A_{r, r'} \, \underline{A_{r, i}} \, A_{r, r''} &\cdots A_{r, r_h}) \, b_{rs} \,  (A_{r, r_h}^{-1} \cdots  A_{r, i}^{-1} \cdots A_{r, r_1}^{-1}) \\ & = (A_{r, r_1} \cdots  A_{r, i+1} \cdots A_{r, r_h}) \, b_{rs} \, (A_{r, r_h}^{-1} \cdots A_{r, r''}^{-1} \, \underline{A_{r, i+1}^{-1}} \, A_{r, r'}^{-1} \, \cdots A_{r, r_1}^{-1}) \underline{\sigma_i}
\end{split}
\]

In the braid group we have that, from relation $M_3)$ of Proposition~\ref{redBn}, $\sigma_i A_{r, i} = A_{r, i+1} \sigma_i$. Then $\sigma_i$ and its inverse can slide closer to $b_{rs}$. They pass the outer blocks $(A_{r, r_1} \dots A_{r, r'})^{\pm 1}$ with whom they commute because the indices are separated, then they interact with $A_{r, i}$ and $A_{r, i+1}^{-1}$, and in conclusion they pass the innermost blocks $(A_{r, r''}\cdots A_{r, r_h})^{\pm 1}$.
The relation becomes:
\begin{equation} \label{over82}
    \begin{split}
        (A_{r, r_1} \cdots A_{r, i+1} \cdots A_{r, r_h}) \sigma_i \, &b_{rs} \, (A_{r, r_h}^{-1} \cdots  A_{r, i}^{-1} \cdots A_{r, r_1}^{-1}) \\  = (A_{r, r_1} \cdots  A_{r, i+1} \cdots A_{r, r_h}) & \, b_{rs}\, \sigma_i (A_{r, r_h}^{-1} \cdots A_{r, i}^{-1} \cdots A_{r, r_1}^{-1}) \\
        & \iff \sigma_i \, b_{rs}  = b_{rs} \, \sigma_i.
    \end{split}
\end{equation}

For concluding the proof of the Theorem combine the relations unaffected by threading the underpassing bonds with Equations ~\ref{over3}, ~\ref{over4}, ~\ref{over51}, ~\ref{over53}, ~\ref{over7}, ~\ref{over81}, and ~\ref{over82}.  
\end{proof}


\section{Homomorphisms of the Bonded Braid Monoid to the Artin Braid Group} \label{Representations}

In this section we construct a homomorphism $\Phi$ of the bonded braid monoid $BB_n$ to the classical braid group $B_n$, which is surjective but not injective. We will then generalise it to a family of homomorphisms $\Phi_p$, $p \in \mathbb{Z}$. This is, in particular, an algebraic analogous of the tangle insertion, which is a topological tool used in constructing invariants in various applications. 
 The underlying idea is that bonds preserve the induced permutation of the braid strands. A tangle insertion of a bond then calls for a local pure braid. We now state the following result, which is one of the main targets of this work:

\begin{theorem} \label{mainthm} The map $\Phi \colon BB_n \to B_n$ that maps
\[
\begin{array}{lcl}
     \sigma_k^{\pm 1} & \mapsto  & \sigma_k^{\pm 1} \vspace{1.5mm}\\
     b_{ij} & \mapsto  & A_{ij} \vspace{1.5mm} \\
     b_{ij}^{(l_1,\dots , l_h)} & \mapsto  & A_{i,l_1} \cdots A_{i,l_h} A_{ij} A_{i,l_h}^{-1} \cdots A_{i,l_1}^{-1}
\end{array}
\]
with $A_{im}$ the right pure braid generator for $1 \leq i < m \leq n$, induces a surjective but not injective homomorphism of monoids. 
\end{theorem}

\begin{proof} On a word $w = w_1 \cdots w_n \in BB_n$, $\Phi$ is defined as $\Phi(w) = \Phi(w_1) \cdots \Phi(w_n)$. By Proposition~\ref{oversuff} and Theorem~\ref{oBB} it suffices to prove the validity of the Theorem for the mapping of the $\sigma_k$'s and the uniform overpass generators  $b_{ij}$. Then, by virtue of Eq.~\ref{undertoover}, the Theorem will be also valid for the mapping of the generic bonds $b_{ij}^{(l_1, \dots , l_h)}$. 

\smallbreak
\noindent \textit{Relations  1), 2)} of Theorem~\ref{oBB} are the usual braid  relations. We shall verify that relations 3) to 9) of Theorem~\ref{oBB} are still satisfied when $b_{ij}$ is replaced by $A_{ij}$.  For this purpose it is convenient to use for the braid group $B_n$ the redundant presentation given in Proposition~\ref{redBn}, with generators $\sigma_k$ and $A_{ij}$. 
In fact, as we shall see, six relations of Theorem~\ref{oBB} are mapped to six distinct defining relations of the presentation in Proposition~\ref{redBn}. The other three are mapped in particular cases of the remaining ones. 

\bigbreak
\noindent \textit{Relations 3):} $b_{ij} \, b_{rs} \, = \, b_{rs} \, b_{ij} \, \mapsto \, A_{ij} A_{rs} = A_{rs}A_{ij}$ : both the cases  $i<j<r<s$ and $i<r<s<j$ are included in relations $P_1)$ of Proposition~\ref{redBn}.

\bigbreak
\noindent \textit{Relations 4):} $\sigma_i^{-1} \, b_{i+1, j} \, \sigma_{i} \, = \, b_{ij} \,  \mapsto \, \sigma_i^{-1} A_{i+1, j} \sigma_i = A_{ij}$ : rescaling $i \to i-1$ the relations become relations $M_1)$ of Proposition~\ref{redBn}.

\bigbreak
\noindent \textit{Relations 5):} $\sigma_{j-1} \, b_{i, j-1} \, \sigma_{j-1}^{-1} = b_{ij} \, \mapsto \, \sigma_{j-1} A_{i, j-1} \sigma_{j-1}^{-1} = A_{ij}$  : moving the $\sigma_{j-1}$ and $\sigma_{j-1}^{-1}$ to the other hand side the relations become relations $M_3)$ of Proposition~\ref{redBn}.

\bigbreak
\noindent \textit{Relations 6):} $\sigma_i \, b_i \, = \, b_i \, \sigma_i \, \mapsto \, \sigma_i A_{i, i+1} = A_{i, i+1} \sigma_i$ : clearly valid by relations $\Sigma_3)$ of Proposition~\ref{redBn}.

\bigbreak
\noindent \textit{Relations 7):} $\sigma_i \, b_{rs} \, = \, b_{rs} \, \sigma_i \mapsto \sigma_i A_{rs} = A_{rs} \sigma_i$ : these are relations $M_5)$ of Proposition~\ref{redBn}.

\bigbreak
\noindent \textit{Relations 8):} $b_i \sigma_{i+1} \sigma_i \, = \, \sigma_{i+1}\sigma_i b_{i+1} \, \mapsto \, A_{i, i+1} \, \sigma_{i+1} \sigma_i \, = \, \sigma_{i+1}\sigma_i \, A_{i+1, i+2}$ : rewrite it as: \\ 
$\sigma_{i+1}^{-1} A_{i, i+1} \sigma_{i+1}  \, = \, \sigma_i A_{i+1, i+2} \sigma_i^{-1}$. Insert a pair of $\sigma_i \sigma_i^{-1}$ in the right-hand side as follows and proceed:
\[ 
    \sigma_i \, A_{i+1, i+2} \, \sigma_i^{-1} = \underline{\sigma_i (\sigma_i} \, \underline{\sigma_i^{-1}) A_{i+1, i+2} (\sigma_i} \, \underline{\sigma_i^{-1}) \sigma_i^{-1}} = A_{i, i+1} A_{i, i+2} A_{i, i+1}^{-1}
 \]
where in the last equality we have used relations $\Sigma_3)$ of Proposition~\ref{redBn} on $\sigma_i^{\pm 2}$ and  relations $M_1)$ on $\sigma_i^{-1} \, A_{i+1, i+2} \, \sigma_i$. Recomposing, we get
$$\sigma_{i+1}^{-1} A_{i, i+1} \sigma_{i+1} = A_{i, i+1} A_{i, i+2} A_{i, i+1}^{-1}$$
which is relations $M_4)$ of Proposition~\ref{redBn} in the case of $j = i+1$.

\bigbreak
\noindent \textit{Relations 9):} $\sigma_i \sigma_{i+1} b_i \, = \, b_{i+1} \sigma_i \sigma_{i+1} \, \mapsto \, \sigma_i \sigma_{i+1} \, A_{i, i+1} \, = \, A_{i+1, i+2} \, \sigma_i \sigma_{i+1}$ : rewrite it as:\\
$ \sigma_{i+1} A_{i, i+1} \sigma_{i+1}^{-1} \, = \, \sigma_i^{-1} A_{i+1, i+2} \sigma_i $ and conjugate both hand sides by $\sigma_i^{-1}$:
$$ \sigma_i^{-1} \underline{\sigma_{i+1} A_{i, i+1} \sigma_{i+1}^{-1}} \sigma_i \, = \, \underline{\sigma_i^{-1} \sigma_i^{-1}} A_{i+1, i+2} \underline{\sigma_i \sigma_i} \iff \sigma_i^{-1} A_{i, i+2} \sigma_i = A_{i, i+1}^{-1} A_{i+1, i+2} A_{i, i+1}$$
where for obtaining the last equality we have used a variation of  relations $M_3)$ of Proposition~\ref{redBn}   on the left-hand side and relations $\Sigma_3)$ on the right-hand side. This last equation is relations $M_2)$ of Proposition~\ref{redBn} when $j = i+2$.

\smallbreak
The previous part of the proof shows that the images via $\Phi$ of the relations in the presentation of $BB_n$ in Theorem~\ref{oBB} are relations in $B_n$. 
This proves that $\Phi$ is indeed a homomorphism of monoids. Moreover, since $\Phi$ restricted on the braid generators is the identity map, we get that $\Phi$ is surjective. Furthermore, $\Phi$ is not injective, since $b_{ij}$ and $A_{ij}$ are mapped to the same element.
\end{proof}

\begin{remarks} \label{cyclicnotbond}
There are relations in $B_n$ which do not come from relations in $BB_n$ involving exclusively bonds.
For example, the cyclic relations $P_2)$ and $P_3)$ of Theorem~\ref{redBn} (recall illustrations in Figure~\ref{cyclicrels}) are not included in the image set of the relations of the bonded braid monoid involving exclusively bonds. However, allowing a mix of bonds and pure braid generators, this can be the case. For example, $P_2)$ is the image under $\Phi$ of the following relation in $BB_n$, illustrated in Figure~\ref{kernelphi}:
$$
A_{ir} A_{is} \, b_{rs} \, = \, b_{rs} \, A_{ir} A_{is}
$$ 
 In the same spirit relations $P_4)$ are in the image set under $\Phi$ of relations in Theorem~\ref{oBB}.
Consider, for example, Eq.~\ref{bondslide} and apply first the threading and then the map $\Phi$:
$$b_{ij} \, b_{rs}^{(j)} \, = \, b_{rs}^{(j)} \, b_{ij} \quad \mapsto \quad b_{ij} \, A_{rs} \, b_{rs} \, A_{rj}^{-1} = A_{rs} \, b_{rs} \, A_{rj}^{-1} \, b_{ij} \quad \stackrel{\Phi}{\mapsto} \quad A_{ij} \, A_{rs} \, A_{rs} \, A_{rj}^{-1} = A_{rs} \, A_{rs} \, A_{rj}^{-1} \, A_{ij}$$

\begin{figure}[H]
\begin{center} 
    \includegraphics[width=11cm]{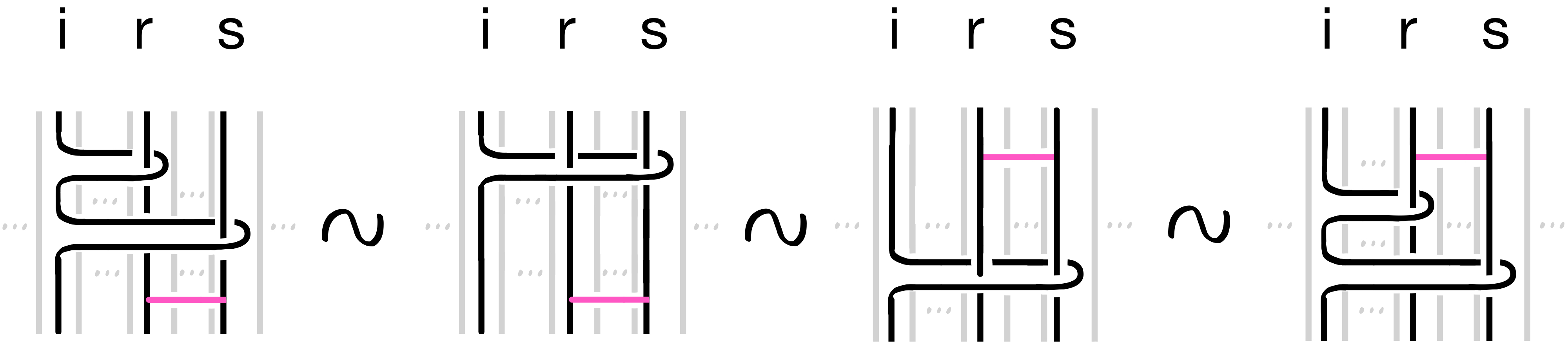}
\end{center}
\caption{ The relation $A_{ir} \, A_{is} \, b_{rs} \, = \, b_{rs} \, A_{ir} \, A_{is}$: the bond $b_{rs}$ passes through the loop formed by $A_{ir} \, A_{is}$}
\label{kernelphi}
\end{figure}

The above remark leads to considerations about the kernel  $\mathrm{ker} \, \Phi$ of $\Phi$.  $\mathrm{ker} \, \Phi$ contains elements in the \textit{pure bonded braid monoid}, which is the submonoid of all bonded braids inducing the trivial permutation. For example, all expressions of the form $b_{ij}\,A_{ij}^{-1} \in \mathrm{ker} \, \Phi$.

Along the same thread, it is worth pointing out that the element $A_{ij}^{-1}$ can only appear identically in the image set of $\Phi$. Enhancing a bond with a
formal inverse leads to extending the bonded braid monoid to the \textit{bonded braid group} $EBB_n$ \cite{DKL}. Then, in the extension of $\Phi$ to the bonded braid group, $A_{ij}^{-1}$ can also be the image of $b_{ij}^{-1}$.

The main differences between pure braid generators and bonds are: first that the former have interesting three-elements interactions (in the form of the cyclic relations) which bonds do not have. Second, in $B_n$ pure braid generators admit expressions in terms of elementary braid generators, as $A_{i, i+1} = \sigma_i^2$. This is not true for bonds, for which the \textit{flype move} holds instead (relation 6) in Definition~\ref{BB}). As noted in the previous Theorem's proof, the image under $\Phi$ of the flype move holds as a relation in $B_n$ because we have $A_{i, i+1} = \sigma_i^2$ (relation $\Sigma_3$ in Proposition~\ref{redBn}).

The above considerations will be the subject of study of sequel work. 

\end{remarks}

\begin{remark} \label{alternativemainthm}
An alternative proof of Theorem~\ref{mainthm} consists in applying $\Phi$ directly on the full set of generators and relations of $BB_n$ in the presentation of Definition~\ref{BB}. The proof would work as in Theorem~\ref{oBB}: relations are first mapped via $\Phi$ and then threaded to the familiar forms of relations in Proposition~\ref{redBn}.
Following this approach, some computations in the application of the threading algorithm are shortened (but the majority would get much more complicate).

Take for example relations 4) in Definition~\ref{BB}. In Step II) of the general case $r, s \in \cI$  of the proof of Theorem~\ref{oBB}, the vertex slide moves (relations $5.1)-5.4)$) are used to prove the equality $(A_{ir} A_{il}^{-1} A_{ir}^{-1}) \, b_{rs} = b_{rs} \, (A_{ir} A_{il}^{-1} A_{ir}^{-1})$. The vertex slide moves involve elementary braid generators. However, by first mapping relations 4) via $\Phi$ and then threading, the above equality becomes $(A_{ir} A_{il}^{-1} A_{ir}^{-1}) \, A_{rs} = A_{rs} \, (A_{ir} A_{il}^{-1} A_{ir}^{-1})$, which corresponds to relations $P_4)$ of Proposition~\ref{redBn}. 
Therefore, the verification that the image via $\Phi$ of relation 4) in Definition~\ref{BB} holds as a relation in $B_n$ can be computed within the subgroup $P_n$ without interacting with the elementary braid generators. 
Something similar is true for Eq.~\ref{bondslidegeneral}. Map it first to $B_n$ via $\Phi$, then thread. The relation is mapped to: 
\[
\begin{split}
    A_{i, i_1} \dots A_{i, ih} A_{ij} A_{i, i_h}^{-1} & \dots A_{i, i_1}^{-1} \, A_{r, r_1} \dots A_{rj} \dots A_{r, r_l} A_{rs} A_{r, r_l}^{-1} \dots A_{rj}^{-1} \dots A_{r, r_1}^{-1} \\ & = A_{r, r_1} \dots A_{rj} \dots A_{r, r_l} A_{rs} A_{r, r_l}^{-1} \dots A_{rj}^{-1} \dots A_{r, r_1}^{-1} \, A_{i, i_1} \dots A_{i, ih} A_{ij} A_{i, i_h}^{-1} \dots A_{i, i_1}^{-1}
\end{split}
\]
where $\cR =(r_1, \dots , j , \dots r_l) \subseteq \cI = (i_1 , \dots, i_h)$ in the interval $\{ r+1 , \dots , j-1 \}$. With a bit of work it is possible to prove that this relation is equivalent to Eq.~\ref{sliderelA} only using pure braid relations. 
Similarly, the bond threading of relations 3) in Theorem~\ref{oBB} did not make use of computations involving elementary braid generators.  As a consequence, by first mapping via $\Phi$ relations 3), all computations of the threading in the image are carried out in $P_n$.
\end{remark}

Relations 3), 4) in Definition~\ref{BB}, and Eq.~\ref{bondslidegeneral} and its symmetric, Eq.~\ref{bondslidegeneral2} showcase every possible bond-bond interaction. We can then define:

\begin{definition} 
The \textit{submonoid of bonds} in $BB_n$, denoted $bB_n$ is the monoid generated by all $b_{ij}^{\cI}$, for $\mathcal{I} = (l_1, \dots , l_k)$, $i < l_1 < \dots < l_k < j$, subject to Relations 3), 4) in Definition~\ref{BB}, and Eq.~\ref{bondslidegeneral} and its symmetric, Eq.~\ref{bondslidegeneral2}. Furthermore, $bB_n$ extends to the \textit{subgroup  of bonds} in the bonded braid group $EBB_n$, denoted $EbB_n$. 
\end{definition}

Therefore, we have obtained the following: 
\begin{corollary} 
The restriction of $\Phi$ to the submonoid of bonds $bB_n$ in $BB_n$ is a homomorphism of monoids to the pure braid group $P_n$: 
$$
\Phi : bB_n \mapsto P_n
$$
This is not an injective morphism since for $i < r < s$ the words $b_{is}^{(r)} \, b_{ir}$ and $b_{ir} \, b_{is}$ have both as image $A_{ir} \, A_{is}$. Moreover, as the inverses of pure braid generators are not included in the image of bonds, it is also not surjective. Yet, it becomes surjective from the group of bonds $EbB_n$.
\end{corollary} 

We conclude this section with a generalisation  of Theorem~\ref{mainthm}.

\begin{theorem} \label{mainthm2} The map $\Phi_p \colon BB_n \to B_n$, $p \in \mathbb{Z}$, that sends
\[
\begin{array}{lcl}
\sigma_k^{\pm 1} & \mapsto & \sigma_k^{\pm 1} \vspace{1.7mm}   \\ 

b_{ij} & \mapsto  & A_{ij}^p \vspace{1.5mm}  \\ 

b_{ij}^{(l_1,\cdots , l_k)} & \mapsto  & A_{i,l_1} \cdots A_{i,l_k} \, A_{ij}^p \, A_{i,l_k}^{-1} \cdots A_{i,l_1}^{-1}, \quad i<l_1<\dots<l_k<j
\end{array}
\]
with $A_{ij}^p$ the $p$-th power of the right pure braid generator $A_{ij}$, defines a surjective but not injective homomorphism of monoids for all $p \in \mathbb{Z}$. For $p=0$, the map $\Phi_0$ is the forgetful map trivializing all bonds.
\end{theorem}

\begin{proof}
As in Theorem~\ref{mainthm}, we shall define $\Phi_p$ on a word $w = w_1 \cdots w_n \in BB_n$ as $\Phi_p(w) = \Phi_p(w_1) \cdots \Phi_p(w_n)$. The image of a generic bond bonds $b_{ij}^{\cI}$ is compatible with Eq.~\ref{undertoover}, so by Proposition~\ref{oversuff} and Theorem~\ref{oBB} it suffices to prove the Theorem based on the images of the $\sigma_k$ and the overpassing bonds. Similarly to Theorem~\ref{mainthm}, we shall verify the statement only on relations $3)$ to $9)$ of Theorem~\ref{oBB}.

    \bigbreak
    \noindent \textit{Relations 3):} $b_{ij} \, b_{rs} \, = \, b_{rs} \, b_{ij} \, \mapsto \, A_{ij}^p \, A_{rs}^p = A_{rs}^p \, A_{ij}^p$ : using relations $P_1)$ of Proposition~\ref{redBn}, the relation can be rewritten as $(A_{ij} \, A_{rs})^p = (A_{rs} \, A_{ij})^p$, which follows from relations $P_1)$ in both  cases $i<j<r<s$ and $i<r<s<j$.

    \bigbreak
\noindent \textit{Relations 4):} $\sigma_i^{-1} \, b_{i+1, j} \, \sigma_{i} \, = \, b_{ij} \,  \mapsto \, \sigma_i^{-1} A_{i+1, j}^p \sigma_i = A_{ij}^p$ : placing in between every two factors in the left-hand side the pair $\sigma_i^{-1} \sigma_i$ the relations are proven true by relations $M1)$ of Proposition~\ref{redBn}.

\bigbreak
\noindent \textit{Relations 5):} $\sigma_{j-1} \, b_{i, j-1} \, \sigma_{j-1}^{-1} = b_{ij} \, \mapsto \, \sigma_{j-1} A_{i, j-1}^p \sigma_{j-1}^{-1} = A_{ij}^p$ : placing in between every two factors in the right-hand side the pair $\sigma_{j-1} \sigma_{j-1}^{-1}$ and moving the $\sigma_{j-1}$ and $\sigma_{j-1}^{-1}$ from the left to the right-hand side, the relations are proven true by relations $M_3)$ of Proposition~\ref{redBn}.

\bigbreak
\noindent \textit{Relations 6):} $\sigma_i \, b_i \, = \, b_i \, \sigma_i \, \mapsto \, \sigma_i A_{i, i+1}^p = A_{i, i+1}^p \sigma_i$ : clearly valid by relations $\Sigma_3)$ of Proposition~\ref{redBn}.

\bigbreak
\noindent \textit{Relations 7):} $\sigma_i \, b_{rs} \, = \, b_{rs} \, \sigma_i \mapsto \sigma_i A_{rs}^p = A_{rs}^p \sigma_i$ : they follow from relations $M_5)$ of Proposition~\ref{redBn}.

\bigbreak
\noindent \textit{Relations 8):} $b_i \sigma_{i+1} \sigma_i \, = \, \sigma_{i+1}\sigma_i b_{i+1} \, \mapsto \, A_{i, i+1}^p \, \sigma_{i+1} \sigma_i \, = \, \sigma_{i+1}\sigma_i \, A_{i+1, i+2}^p$ : rewrite it as \\ 
$\sigma_{i+1}^{-1} A_{i, i+1}^p \sigma_{i+1}  \, = \, \sigma_i A_{i+1, i+2}^p \sigma_i^{-1}$. Insert $p+1$ pairs of $\sigma_i \sigma_i^{-1}$ in the right-hand side as below and proceed:
\[ 
    \sigma_i \, A_{i+1, i+2}^p \, \sigma_i^{-1} = \underline{\sigma_i (\sigma_i} \, \underline{\sigma_i^{-1}) A_{i+1, i+2} (\sigma_i} \, \underline{\sigma_i^{-1}) A_{i+1, i+2} (\sigma_i} \underline{\sigma_i^{-1} ) } \cdots \underline{ A_{i+1, i+2} (\sigma_i} \underline{\sigma_i^{-1} ) \sigma_i^{-1}} = A_{i, i+1} A_{i, i+2}^p A_{i, i+1}^{-1}
 \]
where in the last equality we have used relations $\Sigma_3)$ of Proposition~\ref{redBn} on $\sigma_i^{\pm 2}$ and relations $M_1)$ $p$ times on all the terms $\sigma_i^{-1} \, A_{i+1, i+2} \, \sigma_i$. Recomposing, we get
$$\sigma_{i+1}^{-1} A_{i, i+1}^p \sigma_{i+1} = A_{i, i+1} A_{i, i+2}^p A_{i, i+1}^{-1} \iff (\sigma_{i+1}^{-1} A_{i, i+1} \sigma_{i+1})^p = (A_{i, i+1} A_{i, i+2} A_{i, i+1}^{-1})^p$$
where the last equation is true from relations $M_4)$ of Proposition~\ref{redBn} in the case of $j = i+1$.

\bigbreak
\noindent \textit{Relations 9):} $\sigma_i \sigma_{i+1} b_i \, = \, b_{i+1} \sigma_i \sigma_{i+1} \, \mapsto \, \sigma_i \sigma_{i+1} \, A_{i, i+1}^p \, = \, A_{i+1, i+2}^p \, \sigma_i \sigma_{i+1}$ : rewrite it as:\\
$ \sigma_{i+1} A_{i, i+1}^p \sigma_{i+1}^{-1} \, = \, \sigma_i^{-1} A_{i+1, i+2}^p \sigma_i $ and conjugate both hand sides by $\sigma_i^{-1}$:
\[  \begin{split}
    \sigma_i^{-1} \underline{\sigma_{i+1} A_{i, i+1}^p \sigma_{i+1}^{-1}} \sigma_i \, = \, \underline{\sigma_i^{-1} \sigma_i^{-1}} A_{i+1, i+2}^p \underline{\sigma_i \sigma_i} & \iff \sigma_i^{-1} A_{i, i+2}^p \sigma_i = A_{i, i+1}^{-1} A_{i, i+2}^p A_{i, i+1} \\ & \iff (\sigma_i^{-1} A_{i, i+2} \sigma_i)^p = (A_{i, i+1}^{-1} A_{i, i+2} A_{i, i+1})^p
\end{split}\]
where for obtaining the last equality we have used a variation of relations  $M_3)$ of Proposition~\ref{redBn} for $p$ times on the left-hand side and $\Sigma_3)$ on the right-hand side. The last equation comes from relations $M_2)$ when $j = i+2$.

\smallbreak
The previous part of the proof shows that the images via $\Phi_p$ of the relations in the bonded braid monoid presentation in Theorem~\ref{oBB} are true by relations in $B_n$. 
This proves that $\Phi_p$ is indeed a homomorphism of monoids. Moreover, since $\Phi_p$ restricted on the braid generators is the identity map, we get that $\Phi_p$ is surjective. Furthermore, $\Phi_p$ is not injective, since $b_{ij}$ and $A_{ij}^p$ are mapped to the same element for $p\neq 0$. For $p=0$ the homomorphism $\Phi_0$ is the forgetful map sending all bonds to the identity element.
\end{proof} 

\begin{remark}
    Theorems~\ref{mainthm} and~\ref{mainthm2} hold analogously for using the left pure braid generators and Eq.~\ref{undertooverleft}. In this case the map $\Phi_p^l \colon BB_n \to B_n$, $p \in \mathbb{Z}$,  sends
\[
\begin{array}{lcl}
\sigma_k^{\pm 1} & \mapsto & \sigma_k^{\pm 1} \vspace{1.7mm}   \\ 

b_{ij} & \mapsto  & A_{ij}^p \vspace{1.5mm}  \\ 

b_{ij}^{(l_1,\cdots , l_k)} & \mapsto  &  A_{l_k, j}^{-1} \cdots A_{l_1, j}^{-1} \, A_{ij}^p \, A_{l_1, j} \cdots A_{l_k, j} , \quad i<l_1<\dots<l_k<j
\end{array}
\]
with $A_{lm}$ the left pure braid generator of indices $1 \leq l < m \leq n$, defines a surjective but not injective homomorphism of monoids for all $p \in \mathbb{Z}$.  
\end{remark}


\section{A new presentation for $P_n$ using the clasp generators} \label{symm}

The pure braid group on $n$ strands is generated by all elements $A_{ij} = \sigma_i^{-1} \sigma_{i+1}^{-1} \cdots \sigma_{j-2}^{-1} \sigma_{j-1}^{2} \sigma_{j-2} \cdots \sigma_i, \, 1 \leq i < j \leq n$,  which including all $A_i \coloneqq A_{i, i+1} = \sigma_i^{2}$, which we call \textit{clasps}. From now on we set 
 $$
 p_i \coloneqq \sigma_i^2
 $$ 
 so $p_i^{-1} = \sigma_i^{-2}$. The goal is to give a presentation for $P_n$ as a subgroup of $B_n$ with normal generators $p_i$ for all $i$. This presentation will be useful when relating to the (tight) bonded braid monoid.

\begin{figure}[H]
\begin{center} 
    \includegraphics[width=2cm]{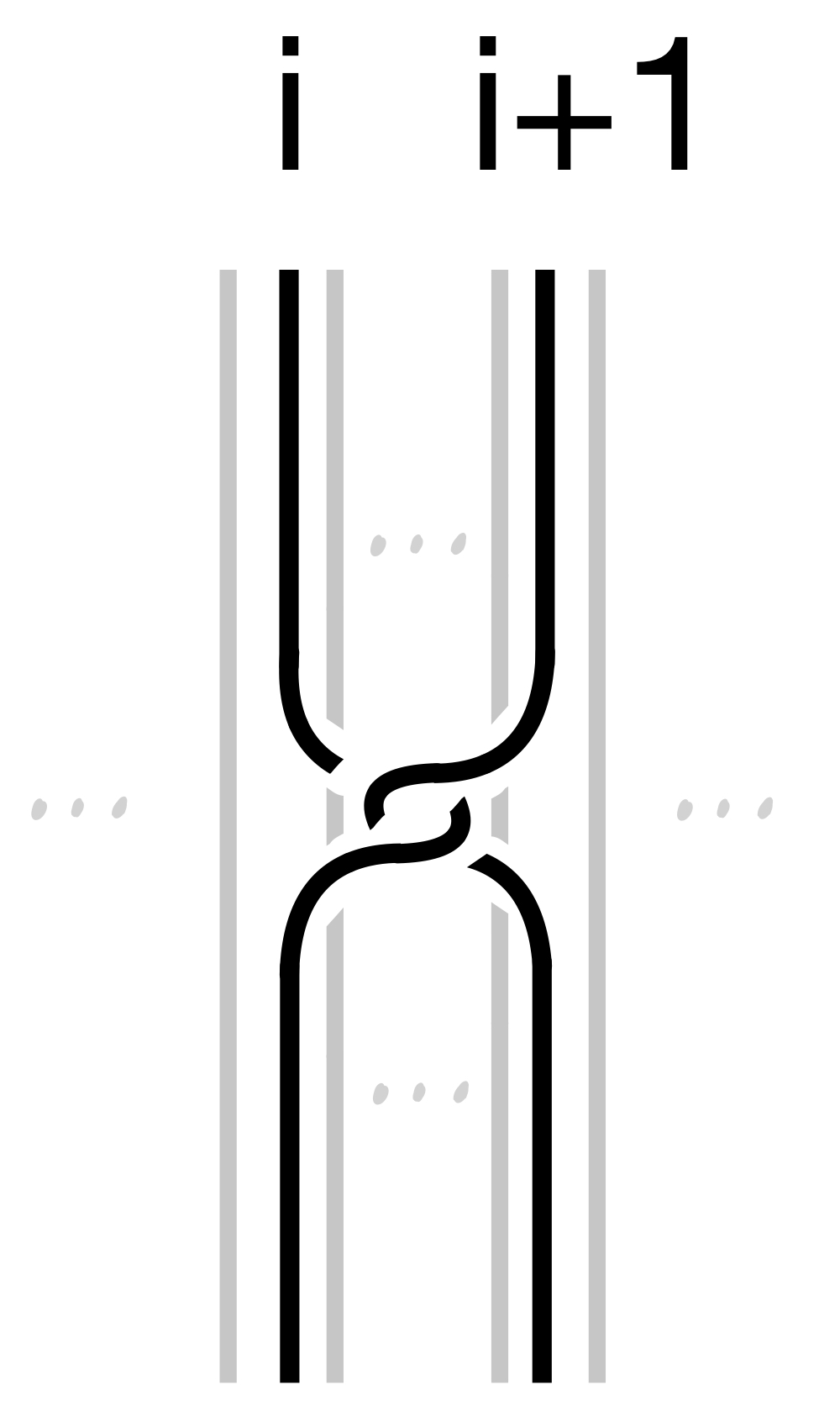}
\end{center}
\caption{The clasp generator $p_i$.}
\label{Pgen}
\end{figure}

\subsection{Rewriting pure braid generators in terms of clasps.}

The usual right loop generators $A_{ij}$ of Eq.~\ref{AG} can be obtained acting by conjugation with a string $\sigma_i \cdots \sigma_{j-2}$ on $p_{j-1}$, namely they can be rewritten as 
\begin{equation} \label{claspgen}
    A_{ij} = \sigma_i^{-1} \ldots \sigma_{j-2}^{-1} \, p_{j-1} \sigma_{j-2} \ldots \sigma_i  :=  p_{ij}  
\end{equation}
 This is not new, however writing them in terms of the $p_i$ helps in focusing on the clasp. 
 
 \smallbreak

The clasp generators $p_i$ and the classical generators $\sigma_i$,  for all $i$, form now an overabundant set for presenting $P_n$, since
\begin{equation} \label{Y}
    p_i \, \sigma_{i+1}^{-1} \sigma_i^{-1} = \sigma_i^2 \sigma_{i+1}^{-1} \sigma_i^{-1} = \sigma_i \sigma_{i+1}^{-1} \sigma_i^{-1} \sigma_{i+1} = \sigma_{i+1}^{-1} \sigma_i^{-1} \sigma_{i+1}^2 = \sigma_{i+1}^{-1} \sigma_i^{-1} p_{i+1}
\end{equation}
equivalently,
\begin{equation} \label{Prel1}
   \sigma_i^{-1} p_{i+1} \sigma_i = \sigma_{i+1} p_i \sigma_{i+1}^{-1} \ \Longleftrightarrow \ p_{i+1} \sigma_i \sigma_{i+1} = \sigma_i \sigma_{i+1} \, p_i \ \Longleftrightarrow \ p_{i+1}  = \sigma_i \sigma_{i+1} p_i \, \sigma_{i+1}^{-1} \sigma_i^{-1}.
\end{equation}
In particular, all generators $p_{k}$ can be obtained from suitable conjugation of $p_1$ (or $p_{n-1}$ or of any fixed $p_i$).
This means that one really needs only \textit{one} clasp generator. See Subsection~\ref{irredundant} further below.

\smallbreak

Using braid generators instead of their inverses in Eq.~\ref{Y} leads to the following alternative for Eq.~\ref{Prel1}:
\begin{equation} \label{Prel3}
    \sigma_i \, p_{i+1} \sigma_i^{-1} = \sigma_{i+1}^{-1}  \, p_i  \, \sigma_{i+1} \ \Longleftrightarrow \ \sigma_{i+1} \sigma_i  \, p_{i+1} = p_i  \,  \sigma_{i+1} \sigma_i  \ \Longleftrightarrow \  p_{i+1} = \sigma_i^{-1} \sigma_{i+1}^{-1} \, p_i \, \sigma_{i+1} \sigma_i. 
\end{equation}

All equations come directly from the fact that $p_k = \sigma_k^2$. 
From Eqs.~\ref{Prel1} and~\ref{Prel3} we obtain the following relations, involving the clasp generators and the classical braid group generators: 
\begin{equation} \label{Prel2}
    (\sigma_{i+1} \sigma_i) (\sigma_i \sigma_{i+1}) \, p_i = p_i \, (\sigma_{i+1} \sigma_i) (\sigma_i \sigma_{i+1}) \, , \quad
    (\sigma_i \sigma_{i+1})(\sigma_{i+1} \sigma_i) \, p_{i+1} = p_{i+1} \, (\sigma_i \sigma_{i+1})(\sigma_{i+1} \sigma_i) 
\end{equation}
Other useful relations are: 
\begin{equation} \label{Prel4}
    \sigma_i p_i \sigma_i^{-1} = p_i \quad \text{and} \quad p_i \sigma_i^{-1} = \sigma_i \quad \text{and} \quad \sigma_i^{-1} p_i  = \sigma_i \quad \text{and} \quad p_i \sigma_j = \sigma_j p_i \quad \text{for} \quad  |i-j| \geq 2.
\end{equation}
In order to obtain a presentation for $P_n \unlhd B_n$ using the clasp generators $p_i$ in Eq.~\ref{claspgen}, we will examine how the relations in Proposition~\ref{alternative} are affected when rewriting the $A_{ij}$'s as conjugates of clasp generators and apply Tietze transformations \cite{Tit}. We will use the interaction of clasps with the elementary braid generators $\sigma_i$, from Eq.~\ref{Prel4}.

\subsection{Rewriting Relations 1) in Proposition~\ref{alternative} using the clasp generators.}

Relations 1) of Proposition~\ref{alternative}, $A_{ij} \, A_{rs} = A_{rs} \, A_{ij}$ \ for \ $i<r<s<j$ \ and \ $i<j<r<s$, comprise two subcases, which should be treated separately.

\noindent \textit{Case 1:} $i<j<r<s$. 
\begin{equation*} 
\begin{split}
    (\sigma_i^{-1} & \cdots \sigma_{j-2}^{-1} \, p_{j-1} \,\underline{\sigma_{j-2} \cdots \sigma_i}) (\underline{ \underline{\sigma_r^{-1}} \cdots \sigma_{s-2}^{-1}} \, p_{s-1} \, \sigma_{s-2} \cdots \sigma_r)\\ & = (\sigma_r^{-1} \cdots \sigma_{s-2}^{-1} \, p_{s-1} \, \underline{\sigma_{s-2} \cdots \sigma_r})  (\underline{\sigma_i^{-1} \cdots \sigma_{j-2}^{-1}} \, p_{j-1} \, \sigma_{j-2} \cdots \sigma_i)
\end{split}
\end{equation*}
Since the two blocks between the clasp generators are separated indices-wise in either side of the equation, they commute freely. The equation can then be reorganised into:
\begin{equation*}
\begin{split}
    \sigma_i^{-1} \cdots & \sigma_{j-2}^{-1} \sigma_r^{-1} \cdots \sigma_{s-2}^{-1} \, p_{j-1} \, p_{s-1} \, \sigma_{j-2} \cdots \sigma_i \sigma_{s-2} \cdots \sigma_r\\ &= \sigma_i^{-1} \cdots \sigma_{j-2}^{-1} \sigma_r^{-1} \cdots \sigma_{s-2}^{-1} \, p_{s-1} \, p_{j-1} \, \sigma_{j-2} \cdots \sigma_i \sigma_{s-2} \cdots \sigma_r
\end{split}
\end{equation*}
Deleting the two equal blocks on the right and left of both hand sides, the relation becomes 
\begin{equation} \label{clasp11}
    p_{j-1} \, p_{s-1} \, = \, p_{s-1} \, p_{j-1}.
\end{equation}

\noindent \textit{Case 2:} $i<r<s<j$. 
\begin{equation} \label{altPrel1}
\begin{split}
    (\sigma_i^{-1} \cdots & \sigma_r^{-1} \cdots \sigma_{j-2}^{-1} \, p_{j-1} \,\underline{\sigma_{j-2} \cdots \sigma_r \cdots \sigma_i}) (\underline{ \underline{\sigma_r^{-1}} \cdots \sigma_{s-2}^{-1}} \, p_{s-1} \, \sigma_{s-2} \cdots \sigma_r)\\ & = (\sigma_r^{-1} \cdots \sigma_{s-2}^{-1} \, p_{s-1} \, \underline{\sigma_{s-2} \cdots \sigma_r})  (\underline{\sigma_i^{-1} \cdots \sigma_r^{-1} \cdots \sigma_{j-2}^{-1}} \, p_{j-1} \, \sigma_{j-2} \cdots \sigma_r \cdots \sigma_i)
\end{split}
\end{equation}

The situation is quite similar, but it is not granted that the two central braid blocks commute and can be pushed to the sides of the expressions. The first thing to do is to rewrite the two central blocks. In the left-hand side of Eq.~\ref{altPrel1}, pass $\sigma_r^{-1}$ to the left of $\sigma_{j-2} \cdots \sigma_r \cdots  \sigma_i$
\begin{equation*}
\begin{split}
    \sigma_{j-2} \cdots \sigma_i \underline{\sigma_r^{-1}} \cdots \sigma_{s-2}^{-1} & = \sigma_{j-2} \cdots \sigma_{r+1} \underline{\sigma_r \sigma_{r-1} \sigma_r^{-1}} \sigma_{r-2} \cdots \sigma_i \sigma_{r+1}^{-1} \cdots \sigma_{s-2}^{-1} \\
    & \stackrel{\text{2nd braid rel.}}{=} \sigma_{j-2} \cdots \sigma_{r+1} \underline{\sigma_{r-1}^{-1}} \sigma_{r} \sigma_{r-1} \sigma_{r-2}\cdots \sigma_i \sigma_{r+1}^{-1} \cdots \sigma_{s-2}^{-1} \\ 
    & \stackrel{\text{pass left}}{=} \sigma_{r-1}^{-1} \sigma_{j-2} \cdots \sigma_{r+1}  \sigma_{r} \sigma_{r-1} \sigma_{r-2}\cdots \sigma_i \sigma_{r+1}^{-1} \cdots \sigma_{s-2}^{-1} \\ 
    & \cdots \\
\end{split}
\end{equation*}
Reiterating the algorithm, one gets
\[
\sigma_{j-2} \cdots \sigma_i \sigma_r^{-1} \cdots \sigma_{s-2}^{-1} = \sigma_{r-1}^{-1} \cdots \sigma_{s-3}^{-1}  \sigma_{j-2} \cdots \sigma_i
\]
Analogously, for its inverse in the right-hand side one gets
\[
\sigma_{s-2} \cdots \sigma_r  \sigma_i^{-1} \cdots \sigma_{j-2}^{-1} = \sigma_i^{-1} \cdots \sigma_{j-2}^{-1} \sigma_{s-3} \cdots \sigma_{r-1}
\]
Equation~\ref{altPrel1} can, thus, be rewritten as: 
\begin{equation*}
\begin{split}
    \sigma_i^{-1} \cdots & \sigma_{j-2}^{-1} \, p_{j-1} \, \sigma_{r-1}^{-1} \cdots \sigma_{s-3}^{-1}  \sigma_{j-2} \cdots \sigma_i \, p_{s-1} \, \sigma_{s-2} \cdots \sigma_r\\ & = \sigma_{r}^{-1} \cdots \sigma_{s-2}^{-1} \, p_{s-1} \, \sigma_i^{-1} \cdots \sigma_{j-2}^{-1} \sigma_{s-3} \cdots \sigma_{r-1} \, p_{j-1} \, \sigma_{j-2}  \cdots \sigma_i
\end{split}
\end{equation*} 
In the left-hand side, the block $\sigma_{r-1}^{-1} \cdots \sigma_{s-3}^{-1}$ slides to the left of $p_{j-1}$, since the indices are far apart ($r<s<j$). However, because $i<s<j$, the process of passing $\sigma_{j-2} \cdots \sigma_i$ to the right of $p_{s-1}$ is a bit longer:
\[
\begin{split}
    (\sigma_{j-2} \cdots \sigma_{s+1} \sigma_s \sigma_{s-1} \sigma_{s-2} \cdots\sigma_i) \, \underline{p_{s-1}} = & \sigma_{j-2} \cdots \sigma_{s+1} \sigma_s \underline{\sigma_{s-1} \sigma_{s-2} \,  p_{s-1}} \, \sigma_{s-3} \cdots \sigma_i \\ \stackrel{\text{Eq.}\ref{Prel3}}{=} & \sigma_{j-2} \cdots \sigma_{s+1} \sigma_s \, \underline{p_{s-2}} \, \sigma_{s-1} \sigma_{s-2} \sigma_{s-3} \cdots\sigma_i\\ \stackrel{\text{slide left}}{=} & \, p_{s-2} \, (\sigma_{j-2} \cdots \sigma_{s+1} \sigma_s \sigma_{s-1} \cdots\sigma_i)
\end{split}
\]
The same algorithm can be carried out on the right-hand side, so as to obtain the equivalent relation:

\begin{equation}
    \begin{split}
    \sigma_i^{-1} \cdots & \sigma_{j-2}^{-1} \sigma_{r-1}^{-1} \cdots \sigma_{s-3}^{-1} \, p_{j-1}  \, p_{s-2} \, \sigma_{j-2} \cdots \sigma_i \sigma_{s-2} \cdots \sigma_r \\ & = \sigma_r^{-1} \cdots \sigma_{s-2}^{-1} \sigma_i^{-1} \cdots \sigma_{j-2}^{-1} \, p_{s-2} \, p_{j-1} \, \sigma_{s-3} \cdots \sigma_{r-1} \sigma_{j-2} \cdots \sigma_i
\end{split}
\end{equation}
which simplifies as 
\begin{equation} \label{clasp12}
    p_{j-1}\, p_{s-2} \, = \, p_{s-2} \, p_{j-1}
\end{equation} 
since the braiding sub-blocks on the left and right satisfy
$\sigma_{j-2} \cdots \sigma_i \sigma_{s-2} \cdots \sigma_r = \sigma_{s-3} \cdots \sigma_{r-1} \sigma_{j-2} \cdots \sigma_i$ and can be mutually cancelled.

 A comment on the indices is due. The constraints on the indices in relations 1) of Proposition~\ref{alternative} impose that these new relations are true for $|j-s|>1$. Indeed, we had $i<j<r<s$, while in Case 2), we had $i<r<s<j$, so that in Case 1) the new relations became $p_{j-1} \, p_{s-2} \, = \, p_{s-2} \, p_{j-1}$, rather than $\, p_{j-1} \, p_{s-1} \, = \, p_{s-1} \, p_{j-1}$, so that the distance between $j$ and $s$ was in either case at least $2$ to begin with anyway. 
 
In addition, $j$ was supposed to be greater than $i$, and thus  in Case 1) it could not be chosen to be equal to $1$. 
However, the new relations hold for indices $j-1$ and $s-1$, so choosing $j=2$ produced a relation involving a clasp between the first and the second strands. 
 
 Something analogous happens for Case 2), where $s$ is strictly smaller than $j$ and strictly greater than both $i$ and $r$, with the relation holding for $s-2$. An analogous argument works for the choice $j = n-1$: the initial assumption was $j \leq n$, because $j$ denotes the last strand involved in the braid, the one around which the moving strand would clasp. Subtracting $1$ on both sides, the assumption on the indices becomes $j-1 \leq n-1$, so, when replacing $j-1$ with $j$ the inequality becomes $j \leq n-1$.

\smallbreak
 Therefore, we have proven that, bringing together Eq.~\ref{clasp11} and Eq.~\ref{clasp12}, Relations 1) of Proposition~\ref{alternative} are transformed to the following commuting relations between clasp generators: 
\begin{equation} \label{newrel1}
    p_{j} \, p_{s} \, = \, p_{s} \, p_{j} \quad \text{for } \ |j-s| > 1.
\end{equation}

These relations can be further reduced using the so-called contraction. 


\subsection{Contraction} \label{contractionsubsection}

The conjugation blocks in the new writings of the generators (Eq.~\ref{claspgen}) make it possible to put aside intermediate strands between two clasps and pull the clasps close together. This process shall be called \textit{contraction}, and works as an algorithm in $B_n$. The contraction of a pure braid written on conjugates of clasps consists in bringing all of the clasps as close as possible (which in the word representing it means manipulating the conjugation blocks around a $p_i$, cancelling them reciprocally, and/or taking them to the sides of the word), and then isolating the block of clasps from the rest of the braid (which means cancelling from the word the conjugation blocks at the sides). This algorithm is particularly useful for reducing relations, since the removal of the conjugation blocks can be done on both hand sides at the same time while maintaining the relation true. The algorithm uses the two relations in the braid group.
 We shall demonstrate contraction on Relations 1). Using Eq.~\ref{Prel3} on Relation 1) in the case of $i<j$: 
\begin{equation*}
    \begin{split}
        p_i p_j = p_j p_i \stackrel{\text{Eq.}~\ref{Prel3}}{\implies} & \underline{p_i} \underline{(\sigma_{j-1}^{-1} \sigma_j^{-1})(\sigma_{j-2}^{-1}  \sigma_{j-1}^{-1}) \cdots (\sigma_{i+2}^{-1} \sigma_{i+3}^{-1})} p_{i+2} (\sigma_{i+3} \sigma_{i+2}) \cdots (\sigma_{j-1} \sigma_{j-2}) (\sigma_j \sigma_{j-1})\\ = & (\sigma_{j-1}^{-1} \sigma_j^{-1})(\sigma_{j-2}^{-1}  \sigma_{j-1}^{-1}) \cdots (\sigma_{i+2}^{-1} \sigma_{i+3}^{-1}) p_{i+2} \underline{(\sigma_{i+3} \sigma_{i+2}) \cdots (\sigma_{j-1} \sigma_{j-2}) (\sigma_j \sigma_{j-1})} \underline{p_i} \\ 
        \overset{\text{commute}}{\implies} & \underline{(\sigma_{j-1}^{-1} \sigma_j^{-1})(\sigma_{j-2}^{-1}  \sigma_{j-1}^{-1}) \cdots  (\sigma_{i+2}^{-1} \sigma_{i+3}^{-1})} p_i \, p_{i+2} \underline{(\sigma_{i+3} \sigma_{i+2}) \cdots (\sigma_{j-1} \sigma_{j-2}) (\sigma_j \sigma_{j-1})}\\ = & \underline{(\sigma_{j-1}^{-1} \sigma_j^{-1})(\sigma_{j-2}^{-1}  \sigma_{j-1}^{-1}) \cdots  (\sigma_{i+2}^{-1} \sigma_{i+3}^{-1})} p_{i+2} \, p_i \underline{(\sigma_{i+3} \sigma_{i+2}) \cdots (\sigma_{j-1} \sigma_{j-2}) (\sigma_j \sigma_{j-1})} \\
        \overset{\text{cancel}}{\implies} & p_{i+2} \, p_i = p_i \, p_{i+2}
    \end{split} 
\end{equation*}
The number of relations of type 1) reduces from $\frac{n(n-1)}{2}$ to $n-2$, and relations 1) can be replaced with
\begin{equation} \label{newclasprel1}
    p_i \, p_{i+2} = p_{i+2} \, p_i, \quad 1 \leq i \leq n-3
\end{equation}


\subsection{Rewriting Relations 2) and 3) using the clasp generators}

 We shall next examine Relations 2) and 3) in Proposition~\ref{alternative}. Contraction is going to be applied in the reduction of both relations 2) and 3), in the sense that they can be reduced to relations involving three consecutive indices.   Expanding one hand side of both relations 2) and 3), we obtain the expression: 

\smallbreak
\noindent $A_{ir}A_{is}A_{rs} =$

\smallbreak
\noindent $\sigma_i^{-1} \cdots \sigma_{r-2}^{-1} p_{r-1}^{2} \underline{\sigma_{r-2} \cdots \sigma_i \sigma_i^{-1} \cdots \sigma_{r-2}^{-1}} \sigma_{r-1}^{-1} \sigma_r^{-1} \cdots \sigma_{s-2}^{-1} p_{s-1}^{2} \sigma_{s-2} \cdots \sigma_i \sigma_r^{-1} \cdots \sigma_{s-2}^{-1} p_{s-1}^2 \sigma_{s-2} \cdots \sigma_r $

\smallbreak
\noindent The block $\sigma_{r-2} \cdots \sigma_i \sigma_i^{-1} \cdots \sigma_{r-2}^{-1}$ between $A_{ir}$ and $A_{is}$ can be cancelled. Using repeatedly Eq.~\ref{braidsqrels} on the rest of $A_{is}$ and $A_{rs}$, and highlighting the resulting blocks in the second line below with parentheses, for visual easiness, the expression becomes: 
\begin{equation*}\begin{split}
& \sigma_i^{-1} \cdots \sigma_{r-2}^{-1} p_{r-1}^{2}  \sigma_{r-1}^{-1} \underline{\sigma_r^{-1} \cdots \sigma_{s-2}^{-1} p_{s-1}^{2} \sigma_{s-2} \cdots \sigma_{r+1} \sigma_r} \sigma_{r-1} \cdots \sigma_i \underline{\sigma_r^{-1} \cdots \sigma_{s-2}^{-1} p_{s-1}^2 \sigma_{s-2} \cdots \sigma_r} \\
    & \stackrel{\text{Eq.}~\ref{braidsqrels}}{=} \sigma_i^{-1} \cdots \sigma_{r-2}^{-1} p_{r-1}^{2} \sigma_{r-1}^{-1}(\sigma_{s-1} \cdots \sigma_{r+1} p_r^{2} \underline{\sigma_{r+1}^{-1} \cdots \sigma_{s-1}^{-1}}) \underline{\sigma_{r-1} \cdots \sigma_i} (\sigma_{s-1} \cdots \sigma_{r+1} p_r^{2} \sigma_{r+1}^{-1} \cdots \sigma_{s-1}^{-1}) \\
    & \overset{\text{switch}}{=} \sigma_i^{-1} \cdots \sigma_{r-2}^{-1} p_{r-1}^{2} \sigma_{r-1}^{-1} \sigma_{s-1} \cdots \sigma_{r+1} p_r^{2} \sigma_{r-1} \cdots \sigma_i \underline{\sigma_{r+1}^{-1} \cdots \sigma_{s-1}^{-1} \sigma_{s-1} \cdots \sigma_{r+1}} p_r^{2} \sigma_{r+1}^{-1} \cdots \sigma_{s-1}^{-1} \\
    & \overset{\text{cancel}}{=} \sigma_i^{-1} \cdots \sigma_{r-2}^{-1} p_{r-1}^{2} \sigma_{r-1}^{-1} \underline{\sigma_{s-1} \cdots \sigma_{r+1}} p_r^{2} \sigma_{r-1} \cdots \sigma_i \underline{p_r^{2}} \sigma_{r+1}^{-1} \cdots \sigma_{s-1}^{-1}\\
    & \overset{\text{slide left}}{=} (\sigma_{s-1} \cdots \sigma_{r+1} \sigma_i^{-1} \cdots \sigma_{r-2}^{-1}) \, p_{r-1}^{2} \sigma_{r-1}^{-1} p_r^{2} \sigma_{r-1} p_r^{2} \, (\sigma_{r-2} \cdots \sigma_i  \sigma_{r+1}^{-1} \cdots \sigma_{s-1}^{-1}).
\end{split}
\end{equation*}
The last expression is divided in three pieces: two conjugating blocks at the sides and a central braid block. 

\smallbreak
 An analogous procedure can be carried out on both $A_{is} A_{rs} A_{ir}$ and $A_{rs} A_{ir}A_{is}$ (the other hand sides of relations 2) and 3) in Proposition~\ref{alternative}, respectively), so as to obtain accordingly the expressions: 
\begin{equation*}
    \begin{split}
    & (\sigma_{s-1} \cdots \sigma_{r+1} \sigma_i^{-1} \cdots \sigma_{r-2}^{-1}) \, \sigma_{r-1}^{-1} p_r^{2} \sigma_{r-1} p_r^{2} p_{r-1}^{2} \, (\sigma_{r-2} \cdots \sigma_i  \sigma_{r+1}^{-1} \cdots \sigma_{s-1}^{-1})\\
    & (\sigma_{s-1} \cdots \sigma_{r+1} \sigma_i^{-1} \cdots \sigma_{r-2}^{-1}) \, p_r^{2} p_{r-1}^{2} \sigma_{r-1}^{-1} p_r^{2} \sigma_{r-1} \,(\sigma_{r-2} \cdots \sigma_i \sigma_{r+1}^{-1} \cdots \sigma_{s-1}^{-1})
    \end{split}
\end{equation*}
Conjugation by $(\sigma_{s-1} \cdots \sigma_{r+1} \sigma_i^{-1} \cdots \sigma_{r-2}^{-1})$ can be removed from all hand sides, so relations 2) contract to:
\begin{equation*}
    p_{r-1, r} \, p_{r-1, r+1} \, p_{r, r+1} = p_{r-1, r+1} \, p_{r, r+1} \, p_{r-1, r} 
\end{equation*}
whereas relations 3) contract to: 
\begin{equation*}
    p_{r-1, r} \, p_{r-1, r+1} \, p_{r, r+1} = p_{r, r+1} \, p_{r-1, r} \, p_{r-1, r+1}.
\end{equation*}

\noindent  A change of indices ($r \to i+1$) gives the more familiar forms. Therefore, Relations 2) of Proposition~\ref{alternative}  can be contracted  to the relations: 
\begin{equation} \label{P2}
p_{i, i+1} \, p_{i, i+2} \, p_{i+1, i+2} = p_{i, i+2}\, p_{i+1, i+2} \, p_{i, i+1}.
\end{equation}
Which can be rewritten as:
\begin{equation} \label{P21}
    p_i \, (\sigma_i^{-1} \, p_{i+1} \, \sigma_i) \, p_{i+1} = (\sigma_i^{-1} \, p_{i+1} \, \sigma_i) \, p_{i+1} \, p_i
\end{equation}
We further note that we can simplify the left-hand side of Eq.~\ref{P21} using the interaction of clasps with classical braid generators Eq.~\ref{Prel4}, to obtain: 
\[
\underline{p_i \, \sigma_i^{-1}} \, p_{i+1} \, \sigma_i \, p_{i+1} = \sigma_i \, p_{i+1} \, \sigma_i \, p_{i+1}
\]

\noindent  The right-hand side can be also simplified to obtain:
\begin{equation*}
    \sigma_i^{-1} \, \underline{p_{i+1} \, \sigma_i \, p_{i+1} \, \sigma_i^2 }\stackrel{\text{Eq.}~\ref{Prel2}}{=} \underline{\sigma_i^{-1} \sigma_i} \, p_{i+1} \, \sigma_{i} \, p_{i+1} \, \sigma_{i} = p_{i+1} \, \sigma_{i} \, p_{i+1} \,\sigma_{i}.
\end{equation*}
    
The recomposition of the two hand sides of Eq.~\ref{P2} is equivalent to:
\begin{equation} \label{Anotherclaspgen}
    (\sigma_i \, p_{i+1} \, \sigma_i) \, p_{i+1} = p_{i+1} \, (\sigma_{i} \, p_{i+1} \,\sigma_{i})
\end{equation}

\noindent which is of the form of Eq.~\ref{Prel2}. 

Another comment on the indices is due. The bounds for Relations 2) in Proposition~\ref{alternative} were $1 \leq i<r<j \leq n$. In these relations' new version in Eq.~\ref{P2}, the index $r-1$ is replaced by $i$ and $r$ by $i+1$. This means that one still needs $1 \leq i$ for the lower bound. For the upper bound, we have $i+1$ as highest index in Eq.~\ref{P2}. Since $i + 1 = r < j \leq n$, we have $i + 1 = r < n$ and hence $ i < n-1$, so that $p_{i+1}$ can be at most $p_{n-1}$. 

\bigbreak

 In the same spirit of Relations 2), Relations 3) of Proposition~\ref{alternative} can be contracted, using the classical braid relations, to the following:  
 \begin{equation} \label{Againclaspgen3}
 p_{i, i+1} \, p_{i, i+2} \, p_{i+1, i+2} = p_{i+1, i+2} \, p_{i, i+1} \, p_{i, i+2}.
 \end{equation}
 Rewriting the equation in terms of clasp generators we obtain:
\begin{equation} \label{Againclaspgen31}
    p_i \, (\sigma_i^{-1} \, p_{i+1} \, \sigma_i) \, p_{i+1} = p_{i+1} \, p_i \, (\sigma_i^{-1} \, p_{i+1} \, \sigma_i)
\end{equation}
which is again a relation among three new generators, on the same index set as Relations 2). Additionally, using the interaction of clasp generators with classical braid generators Eq.~\ref{Prel4} on both hand sides of Eq.~\ref{Againclaspgen3}, we recover Eq.~\ref{Anotherclaspgen}:
\begin{equation*}
   (\sigma_i \, p_{i+1} \, \sigma_i) \, p_{i+1} = p_{i+1} \, (\sigma_i \, p_{i+1} \, \sigma_i)
\end{equation*}

\noindent which is again of the form of Eq.~\ref{Prel2}, and surprisingly implies that Relations 3) can be recovered from Relations 2) in this formalism.

\subsection{Rewriting Relations 4) using the clasp generators}

 Inspired by Relations 2) and 3), Relations 4) of Proposition~\ref{alternative} are contracted into the following, assuming the classical braid relations: 
\begin{equation}\label{Againclaspgen4}
p_{i, i+1} \, p_{i, i+2} \, p_{i, i+1}^{-1} \, p_{i-1, i+1} = p_{i-1, i+1} \, p_{i, i+1} \, p_{i, i+2} \, p_{i, i+1}^{-1} \quad \text{for } i<r<j<s 
\end{equation}  
Rewriting in terms of the $p_i$'s:
\begin{equation} \label{Againclaspgen42}
    p_i \, (\sigma_i^{-1} p_{i+1} \sigma_i) \, p_{i}^{-1} \, (\sigma_{i-1}^{-1} p_i \sigma_{i-1})  = (\sigma_{i-1}^{-1} p_i \sigma_{i-1}) p_i \, (\sigma_i^{-1} p_{i+1} \sigma_i) p_i^{-1}
\end{equation} 
Using further the interaction between classical braid generators and clasp generators, we obtain:
\begin{equation*}
    \begin{split}
        \underline{\sigma_i} p_{i+1} \sigma_i^{-1} \sigma_{i-1}^{-1} p_i \sigma_{i-1} & = \sigma_{i-1}^{-1} p_i \sigma_{i-1} \sigma_i p_{i+1} \underline{\sigma_i^{-1}} \Longleftrightarrow \\
        p_{i+1} \underline{\sigma_i^{-1} \sigma_{i-1}^{-1} p_i \sigma_{i-1} \sigma_i} & = \underline{\sigma_i^{-1} \sigma_{i-1}^{-1} p_i \sigma_{i-1} \sigma_i} p_{i+1} \stackrel{Eq.~\ref{Prel1}}{\Longleftrightarrow} \\
        p_{i+1} p_{i-1} & = p_{i-1} p_{i+1}
    \end{split}
\end{equation*}
recovering Relations 1).

\subsection{A presentation for $P_n$ using the clasp generators}

Combining the relations we have found using the classical braid relations (so Eqs.~\ref{Prel3}, ~\ref{newclasprel1},~\ref{P2},~\ref{Againclaspgen3}, and~\ref{Againclaspgen4}), we have proven the following:

\begin{theorem} \label{Pclaspres} The pure braid group $P_n$ on $n$ strands admits a presentation with generators of normal closure conjugates of the clasp generators \begin{equation*}
   p_{ij}  = \sigma_i^{-1} \cdots \sigma_{j-2}^{-1} p_{j-1} \sigma_{j-2} \cdots \sigma_i , \quad 1 \leq i<j \leq n
\end{equation*} 
and the following contracted relations, assuming the defining relations for $B_n$, Eq.~\ref{Prel3}, and that $\sigma_j \, p_i \, \sigma_j^{-1} = p_i$ for $|i-j| > 1$: 
\[ \begin{array}{lrcll}
    1) & p_i \, p_{i+2} & = & p_{i+2} \, p_i & \text{for} \hspace{2mm} 1 \leq i \leq n-3\\
    2) & p_i \, p_{i, i+2} \, p_{i+1} & = & p_{i, i+2} \, p_{i+1} \, p_i & \text{for} \hspace{2mm} 1 \leq i \leq n-2\\
    3) & p_i \, p_{i, i+2} \, p_{i+1} & = & p_{i+1} \, p_i \, p_{i, i+2} & \text{for} \hspace{2mm} 1 \leq i \leq n-2\\
    4) & (p_i \, p_{i, i+2} \, p_{i}^{-1}) \, p_{i-1, i+1} & = & p_{i-1, i+1} \, (p_i \, p_{i, i+2} \, p_i^{-1}) & \text{for} \hspace{2mm} 2 \leq i \leq n-2 \\
\end{array} \] 
\end{theorem}

Furthermore, assuming also Eq.~\ref{Prel1} we obtain the following result, based on the reduced Eqs.~\ref{newclasprel1}, ~\ref{Anotherclaspgen}, and~\ref{Againclaspgen4}, which was the main target of this section: 

\begin{theorem} \label{braidpreswclasp}
 The braid group $B_n$ on $n$ strands admits a redundant presentation with the classical braid generators $\sigma_i$ and the clasp generators $p_i$, satisfying the over-abundant set of relations:
\[ \begin{array}{lrcll} \\
    1) & \sigma_i \sigma_j & = & \sigma_j \sigma_i & \text{for} \hspace{2mm} |i-j|>1\\[3pt]
    2) & \sigma_i \sigma_{i+1} \sigma_i & = & \sigma_{i+1} \sigma_i \sigma_{i+1} & \text{for} \hspace{2mm} 1 \leq i < n - 1 \\[3pt]
    3) & p_i \, p_{i+2} & = & p_{i+2} \, p_i & \text{for} \hspace{2mm} 1 \leq i < n-2\\[3pt]
   4) &  (\sigma_i \, p_{i+1} \, \sigma_i) \, p_{i+1} & = & p_{i+1} \, (\sigma_{i} \, p_{i+1} \,\sigma_{i}) & \text{for} \hspace{2mm} 1 \leq i < n-1\\[3pt]
   5)& p_i \, \sigma_j & = & \sigma_j \,  p_i  & \text{for} \hspace{2mm} |i-j|>1 \\[3pt]
6) &  p_i \, \sigma_i & = & \sigma_i  \,  p_i   & \text{for} \hspace{2mm} 1 \leq i < n \\[3pt]
7) & p_i \sigma_{i+1} \sigma_i & = & \sigma_{i+1} \sigma_i p_{i+1} & \text{for } 1 \leq i < n - 1 \\[3pt]
8) & \sigma_i \sigma_{i+1} p_i & = & p_{i+1} \sigma_i \sigma_{i+1} & \text{for all } 1 \leq i < n - 1 \\[3pt]
 9) & p_i \, & = & \sigma_i^2 & \text{for all } 1 \leq i < n.  
\end{array} \] 
\end{theorem}

\begin{figure}[H]
\begin{center} 
    \includegraphics[width=17cm]{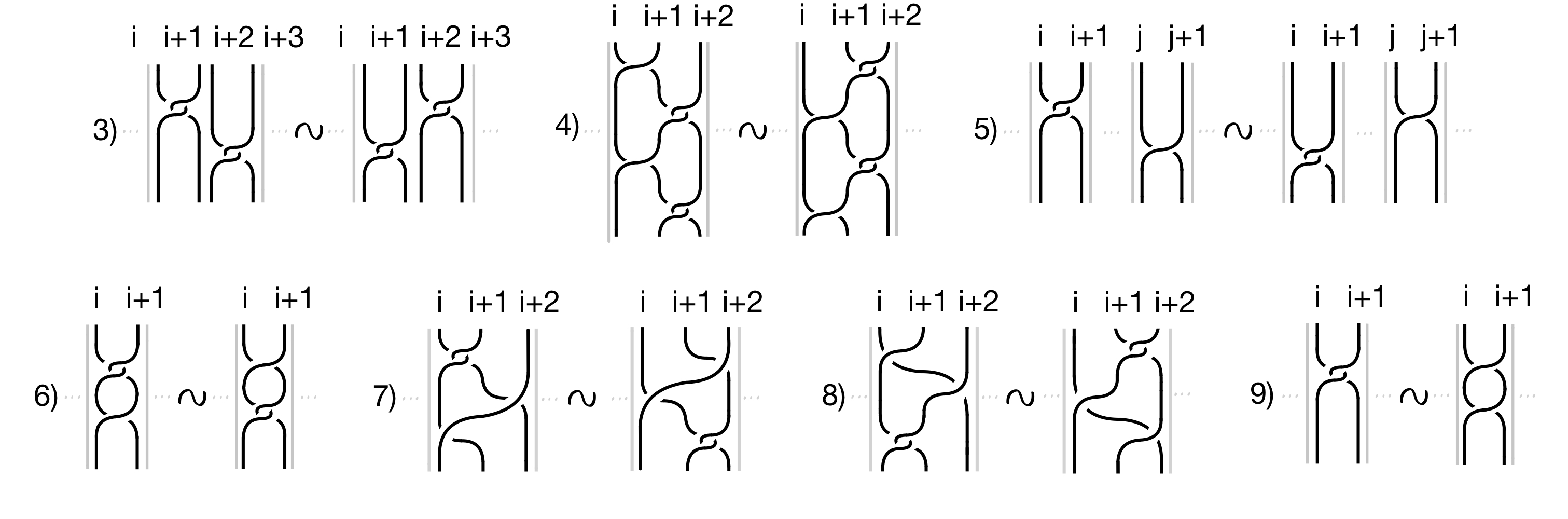}
\end{center}
\caption{Pictures of the clasp relations 3) to 9) in Theorem~\ref{braidpreswclasp}}
\label{figclaspres}
\end{figure}

We may see this presentation as a modification of Proposition~\ref{redBn}.

\begin{remark} \label{claspremarkrel4}
    The off-looking relations 4) of Theorem~\ref{braidpreswclasp} encode the cyclic relations of the alternative presentation of $P_n$ in Proposition~\ref{alternative}. Note that they can be derived by relations 7), 8) and 9) of the Theorem, as similarly Eq.~\ref{Prel2} is derived from Eqs.~\ref{Prel1} \& ~\ref{Prel3}. In the next section, this presentation will be related to that of the bonded braid monoid, and keeping relations 4) in the statement will underline the fact that bonds do not permute cyclically.
\end{remark}


\subsection{A reduced presentation for $P_n$ and for $B_n$ with one clasp generator} \label{irredundant}

Using Eq.~\ref{Prel1}, this presentation of $B_n$ in Theorem~\ref{braidpreswclasp} can be further reduced: every generator $p_i$ can be obtained as a suitable conjugate of $p_1$:
\begin{equation} \label{W}
   \begin{split}
        p_i & = (\sigma_{i-1} \sigma_i)(\sigma_{i-2} \sigma_{i-1}) \cdots (\sigma_1 \sigma_2) \, p_1 \, (\sigma_1 \sigma_2)^{-1} \cdots (\sigma_{i-2} \sigma_{i-1})^{-1} (\sigma_{i-1} \sigma_i)^{-1} \\
        & = (\sigma_{i-1} \cdots \sigma_1) (\sigma_i \cdots \sigma_2) \, p_1 \,  (\sigma_i \cdots \sigma_2)^{-1} (\sigma_{i-1} \cdots \sigma_1)^{-1}
   \end{split}
\end{equation}

The two conjugates of $p_1$ boil down to the same word, as:

\begin{equation} \label{sameconjugates}
    \begin{split}
    (\sigma_{i-1} \sigma_{i-2} \cdots \sigma_1) (\underline{\sigma_i} \sigma_{i-1} \sigma_{i-2} \cdots \sigma_2) & = (\sigma_{i-1} \sigma_i) ( \sigma_{i-2} \sigma_{i-3} \cdots \sigma_1) ( \underline{\sigma_{i-1}} \sigma_{i-2} \cdots \sigma_1)\\
    & = (\sigma_{i-1} \sigma_i) (\sigma_{i-2} \sigma_{i-1}) (\sigma_{i-3} \cdots \sigma_1) ( \underline{\sigma_{i-2}} \cdots \sigma_1)\\
    & \cdots \\
    & = (\sigma_{i-1} \sigma_i)(\sigma_{i-2} \sigma_{i-1}) \cdots (\sigma_1 \sigma_2)
\end{split}
\end{equation}

These expressions can be used to express all relations in Theorem~\ref{braidpreswclasp} in terms of $p_1$. So, we shall prove the following:

\begin{theorem} \label{thm:irredundant}
The braid group \( B_n \) admits the following reduced presentation: it is generated by the classical braid generators   and a single clasp generator \( p_1 \), subject to the relations:
\[
\begin{array}{lrcll}
1) & \sigma_i\, \sigma_j & = & \sigma_j\, \sigma_i & \quad \text{for } |i-j|>1, \\[3pt]
2) & \sigma_i\, \sigma_{i+1}\, \sigma_i  & = &  \sigma_{i+1}\, \sigma_i\, \sigma_{i+1} & \quad \text{for } 1 \leq i \leq n-2, \\[3pt]
3) & p_1 \, (\sigma_2 \sigma_1 \sigma_3 \sigma_2)^{-1}  \, p_1\, (\sigma_2 \sigma_1 \sigma_3 \sigma_2) & = & (\sigma_2 \sigma_1 \sigma_3 \sigma_2)^{-1}  \, p_1\, (\sigma_2 \sigma_1 \sigma_3 \sigma_2) \, p_1  \\[3pt]
4) & (\sigma_2 p_1 \sigma_2^{-1}) \sigma_1 (\sigma_2 p_1 \sigma_2^{-1}) \sigma_1  & = & \sigma_1 (\sigma_2 p_1 \sigma_2^{-1}) \sigma_1 (\sigma_2 p_1 \sigma_2^{-1})\\[3pt]
5) & p_1 \sigma_j & = & \sigma_j p_1 & \quad \text{for } j > 2 \\[3pt]
6) & p_1 \sigma_1 & = & \sigma_1 p_1 \\[3pt]
7) & p_1\, (\sigma_2 \sigma_1)(\sigma_1\, \sigma_2) & = &  (\sigma_2 \sigma_1)(\sigma_1\, \sigma_2) \, p_1 & \\[3pt] 
8) & p_1 \sigma_1^{-1} & = & \sigma_1
\end{array}
\]
\end{theorem}

\begin{proof}
The relations in one clasp generator are pictured in Figure~\ref{fig1claspres}. 

\begin{figure}
\begin{center} 
    \includegraphics[width=16cm]{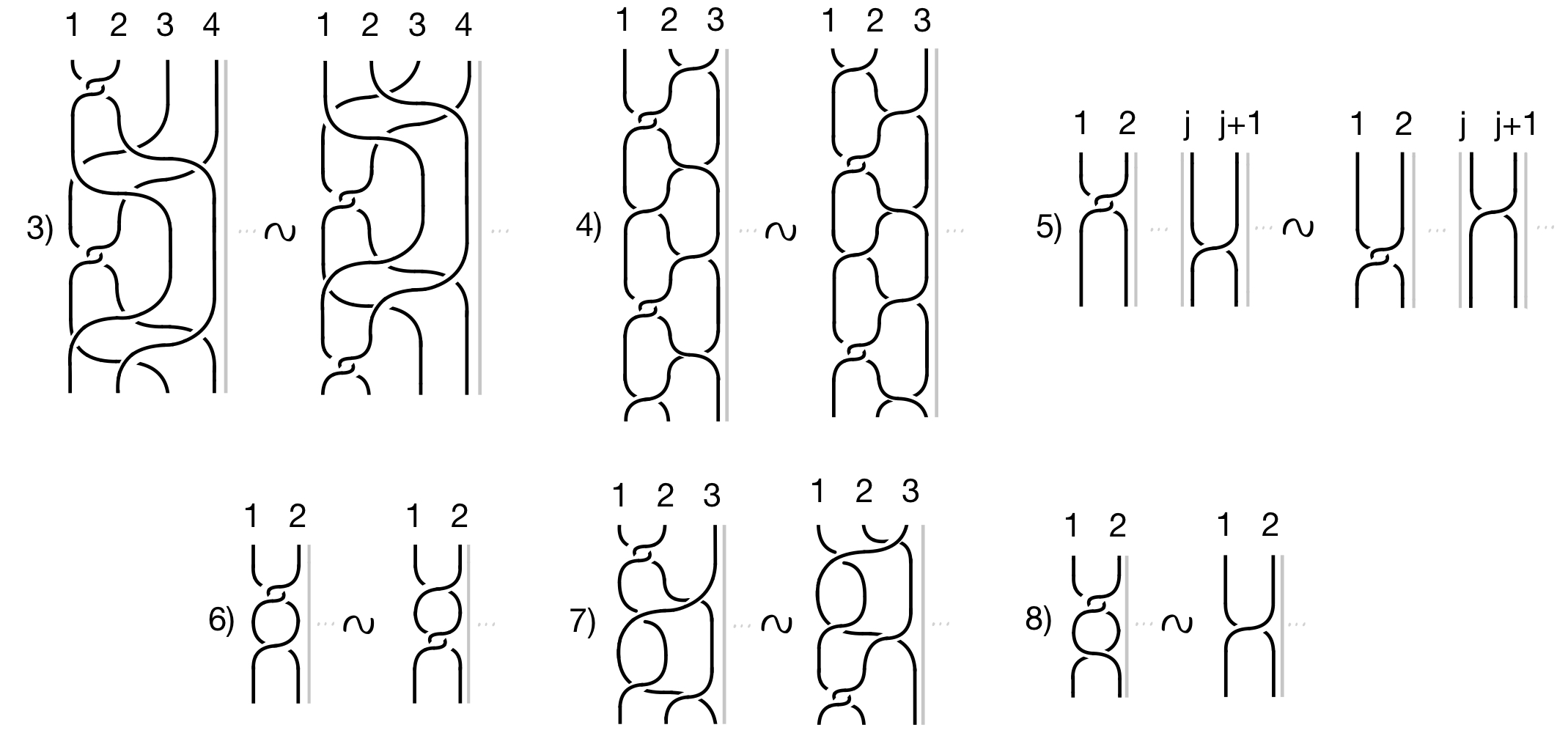}
\end{center}
\caption{Pictures of the clasp relations 3) to 8) in Theorem~\ref{thm:irredundant}.}
\label{fig1claspres}
\end{figure}

We begin with relations 3) of Theorem~\ref{braidpreswclasp} (since 1) and 2) do not involve  clasps). The idea is, using Eq.~\ref{W}, to write $p_i$ in terms of $p_1$ and $p_{i+2}$ in terms of $p_3$, and then to expand the latter as a conjugate of $p_1$.

\smallbreak
\noindent \textit{Relations 3): $p_i \, p_{i+2} = p_{i+2} \, p_i$}.
The left-hand side of relations 3) can be rewritten as: 
\begin{equation*}
    \begin{split}
        p_i \, p_{i+2} & \stackrel{\text{Eq.~\ref{W}}}{=} (\sigma_{i-1} \sigma_i ) \cdots (\sigma_2 \sigma_1) p_1 (\sigma_2 \sigma_1)^{-1} \cdots (\sigma_{i-1} \sigma_i )^{-1} \underline{p_{i+2}} \\
        & \stackrel{\text{shift}}{=} (\sigma_{i-1} \sigma_i ) \cdots (\sigma_2 \sigma_1) p_1  \underline{p_{i+2}} (\sigma_2 \sigma_1)^{-1} \cdots (\sigma_{i-1} \sigma_i )^{-1} \\
        & \stackrel{\text{Eq.~\ref{W}}}{=} (\sigma_{i-1} \sigma_i ) \cdots (\sigma_2 \sigma_1) \underline{p_1}  (\sigma_{i+1} \sigma_{i+2} ) \cdots (\sigma_4 \sigma_3) p_3 (\sigma_4 \sigma_3)^{-1} \cdots (\sigma_{i+1} \sigma_{i+2} )^{-1} (\sigma_2 \sigma_1)^{-1} \cdots (\sigma_{i-1} \sigma_i )^{-1} \\
        & \stackrel{\text{shift}}{=} (\sigma_{i-1} \sigma_i ) \cdots (\sigma_2 \sigma_1) \, (\sigma_{i+1} \sigma_{i+2} ) \cdots (\sigma_4 \sigma_3) p_1 \, p_3 (\sigma_4 \sigma_3)^{-1} \cdots (\sigma_{i+1} \sigma_{i+2})^{-1} \, (\sigma_2 \sigma_1)^{-1} \cdots (\sigma_{i-1} \sigma_i )^{-1}
    \end{split}
\end{equation*}

Similarly, the right-hand side yields
\[
(\sigma_{i-1} \sigma_i ) \cdots (\sigma_2 \sigma_1) \, (\sigma_{i+1} \sigma_{i+2} ) \cdots (\sigma_4 \sigma_3) p_3 \, p_1 (\sigma_4 \sigma_3)^{-1} \cdots (\sigma_{i+1} \sigma_{i+2})^{-1} \, (\sigma_2 \sigma_1)^{-1} \cdots (\sigma_{i-1} \sigma_i )^{-1}
\]

Removing the two conjugating expressions at the extremities of both hand sides, and writing $p_3$ in terms of $p_1$, relations 3) are transformed into the relation:
\begin{equation*}
    p_1 \, p_3 \, = \, p_3 \, p_1 \implies p_1 (\sigma_2 \sigma_3) (\sigma_1 \sigma_2) p_1 (\sigma_1 \sigma_2)^{-1} (\sigma_2 \sigma_3)^{-1} = (\sigma_2 \sigma_3) (\sigma_1 \sigma_2) p_1 (\sigma_1 \sigma_2)^{-1} (\sigma_2 \sigma_3)^{-1} p_1
\end{equation*}
which can be slightly manipulated into the form of relation 3) of the statement of the theorem: 
\begin{equation}
    p_1 \, (\sigma_2 \sigma_1 \sigma_3 \sigma_2)^{-1}  \, p_1\, (\sigma_2 \sigma_1 \sigma_3 \sigma_2) = (\sigma_2 \sigma_1 \sigma_3 \sigma_2)^{-1}  \, p_1\, (\sigma_2 \sigma_1 \sigma_3 \sigma_2) \, p_1.  
\end{equation}

\smallbreak
\noindent \textit{Relations 4): $\sigma_i \, p_{i+1} \, \sigma_i \, p_{i+1} = p_{i+1} \, \sigma_{i} \, p_{i+1} \,\sigma_{i}$.}
For relations 8) of Theorem~\ref{braidpreswclasp}, the clasp shifting should also involve  $\sigma_i$, which we should locally conjugate to $\sigma_1$. As a consequence, $p_{i+1}$ will be locally conjugated to  $p_2$, and later expressed in terms of $p_1$. The process of properly doing so is particularly long. 
 Consider first the piece $\sigma_i \, p_{i+1}$: the clasp $p_{i+1}$ is substituted by the usual conjugate of $p_1$. Then, the first relation in the braid group allows to insert the expression of inverses $(\sigma_{i+1} \cdots \sigma_3) (\sigma_2^{-1} \sigma_1^{-1})(\sigma_1 \sigma_2) (\sigma_3^{-1} \cdots \sigma_{i+1}^{-1})$, once before and once after $p_1$.

\[ \begin{split}
    \sigma_i \, p_{i+1} & = \sigma_i \, (\sigma_i \cdots \sigma_1) (\sigma_{i+1} \cdots \sigma_2) \, p_1 \,  (\sigma_2^{-1} \cdots \sigma_{i+1}^{-1}) (\sigma_1^{-1} \cdots \sigma_i^{-1}) \\
    & = \sigma_i \, (\sigma_i \cdots \sigma_1) (\sigma_{i+1} \cdots \sigma_2) (\sigma_{i+1} \cdots \sigma_3) (\sigma_2^{-1} \sigma_1^{-1}) (\sigma_1 \sigma_2) \underline{(\sigma_3^{-1} \cdots \sigma_{i+1}^{-1})} \, p_1 \, \underline{(\sigma_{i+1} \cdots \sigma_3)} (\sigma_2^{-1} \sigma_1^{-1}) \cdot \\ & \quad \cdot (\sigma_1 \sigma_2) (\sigma_3^{-1} \cdots \sigma_{i+1}^{-1}) (\sigma_2^{-1} \cdots \sigma_{i+1}^{-1}) (\sigma_1^{-1} \cdots \sigma_i^{-1}) \\
    & \stackrel{\text{cancel}}{=} \sigma_i \, (\sigma_i \cdots \sigma_1) (\sigma_{i+1} \cdots \sigma_2) (\sigma_{i+1} \cdots \sigma_3) (\sigma_2^{-1} \sigma_1^{-1}) \underline{(\sigma_1 \sigma_2) \, p_1 \, (\sigma_2^{-1} \sigma_1^{-1})} (\sigma_1 \sigma_2) \cdot \\ 
    & \quad \cdot (\sigma_3^{-1} \cdots \sigma_{i+1}^{-1}) (\sigma_2^{-1} \cdots \sigma_{i+1}^{-1}) (\sigma_1^{-1} \cdots \sigma_i^{-1}) \\
     & \stackrel{\text{Eq.}~\ref{Prel1}}{=} \sigma_i \, (\sigma_i \cdots \sigma_1) \underline{(\sigma_{i+1} \cdots \sigma_2) (\sigma_{i+1} \cdots \sigma_3) (\sigma_2^{-1}} \sigma_1^{-1}) \, p_2 \, (\sigma_1 \underline{\sigma_2) (\sigma_3^{-1} \cdots \sigma_{i+1}^{-1}) (\sigma_2^{-1} \cdots \sigma_{i+1}^{-1})} (\sigma_1^{-1} \cdots \sigma_i^{-1}) \\
     & \stackrel{\text{ braid rels.}}{=} \sigma_i \, \underline{(\sigma_i \cdots \sigma_1) (\sigma_i \cdots \sigma_2) \sigma_1^{-1}} (\sigma_{i+1} \cdots \sigma_3) \, p_2 \, (\sigma_3^{-1} \cdots \sigma_{i+1}^{-1}) \underline{\sigma_1 (\sigma_2^{-1} \cdots \sigma_i^{-1}) (\sigma_1^{-1} \cdots \sigma_i^{-1})} \\
     & \stackrel{\text{ braid rels.}}{=} \sigma_i \, \underline{(\sigma_{i-1} \cdots \sigma_1) (\sigma_i \cdots \sigma_2)} (\sigma_{i+1} \cdots \sigma_3) \, p_2 \, (\sigma_3^{-1} \cdots \sigma_{i+1}^{-1}) \underline{(\sigma_2^{-1} \cdots \sigma_i^{-1}) (\sigma_1^{-1} \cdots \sigma_{i-1}^{-1})} \\
     & \stackrel{\text{ Eq.}~\ref{sameconjugates}}{=} \underline{\sigma_i \, (\sigma_{i-1} \sigma_i)(\sigma_{i-2} \sigma_{i-1}) \cdots (\sigma_1 \sigma_2)} (\sigma_{i+1} \cdots \sigma_3) \, p_2 \, (\sigma_3^{-1} \cdots \sigma_{i+1}^{-1}) \cdot (\sigma_1 \sigma_2)^{-1} \cdots (\sigma_{i-2} \sigma_{i-1})^{-1} (\sigma_{i-1} \sigma_i)^{-1} \\
     & \stackrel{\text{ braid rels.}}{=} (\sigma_{i-1} \sigma_i)(\sigma_{i-2} \sigma_{i-1}) \cdots (\sigma_1 \sigma_2) \underline{\sigma_1} \underline{(\sigma_{i+1} \cdots \sigma_3)} \, p_2 \, (\sigma_3^{-1} \cdots \sigma_{i+1}^{-1}) (\sigma_1 \sigma_2)^{-1} \cdots (\sigma_{i-2} \sigma_{i-1})^{-1} (\sigma_{i-1} \sigma_i)^{-1} \\
     & \stackrel{\text{commute}}{=} (\sigma_{i-1} \sigma_i)(\sigma_{i-2} \sigma_{i-1}) \cdots (\sigma_1 \sigma_2) (\sigma_{i+1} \cdots \sigma_3) \sigma_1 \, p_2 \, (\sigma_3^{-1} \cdots \sigma_{i+1}^{-1}) (\sigma_1 \sigma_2)^{-1} \cdots (\sigma_{i-2} \sigma_{i-1})^{-1} (\sigma_{i-1} \sigma_i)^{-1}
\end{split}
\]
As a result,
\[
\begin{split}
     \sigma_i \, p_{i+1} & \sigma_i \, p_{i+1} \\
     & = (\sigma_{i-1} \sigma_i)(\sigma_{i-2} \sigma_{i-1}) \cdots (\sigma_1 \sigma_2) (\sigma_{i+1} \cdots \sigma_3) \sigma_1 \, p_2 \, \sigma_1 \, p_2 \, (\sigma_3^{-1} \cdots \sigma_{i+1}^{-1}) (\sigma_1 \sigma_2)^{-1} \cdots (\sigma_{i-2} \sigma_{i-1})^{-1} (\sigma_{i-1} \sigma_i)^{-1}
\end{split}
\]
Similarly,
\[
\begin{split}
    p_{i+1} \sigma_i & \, p_{i+1} \, \sigma_i \\
     & = (\sigma_{i-1} \sigma_i)(\sigma_{i-2} \sigma_{i-1}) \cdots (\sigma_1 \sigma_2) (\sigma_{i+1} \cdots \sigma_3) \, p_2 \, \sigma_1 \, p_2 \,\sigma_1 \, (\sigma_3^{-1} \cdots \sigma_{i+1}^{-1}) (\sigma_1 \sigma_2)^{-1} \cdots (\sigma_{i-2} \sigma_{i-1})^{-1} (\sigma_{i-1} \sigma_i)^{-1}
\end{split}
\]
Removing the conjugating expressions and recomposing the two hand sides to the original relation, we obtain
\[
\begin{split}
    & \sigma_1 p_2 \sigma_1 p_2 = p_2 \sigma_1 p_2 \sigma_1 \stackrel{\text{Eq.}~\ref{Prel1}}{\Longleftrightarrow} \\ & \sigma_1\, (\sigma_1 \sigma_2 p_1 \sigma_2^{-1} \sigma_1^{-1}) \sigma_1\, (\sigma_1 \sigma_2 p_1 \sigma_2^{-1} \sigma_1^{-1}) =  (\sigma_1 \sigma_2 p_1 \sigma_2^{-1} \sigma_1^{-1}) \sigma_1\, (\sigma_1 \sigma_2 p_1 \sigma_2^{-1} \sigma_1^{-1}) \, \sigma_1 \Longleftrightarrow  \\
     & \sigma_1 \, \sigma_1 \sigma_2 p_1 \sigma_2^{-1} \sigma_1 \sigma_2 p_1 \sigma_2^{-1} \sigma_1^{-1} = \sigma_1 \sigma_2 p_1 \sigma_2^{-1} \sigma_1 \sigma_2 p_1 \sigma_2^{-1} \Longleftrightarrow \\
      & \sigma_1 (\sigma_2 p_1 \sigma_2^{-1}) \sigma_1 (\sigma_2 p_1 \sigma_2^{-1}) = (\sigma_2 p_1 \sigma_2^{-1}) \sigma_1 (\sigma_2 p_1 \sigma_2^{-1}) \sigma_1
\end{split}
\]
which is precisely the desired expression for relations 4) of Theorem~\ref{thm:irredundant}.

\smallbreak
\noindent \textit{Relations 5): $p_i \, \sigma_j = \sigma_j \, p_i$}. We have to consider the two possible cases:

\noindent \textit{If j>i+1:} In this case $\sigma_j$ does not interfere with the clasp, so the left-hand side can be rewritten as
\begin{equation*}
    \begin{split}
        p_i  \, \sigma_j & = (\sigma_{i-1} \sigma_i)(\sigma_{i-2} \sigma_{i-1}) \cdots (\sigma_1 \sigma_2) \, p_1 \, \underline{(\sigma_1 \sigma_2)^{-1} \cdots (\sigma_{i-2} \sigma_{i-1})^{-1} (\sigma_{i-1} \sigma_i)^{-1}} \sigma_j\\
        & = (\sigma_{i-1} \sigma_i)(\sigma_{i-2} \sigma_{i-1}) \cdots (\sigma_1 \sigma_2) \, p_1 \, \sigma_j (\sigma_1 \sigma_2)^{-1} \cdots (\sigma_{i-2} \sigma_{i-1})^{-1} (\sigma_{i-1} \sigma_i)^{-1}
    \end{split}
\end{equation*}
And similarly, the right-hand side can be rewritten as
\[
(\sigma_{i-1} \sigma_i)(\sigma_{i-2} \sigma_{i-1}) \cdots (\sigma_1 \sigma_2) \, \sigma_j \, p_1 (\sigma_1 \sigma_2)^{-1} \cdots (\sigma_{i-2} \sigma_{i-1})^{-1} (\sigma_{i-1} \sigma_i)^{-1}
\]

Removing conjugation from both hand sides, we obtain
\begin{equation} \label{eq4W1}
    p_1 \, \sigma_j = \sigma_j \,  p_1.
\end{equation}

\noindent \textit{Else j<i-1:} Bringing $\sigma_j$ close to $p_1$ is a bit longer. The left-hand side follows using the second line of Eq.~\ref{W}:

\begin{equation*}
    \begin{split}
        p_i \, \sigma_j & = (\sigma_{i-1} \cdots \sigma_1) (\sigma_i \cdots \sigma_2) p_1 (\sigma_i \cdots \sigma_{j+1} \sigma_j \sigma_{j-1} \cdots \sigma_2)^{-1} (\sigma_{i-1} \cdots \sigma_{j+1} \sigma_j \sigma_{j-1} \cdots \sigma_1)^{-1} \sigma_j\\
        & = (\sigma_{i-1} \cdots \sigma_1) (\sigma_i \cdots \sigma_2) p_1 (\sigma_2^{-1} \cdots \sigma_{j-1}^{-1} \sigma_j^{-1} \sigma_{j+1}^{-1} \cdots \sigma_i^{-1}) (\sigma_{1}^{-1} \cdots \sigma_{j-1}^{-1} \sigma_j^{-1} \sigma_{j+1}^{-1} \cdots \sigma_{i-1}^{-1}) \underline{\sigma_j}\\
        & \stackrel{\text{shift}}{=} (\sigma_{i-1} \cdots \sigma_1) (\sigma_i \cdots \sigma_2) p_1 (\sigma_2^{-1} \cdots \sigma_{j-1}^{-1} \sigma_j^{-1} \sigma_{j+1}^{-1} \cdots \sigma_i^{-1}) (\sigma_{1}^{-1} \cdots \sigma_{j-1}^{-1} \underline{\sigma_j^{-1} \sigma_{j+1}^{-1} \sigma_j} \cdots \sigma_{i-1}^{-1}) \\
        & \stackrel{\text{braid rel.}}{=} (\sigma_{i-1} \cdots \sigma_1) (\sigma_i \cdots \sigma_2) p_1 (\sigma_2^{-1} \cdots \sigma_{j-1}^{-1} \sigma_j^{-1} \sigma_{j+1}^{-1} \cdots \sigma_i^{-1}) (\sigma_{1}^{-1} \cdots \sigma_{j-1}^{-1} \underline{\sigma_{j+1}} \sigma_j^{-1} \sigma_{j+1}^{-1} \cdots \sigma_{i-1}^{-1}) \\
        & \stackrel{\text{shift}}{=} (\sigma_{i-1} \cdots \sigma_1) (\sigma_i \cdots \sigma_2) p_1 (\sigma_2^{-1} \cdots \sigma_{j-1}^{-1} \sigma_j^{-1} \underline{\sigma_{j+1}^{-1} \sigma_{j+2}^{-1} \sigma_{j+1}} \cdots \sigma_i^{-1}) (\sigma_{1}^{-1} \cdots \sigma_{j-1}^{-1} \sigma_j^{-1} \sigma_{j+1}^{-1} \cdots \sigma_{i-1}^{-1}) \\
        & \stackrel{\text{braid rel.}}{=} (\sigma_{i-1} \cdots \sigma_1) (\sigma_i \cdots \sigma_2) p_1 (\sigma_2^{-1} \cdots \sigma_{j-1}^{-1} \sigma_j^{-1} \underline{\sigma_{j+2}} \sigma_{j+1}^{-1} \sigma_{j+2}^{-1} \cdots \sigma_i^{-1}) (\sigma_{1}^{-1} \cdots \sigma_{j-1}^{-1} \sigma_j^{-1} \sigma_{j+1}^{-1} \cdots \sigma_{i-1}^{-1})\\
        & \stackrel{\text{shift}}{=} (\sigma_{i-1} \cdots \sigma_1) (\sigma_i \cdots \sigma_2) p_1 \, \sigma_{j+2} (\sigma_2^{-1} \cdots \sigma_{j-1}^{-1} \sigma_j^{-1} \sigma_{j+1}^{-1} \sigma_{j+2}^{-1} \cdots \sigma_i^{-1}) (\sigma_{1}^{-1} \cdots \sigma_{j-1}^{-1} \sigma_j^{-1} \sigma_{j+1}^{-1} \cdots \sigma_{i-1}^{-1})
    \end{split}
\end{equation*}
Analogously, the right-hand side yields:
\[
(\sigma_{i-1} \cdots \sigma_1) (\sigma_i \cdots \sigma_2) \sigma_{j+2} \, p_1 (\sigma_2^{-1} \cdots \sigma_i^{-1}) (\sigma_{1}^{-1} \cdots \sigma_{i-1}^{-1})
\]
So removing the conjugating expressions we obtain the same relation of the previous case:
\begin{equation} \label{eq4W}
    p_1 \, \sigma_{j+2} = \sigma_{j+2} \, p_1
\end{equation}
Since the admissible indices for $j<i-1$ are $j = 1, \dots , i-2$, we have that $j+2$ can assume all values from $3$ to $i$. Hence, Eqs.~\ref{eq4W1} and~\ref{eq4W} establish relations 5) of the statement. 

\smallbreak
\noindent \textit{Relations 6): $p_i \, \sigma_i = \sigma_i p_i$.} They are a bit mischievous, because not only the clasp needs to be brought on to subtend on the first and second strands, but $\sigma_i$ too.
As always, the right and left-hand sides are going to be symmetric, so we will illustrate the left one only.
\begin{equation*}
    \begin{split}
        p_i \sigma_i = & \, (\sigma_{i-1} \sigma_i)(\sigma_{i-2} \sigma_{i-1}) \cdots (\sigma_1 \sigma_2) \, p_1 \, (\sigma_1 \sigma_2)^{-1} \cdots (\sigma_{i-2} \sigma_{i-1})^{-1} (\sigma_{i-1} \sigma_i)^{-1} \sigma_i \\ 
        = & \, (\sigma_{i-1} \sigma_i)(\sigma_{i-2} \sigma_{i-1}) \cdots (\sigma_1 \sigma_2) \, p_1 \, (\sigma_2^{-1} \sigma_1^{-1}) \cdots (\sigma_{i-1}^{-1} \sigma_{i-2}^{-1}) \underline{(\sigma_i^{-1} \sigma_{i-1}^{-1}) \sigma_i} \\
        \stackrel{\text{braid rel.}}{=} & \, (\sigma_{i-1} \sigma_i)(\sigma_{i-2} \sigma_{i-1}) \cdots (\sigma_1 \sigma_2) \, p_1 \, (\sigma_2^{-1} \sigma_1^{-1}) \cdots \underline{(\sigma_{i-1}^{-1} \sigma_{i-2}^{-1}) \sigma_{i-1}} (\sigma_i^{-1} \sigma_{i-1}^{-1}) \\ 
        \stackrel{\text{iterating}}{=} & \, \cdots = (\sigma_{i-1} \sigma_i)(\sigma_{i-2} \sigma_{i-1}) \cdots (\sigma_1 \sigma_2) \, p_1 \, \sigma_{1} (\sigma_2^{-1} \sigma_1^{-1}) \cdots  (\sigma_{i-1}^{-1} \sigma_{i-2}^{-1}) (\sigma_i^{-1} \sigma_{i-1}^{-1})  \\ 
    \end{split}
\end{equation*}

Reiterating, the elementary braiding eventually comes next to $p_1$.
The right-hand side is manipulated with the analogous steps. 
Comparing with the left-hand side and removing the conjugating expressions, the relation boils down to relation 6) of the statement:
\[
p_1 \, \sigma_1 = \sigma_1 \, p_1.
\]

\smallbreak
\noindent \textit{Relations 7): $p_i \, \sigma_{i+1} \, \sigma_i = \sigma_{i+1} \, \sigma_i \, p_{i+1}$.} We will work on the relation's more convenient writing $p_i \, = \, \sigma_{i+1} \, \sigma_i \, p_{i+1} \, \sigma_i^{-1} \, \sigma_{i+1}^{-1}$. By rewriting $p_i$ and $p_{i+1}$ as conjugates of $p_1$, the relation becomes:

$$(\sigma_{i-1} \sigma_i)(\sigma_{i-2} \sigma_{i-1}) \cdots (\sigma_1 \sigma_2) \, p_1 \, (\sigma_1 \sigma_2)^{-1} \cdots (\sigma_{i-2} \sigma_{i-1})^{-1} (\sigma_{i-1} \sigma_i)^{-1} \, = \, (\sigma_{i+1} \sigma_i) (\sigma_i \sigma_{i+1}) (\sigma_{i-1} \sigma_i) \cdots (\sigma_1 \sigma_2) \, p_1 \, (\sigma_1 \sigma_2)^{-1} \cdots (\sigma_{i-1} \sigma_i)^{-1} (\sigma_i \sigma_{i+1})^{-1} (\sigma_{i+1} \, \sigma_i)^{-1}$$

We will operate on each hand side separately, and then reunite them for a final simplification. The left-hand side is evidently equal to
$$(\sigma_{i-1} \cdots \sigma_1) (\sigma_i \cdots \sigma_2) (\sigma_{i+1} \cdots \sigma_3) \, p_1 \, (\sigma_{i+1} \cdots \sigma_3)^{-1} (\sigma_i \cdots \sigma_2)^{-1} (\sigma_{i-1} \cdots \sigma_1)^{-1}$$
which simply means that we have rearranged the writing and that, together with the clasp $p_i$, we have also dragged along the $i+1$-th strand with us to lie at the left (in position $3$) of the bonded braid.

We slide also the conjugation by $(\sigma_{i+1} \sigma_i)$ in the right-hand side around $p_1$. On the right-hand side, we work on $w = \underline{(\sigma_{i+1} \sigma_i) (\sigma_i \sigma_{i+1}) (\sigma_{i-1} \sigma_i)} (\sigma_{i-2} \sigma_{i-1}) \cdots (\sigma_1 \sigma_2)$, which is the word at the left of $p_1$. The same results hold by symmetry on the word at its right, $w^{-1} = (\sigma_1 \sigma_2)^{-1} \cdots (\sigma_{i-1} \sigma_i)^{-1} (\sigma_i \sigma_{i+1})^{-1} (\sigma_{i+1} \, \sigma_i)^{-1}$. 

Focus on the underlined subword of $w$: it can be modified as following.

\begin{equation} \label{algorithm7}
    \begin{split}
        \sigma_{i+1} \sigma_i \sigma_i \sigma_{i+1} \sigma_{i-1} \sigma_i \, & = \, \sigma_{i+1} \sigma_i \, \sigma_{i-1} \sigma_{i-1}^{-1} \, \sigma_i \underline{\sigma_{i+1}} \, \underline{\sigma_{i-1}} \sigma_i \\
        & \stackrel{\text{switch}}{=} \, \sigma_{i+1} \sigma_i \, \sigma_{i-1} \underline{\sigma_{i-1}^{-1} \, \sigma_i \sigma_{i-1}} \, \sigma_{i+1} \sigma_i \\
        & = \, \sigma_{i+1} \sigma_i \, \sigma_{i-1} \sigma_i \, \sigma_{i-1} \underline{\sigma_i^{-1} \, \sigma_{i+1} \sigma_i} \\
        & = \, \sigma_{i+1} \underline{\sigma_i \, \sigma_{i-1} \sigma_i} \, \sigma_{i-1} \sigma_{i+1} \sigma_i \sigma_{i+1}^{-1}\\
        & = \, \underline{\sigma_{i+1}} \, \underline{\sigma_{i-1}} \sigma_i \underline{\sigma_{i-1}  \, \sigma_{i-1}} \, \underline{\sigma_{i+1}} \sigma_i \sigma_{i+1}^{-1}\\
        & \stackrel{\text{switch}}{=} \, \sigma_{i-1} \underline{\sigma_{i+1} \sigma_i \sigma_{i+1}} \sigma_{i-1}  \, \sigma_{i-1} \sigma_i \sigma_{i+1}^{-1}\\
        & = \, (\sigma_{i-1} \sigma_i \sigma_{i+1}) (\sigma_i \sigma_{i-1}) \, (\sigma_{i-1} \sigma_i) \sigma_{i+1}^{-1}\\
    \end{split}
\end{equation}

This final writing can be plugged back in the word $w$, which now looks like 
$$(\sigma_{i-1} \sigma_i \sigma_{i+1}) \, (\sigma_i \sigma_{i-1}) \, (\sigma_{i-1} \sigma_i) \, \sigma_{i+1}^{-1} \, (\sigma_{i-2} \sigma_{i-1}) \, \cdots \, (\sigma_1 \sigma_2) \, = \, (\sigma_{i-1} \sigma_i \sigma_{i+1}) \, \underline{(\sigma_i \sigma_{i-1}) \, (\sigma_{i-1} \sigma_i) \, (\sigma_{i-2} \sigma_{i-1})} \, \cdots \, (\sigma_1 \sigma_2) \, \sigma_{i+1}^{-1}$$

The algorithm can be repeated on the underlined subword, and then reiterated. In the end, after having rearranged the writing, $w$ will read:

$$(\sigma_{i-1} \cdots \sigma_1) \, (\sigma_i \cdots \sigma_2) \, (\sigma_{i+1} \cdots \sigma_3) \, (\sigma_2 \, \sigma_2 \, \sigma_1 \, \sigma_2 ) \, (\sigma_3^{-1} \cdots \sigma_{i+1}^{-1})$$

Recombining with $p_1$ and $w^{-1}$, the right-hand side of relation 7 has been modified in 

\begin{equation*}
    \begin{split}
    (\sigma_{i-1} \cdots \sigma_1) & \, (\sigma_i \cdots \sigma_2) \, (\sigma_{i+1} \cdots \sigma_3) \, (\sigma_2 \, \sigma_1 \, \sigma_1 \, \sigma_2 ) \, \underline{(\sigma_3^{-1} \cdots \sigma_{i+1}^{-1})} \, p_1 \, \underline{(\sigma_{i+1} \, \cdots \, \sigma_3)} \cdot \\ \cdot & (\sigma_2 \, \sigma_1 \, \sigma_1 \, \sigma_2 )^{-1} \, (\sigma_{i+1} \cdots \sigma_3)^{-1} \,  (\sigma_i \cdots \sigma_2)^{-1} \, (\sigma_{i-1} \cdots \sigma_1)^{-1}
\end{split}
\end{equation*}

The two underlined pieces cancel each other. If we now compare the two hand sides, we notice the same structure: a threefold conjugation of a central core, with clasp $p_1$. By placing the hand sides on both ends of an equation, and removing the same conjugating words, we obtain

$$p_1 \, = \, (\sigma_2 \, \sigma_1 \, \sigma_1 \, \sigma_2 ) \, p_1 \, (\sigma_2 \, \sigma_1 \, \sigma_1 \, \sigma_2 )^{-1} \implies p_1 \, (\sigma_2 \, \sigma_1 \, \sigma_1 \, \sigma_2 ) \, = \, (\sigma_2 \, \sigma_1 \, \sigma_1 \, \sigma_2 ) \, p_1$$

which is relation 7) of the statement.

\smallbreak
\noindent \textit{Relations 8): $\sigma_i \, \sigma_{i+1} \, p_i = p_{i+1} \, \sigma_i \, \sigma_{i+1}$.} These relations are trivialised by the chosen rewriting in terms of $p_1$. In fact, by re-expressing relation 8) as $p_i \, = \, (\sigma_i \, \sigma_{i+1})^{-1} \, p_{i+1} \, (\sigma_i \, \sigma_{i+1})$, the rewriting looks like 
\begin{equation*}
    \begin{split}
        (\sigma_{i-1} \sigma_i) & (\sigma_{i-2} \sigma_{i-1}) \cdots (\sigma_1 \sigma_2) \, p_1 \, (\sigma_1 \sigma_2)^{-1} \cdots (\sigma_{i-2} \sigma_{i-1})^{-1} (\sigma_{i-1} \sigma_i)^{-1} \, \\ & = \underline{(\sigma_i \, \sigma_{i+1})^{-1} \, (\sigma_i \sigma_{i+1})}(\sigma_{i-1} \sigma_i) \cdots (\sigma_1 \sigma_2) \, p_1 \, (\sigma_1 \sigma_2)^{-1} \cdots (\sigma_{i-1} \sigma_i)^{-1} \underline{(\sigma_i \sigma_{i+1})^{-1} \, (\sigma_i \, \sigma_{i+1})}
    \end{split}
\end{equation*}
Where the underlined bits are self-cancellative expressions. The remaining equation is tautological, thus is removed from the relation set.

\smallbreak

\smallbreak
 \noindent \textit{Relations 9): $p_i = \sigma_i^2$.} The left-hand side can be manipulated as in relations 6) to have 
$$p_i = (\sigma_{i-1} \sigma_i)(\sigma_{i-2} \sigma_{i-1}) \cdots (\sigma_1 \sigma_2) \, p_1 \, (\sigma_1 \sigma_2)^{-1} \cdots (\sigma_{i-2} \sigma_{i-1})^{-1} (\sigma_{i-1} \sigma_i)^{-1}$$

The right-hand side of relations 9), $\sigma_i^2$, can be multiplied on the right by the identity expression \[(\sigma_{i-1} \sigma_i)(\sigma_{i-2} \sigma_{i-1}) \cdots (\sigma_1 \sigma_2) (\sigma_2^{-1} \sigma_1^{-1}) \cdots  (\sigma_{i-1}^{-1} \sigma_{i-2}^{-1}) (\sigma_i^{-1} \sigma_{i-1}^{-1})\]  so that $\sigma_i^2$ can slide in via the iterated braid relation, becoming in the end the core $\sigma_1^2$ of the expression:
\[
(\sigma_{i-1} \sigma_i)(\sigma_{i-2} \sigma_{i-1}) \cdots (\sigma_1 \sigma_2) \, \sigma_1^2 \, (\sigma_2^{-1} \sigma_1^{-1}) \cdots  (\sigma_{i-1}^{-1} \sigma_{i-2}^{-1}) (\sigma_i^{-1} \sigma_{i-1}^{-1})
\]

Removing the conjugating expressions from both hand sides, the relations reduce to
\[
p_1= \sigma_1^2 \iff p_1 \, \sigma_1^{-1} = \sigma_1
\]
which is relation 8) of Theorem~\ref{thm:irredundant}.
 The proof of Theorem~\ref{thm:irredundant} is now concluded. 
\end{proof}


\section{The tight bonded braid monoid and the singular braid monoid}

The presentation for the bonded braid monoid in Definition~\ref{BB} can be brought to a \textit{tight form}, as shown in~\cite{DKL}:

\begin{proposition}[\cite{DKL}] \label{tBB} 
The bonded braid monoid $BB_n$ admits a reduced presentation with a reduced set of generators $\sigma_i, \sigma_i^{-1}, b_i$, for $1 \leq i \leq n-1$, where $b_i := b_{i, i+1}$, and relations: 
    \[
    \begin{array}{crcll}
        1) & \sigma_i \sigma_j \, & = & \, \sigma_j \sigma_i & \text{for} \hspace{2mm} 1 \leq i <j \leq n-1, \, |i-j|>1 \vspace{1mm} \\
        2) & \sigma_i \sigma_{i+1} \sigma_i \, & = & \, \sigma_{i+1} \sigma_i \sigma_{i+1} & \text{for } 1 \leq i \leq n-2 \vspace{1mm} \\
        3) & b_i b_j \, & = & \,  b_j b_i & \text{for}\hspace{2mm} 1 \leq i <j \leq n-1, \,  |i-j|>1 \vspace{1mm} \\
        4) & b_i \sigma_j^{\pm 1} \, & = &  \, \sigma_{j}^{\pm1} b_i & \text{for} \hspace{2mm} 1 \leq i <j \leq n-1, \, |i-j|>1 \vspace{1mm} \\
        5) & b_i \sigma_i^{\pm 1} \, & = & \, \sigma_i^{\pm 1} b_i & \text{for } 1 \leq i \leq n-1 \vspace{1mm} \\
        6) & b_i \sigma_{i+1} \sigma_i \, & = & \,  \sigma_{i+1}\sigma_i b_{i+1} & \text{for } 1 \leq i \leq n-2 \vspace{1mm} \\
        7) & \sigma_i \sigma_{i+1} b_i \, & = & \, b_{i+1} \sigma_i \sigma_{i+1} & \text{for } 1 \leq i \leq n-2 \vspace{1mm} \\
        \end{array}
    \]
\end{proposition}

\begin{figure}[H]
\begin{center} 
    \includegraphics[width=15.5cm]{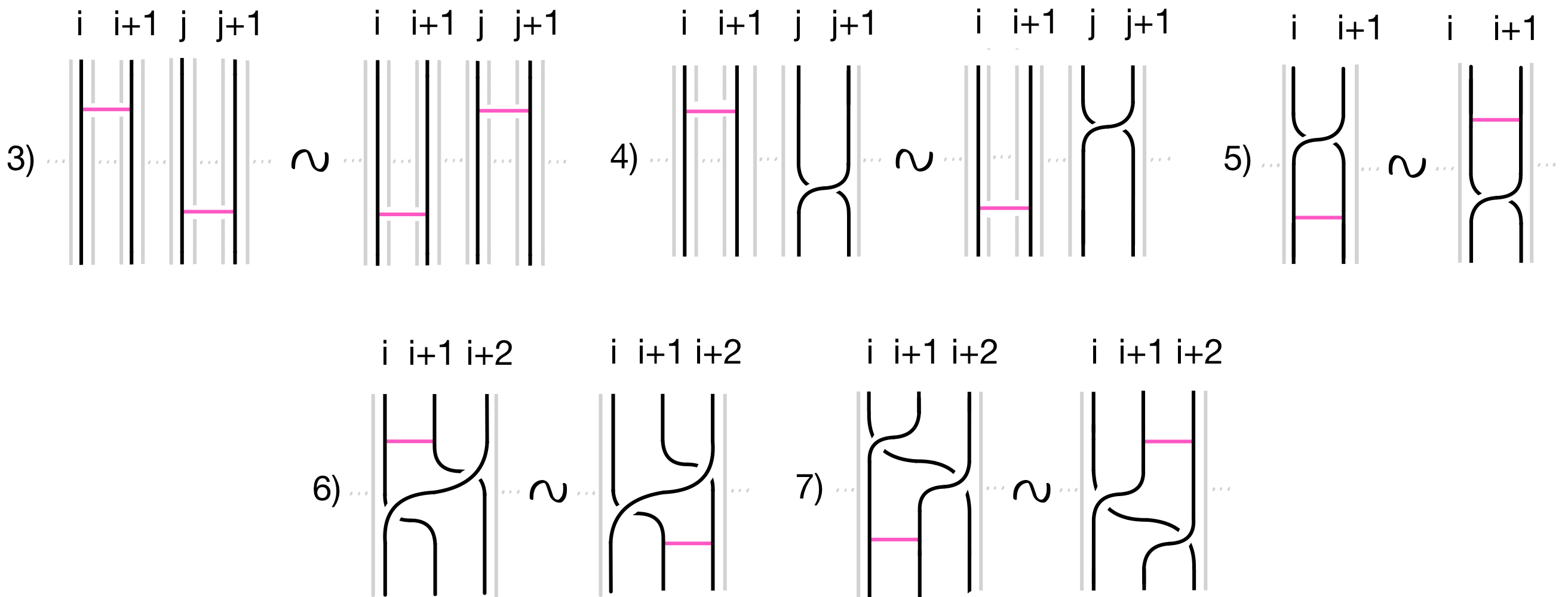}
\end{center}
\caption{Relations 3) to 7) of the tight bonded braid monoid.}
\label{figtightbonded}
\end{figure}

This presentation is also referred to as the \textit{tight bonded braid monoid presentation}. It is named \textit{tight} because in this presentation a sufficient set of generators comprises just the elementary braids and their inverses, and the \textit{tight} bonds $b_i$, which subtend on adjacent strands. They can be seen as a consequence of pulling standard bonds tight, via relations $5.1) - 5.4)$ in Definition~\ref{BB}.

\smallbreak
In the tight presentation, the connection of the bonded braid monoid with the singular braid monoid is more transparent. Singular braids are braids including a special kind of crossing, called \textit{singular}, in which the two strands actually intersect and switch disposition, see \cite{Baez, Birman}.

\definition [\cite{Baez, Birman}] \label{singular} The \textit{singular braid monoid} $SB_n$ is generated by the usual braid generators $\sigma_1^{\pm 1}, \, \dots , \, \sigma_{n-1}^{\pm 1}$ together with singular crossing generators $\tau_1 , \, \dots , \, \tau_{n-1}$, subject to the following relations
\[
\begin{array}{lrcll}
    1) & \sigma_i \sigma_j & = & \sigma_j \sigma_i & \text{for } |i-j|>1 \\
     2) & \sigma_i \sigma_{i+1} \sigma_i & = & \sigma_{i+1} \sigma_i \sigma_{i+1} & \text{for all } i\\
     3) & \tau_i \tau_j & = & \tau_j \tau_i & \text{for } |i-j|>1 \\
     4) & \tau_i \sigma_j & = & \sigma_j \tau_i & \text{for } |i-j|>1 \\
     5) & \tau_i \sigma_i & = & \sigma_i \tau_i & \text{for all } i\\
     6) & \tau_i \sigma_{i+1} \sigma_i & = & \sigma_{i+1} \sigma_i \tau_{i+1} & \text{for all } i \\
     7) & \sigma_i \sigma_{i+1} \tau_i & = & \tau_{i+1} \sigma_i \sigma_{i+1} & \text{for all } i
\end{array}
\]

\begin{figure}[H]
\begin{center} 
    \includegraphics[width=14cm]{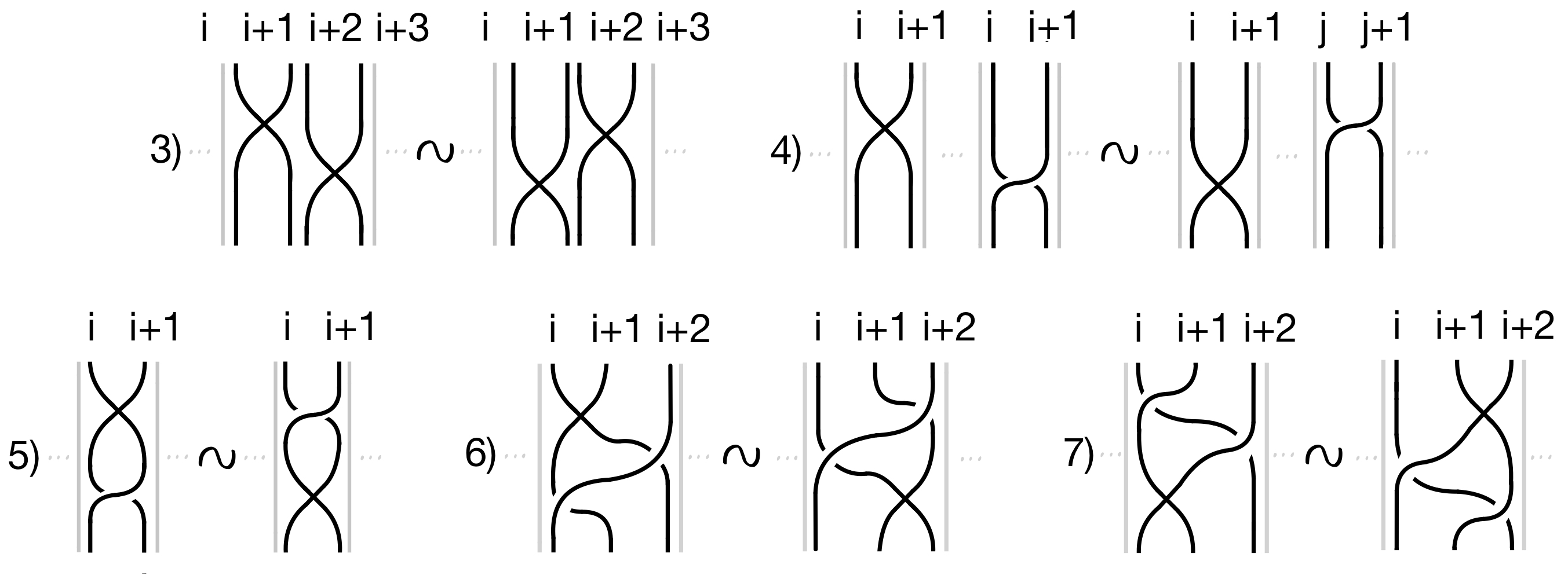}
\end{center}
\caption{Relations 3) to 7) of the singular braid monoid.}
\label{figsingbraids}
\end{figure}

\begin{remark} \label{singisotight}
    As observed in \cite{DKL} there is an isomorphism $s \colon SB_n \to B_n$ which mapps $\tau_i$ to $b_i$ and is the identity on elementary braidings. We can see this in the tight presentation of $BB_n$ in Proposition~\ref{tBB} compared with Definition~\ref{singular}.
\end{remark}

\noindent We can sharpen the presentation in Proposition~\ref{tBB}: using iteratively relations 6) and 7), we have that $$b_{i+1} = (\sigma_{i+1} \sigma_i)^{-1} b_i (\sigma_{i+1} \sigma_i) \qquad , \, \ldots,  \qquad b_{i+1} = (\sigma_{i+1} \sigma_i)^{-1} \ldots (\sigma_2 \sigma_1)^{-1} \, b_1 \, (\sigma_2 \sigma_1) \ldots (\sigma_{i+1} \sigma_i)$$
\smallbreak
Then one comes to the following irredundant presentation for the tight bonded braid monoid.

\begin{proposition}[\cite{DKL}] \label{thm:tightbonded} 
    The tight bonded braid monoid $BB_n$ admits the following irredundant presentation: it is generated by the classical braid generators $\sigma_1, \, \dots, \, \sigma_{n-1}$, their inverses, and a single bond generator $b_1$, subject to the following relations: 
    \[
    \begin{array}{crcll}
        1) &  \sigma_i \sigma_j \, & = & \, \sigma_j \sigma_i& \text{for} \hspace{2mm} 1 \leq i <j \leq n-1, \, |i-j|>1 \vspace{1mm} \\ 
        2) &  \sigma_i \sigma_{i+1} \sigma_i \, & = &  \, \sigma_{i+1} \sigma_i \sigma_{i+1} & \text{for } 1 \leq i \leq n-2 \vspace{1mm} \\
        3) & b_1 (\sigma_2 \sigma_1 \sigma_3 \sigma_2)^{-1}  b_1 (\sigma_2 \sigma_1 \sigma_3 \sigma_2) \, & = & (\sigma_2 \sigma_1 \sigma_3 \sigma_2)^{-1}  b_1 (\sigma_2 \sigma_1 \sigma_3 \sigma_2) b_1 & \text{for}\hspace{2mm} 1 \leq i <j \leq n-1, \,  |i-j|>1 \vspace{1mm} \\
        4) & b_1 \sigma_j \, & = & \, \sigma_j b_1 & \text{for}\hspace{2mm} j>2 \vspace{1mm} \\
        5) &  b_1 \sigma_1 \, & = & \, \sigma_1 b_1  & \\
        6) & b_1 (\sigma_2 \sigma_1) (\sigma_1 \sigma_2) \, & = & \,  (\sigma_2 \sigma_1) (\sigma_1 \sigma_2) b_1 & \vspace{1mm}
        \end{array}
    \]
\end{proposition}

\begin{figure}[H]
\begin{center} 
    \includegraphics[width=17cm]{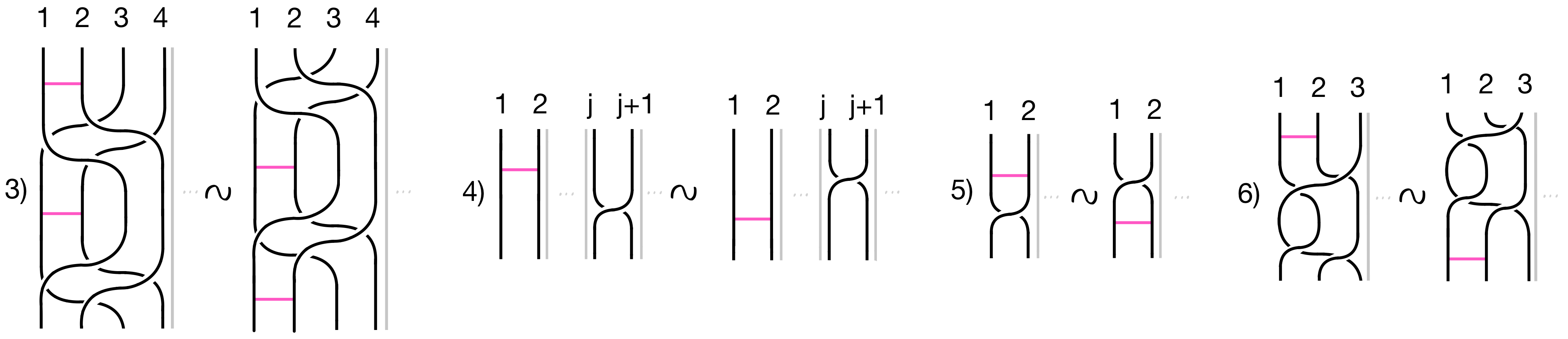}
\end{center}
\caption{Relations 3) to 6) of the tight bonded braid monoid on one bond generator (Proposition~\ref{thm:tightbonded}).}
\label{figtightbonded1}
\end{figure}

\begin{remark} \label{irredsing}
    As a corollary to the isomorphism $SB_n \to BB_n$, the analogous statement of Proposition~\ref{thm:tightbonded} holds for the singular braid monoid $SB_n$, which has an irreduntant presentation as above by replacing $b_1$ with~$\tau_1$.
\end{remark}


\section{Homomorphisms of the tight bonded braid monoid and the singular braid monoid} 

\subsection{Homomorphisms of the tight bonded braid monoid}

The homomorphisms in Theorems~\ref{mainthm} and~\ref{mainthm2} in Section~\ref{Representations} can be reformulated for the tight bonded braid monoid.

\begin{theorem} \label{Tightphi} 
The restriction of $\Phi$ to elementary braidings and tight bonds $b_i$, which maps $\sigma_i^{\pm1} \stackrel{\Phi}{\mapsto} \sigma_i^{\pm1}$ and $b_i \stackrel{\Phi}{\mapsto} p_i$, defines a homomorphism of monoids from $BB_n$ to $B_n$. The homomorphism is surjective but not injective. 
\end{theorem}

\begin{proof} 
The result can be derived from Theorem~\ref{mainthm}, since the proposed map is a restriction of $\Phi$. However, a direct proof can be obtained by checking that $\Phi$ maps \textit{verbatim} the relations for the tight bonded braid monoid in Proposition~\ref{tBB} to relations for the clasps in the presentation of $B_n$ in Theorem~\ref{braidpreswclasp}. Thus, the map $\Phi$ defines a homomorphism of monoids. 

Relations 1)-2) of Proposition~\ref{tBB} are not affected by this or any choice for the image of $b_i$ as opposed to relations 3-7), which need to be checked. As expected, on the braid side, we will use the clasp presentation of Theorem~\ref{braidpreswclasp}.

\smallbreak
\noindent \textit{Relations 3): $b_i \, b_j \, = \,  b_j \, b_i \to p_i \, p_j \, = \, p_j \, p_i$.}
Relation 3) rewrites to $p_i \, p_j = p_j \, p_i$, which can be contracted to relations 3) of Theorem~\ref{braidpreswclasp}.

\smallbreak
\noindent \textit{Relations 4): $b_i \, \sigma_j^{\pm 1} \, = \,  \sigma_j^{\pm 1} \, b_i \to p_i \, \sigma_j^{\pm 1} \, = \, \sigma_j^{\pm 1} \, p_i$.} This is immediately relations 5) of Theorem~\ref{braidpreswclasp}.

\smallbreak
\noindent\textit{Relations 5): $b_i \, \sigma_i^{\pm 1} \, = \,  \sigma_i^{\pm 1} \, b_i \to p_i \, \sigma_i^{\pm 1} \, = \, \sigma_i^{\pm 1} \, p_i$.} This is immediately relations 6) of Theorem~\ref{braidpreswclasp}.

\smallbreak
\noindent \textit{Relations 6): $b_i \, \sigma_{i+1} \, \sigma_i \, = \,  \sigma_{i+1} \, \sigma_i \,  b_{i+1} \to p_i \, \sigma_{i+1} \, \sigma_i \, = \,  \sigma_{i+1} \, \sigma_i \, \, p_{i+1}$.} This is immediately relations 7) of Theorem~\ref{braidpreswclasp}.

\smallbreak
\noindent \textit{Relations
7): $ \sigma_{i+1} \, \sigma_i \, b_i \, =  \, b_{i+1} \, \sigma_{i+1} \, \sigma_i \, \to  \sigma_{i+1} \, \sigma_i \, p_i \, = \, p_{i+1} \, \sigma_{i+1} \, \sigma_i $.} This is immediately relations 8) of Theorem~\ref{braidpreswclasp}.

The second assertion of the Theorem follows by considerations analogous to those adopted in Theorems~\ref{mainthm},~\ref{mainthm2}: since the map restricted to non-bonded braids is the identity, it is surjective. Moreover, as $b_i$ is mapped to $\sigma_i^2$, we have $b_i \, \sigma_i^{-2} \in \text{ker}(\Phi)$, thus $\Phi$ is not injective.
\end{proof}


\begin{remark} \label{relation4remark}
The only relations of Theorem~\ref{braidpreswclasp} in which no relation of the tight bonded braid monoid presentation is mapped, are relation 9) and relation 4):
$ (\sigma_i \, p_{i+1} \, \sigma_i) \, p_{i+1} = p_{i+1} \, (\sigma_i \, p_{i+1} \, \sigma_i)$, which is obtained by combining relations 7), 8) and 9) of Theorem~\ref{braidpreswclasp}, recall Remark~\ref{claspremarkrel4}. Relation 4) of Theorem~\ref{braidpreswclasp} was derived from the cyclic relation Eq.~\ref{cyclicrels} in the pure braid group. So, due to the absence of relation 9), the cyclic bonded relation $\sigma_i \, b_{i+1} \, \sigma_i \, b_{i+1} = b_{i+1} \, \sigma_i \, b_{i+1} \, \sigma_i$, which would  be mapped by $\Phi$ to relation 4), is not valid in the bonded braid monoid. See left-hand side of Figure~\ref{figtightbonded1}.
   
However, we have the mixed relation (consisting of both bonds and clasps): $\sigma_i \, p_{i+1} \, \sigma_i \, b_{i+1} = b_{i+1} \, \sigma_i \, p_{i+1} \, \sigma_i$    which is mapped by $\Phi$ to relation 4) of Theorem~\ref{braidpreswclasp}. In this mixed relation a tight bond passes through a loop formed by a free strand around its two subtending strands. See right-hand side of Figure~\ref{figtightbonded1}. 

A similar instance for the map $\Phi \colon BB_n \to B_n$ in Theorem~\ref{mainthm} was remarked in Remarks~\ref{cyclicnotbond}: the only relations of $B_n$ not in the image set via $\Phi$ of relations of  $BB_n$ involving exclusively bonds were the cyclic relations. 
\end{remark}

\begin{figure}[H]
\begin{center} 
    \includegraphics[width=9cm]{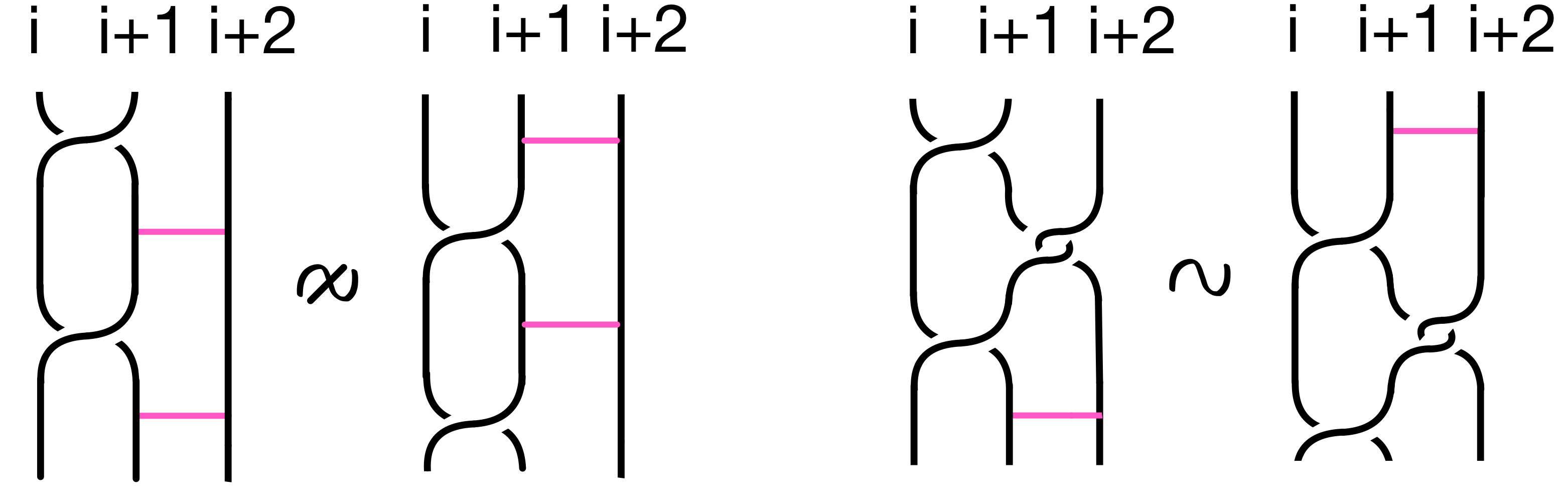}
\end{center}
\caption{The analogue of relation 4) of Theorem~\ref{braidpreswclasp} is not valid in the tight bonded braid monoid. However it is valid in a mixed version.}
\label{figtightbonded1}
\end{figure}

We can further specialise our result to cater for the tight bonded braid monoid presentation with only one bond generator (recall Proposition~\ref{thm:tightbonded}).

\begin{theorem} \label{Tight1phi}
The restriction of $\Phi$ to the first bond generator $b_1$, which maps $b_1 \stackrel{\Phi}{\mapsto} p_1$ defines a homomorphism of monoids from $BB_n$ to $B_n$. The homomorphism is surjective but not injective. 

\end{theorem}

\begin{proof} 
The result can be derived from Theorem~\ref{Tightphi}, since the proposed map is a restriction of $\Phi$ to $b_1$. However, a direct proof can be obtained by checking that the restriction maps \textit{verbatim} the relations for the tight bonded braid monoid in Proposition~\ref{thm:tightbonded} to relations for the clasps in the presentation of $B_n$ in Proposition~\ref{thm:irredundant}, thus defining a homomorphism of monoids. 

Relations 1)-2) of Proposition~\ref{thm:tightbonded} identically mapped by $\Phi$ as opposed to relations 3-6), that need to be checked. On the braid side we will use the clasp presentation of Proposition~\ref{thm:irredundant}.

\smallbreak
\noindent \textit{Relation 3):} \begin{equation*}
    \begin{split}
        b_1 \, (\sigma_2 \sigma_1 \sigma_3 \sigma_2)^{-1} \, b_1 \, (\sigma_2 \sigma_1 \sigma_3 \sigma_2) & = (\sigma_2 \sigma_1 \sigma_3 \sigma_2)^{-1} \, b_1 \, (\sigma_2 \sigma_1 \sigma_3 \sigma_2) \, b_1 \\ \mapsto p_1 \, (\sigma_2 \sigma_1 \sigma_3 \sigma_2)^{-1} \, p_1 \, (\sigma_2 \sigma_1 \sigma_3 \sigma_2) & = (\sigma_2 \sigma_1 \sigma_3 \sigma_2)^{-1} \, p_1 \, (\sigma_2 \sigma_1 \sigma_3 \sigma_2) \, p_1. 
    \end{split}
\end{equation*} This relation is perfectly mapped in relations 3) of Theorem~\ref{thm:irredundant}.

\smallbreak
\noindent \textit{Relations 4): $b_1 \, \sigma_j = \sigma_j \, b_1 \mapsto p_1 \, \sigma_j = \sigma_j \, p_1$.} These relations are perfectly mapped in relations 5) of Theorem~\ref{thm:irredundant}.

\smallbreak
\noindent \textit{Relation 5): $b_1 \, \sigma_1 = \sigma_1 \, b_1 \mapsto p_1 \, \sigma_1 = \sigma_1 \, p_1$.} This relation is perfectly mapped in relations 6) of Theorem~\ref{thm:irredundant}.

\smallbreak
\noindent \textit{Relation 6): $b_1 \, (\sigma_2 \sigma_1) ( \sigma_1 \sigma_2) = (\sigma_2 \sigma_1) ( \sigma_1 \sigma_2) \, b_1 \mapsto p_1 \, (\sigma_2 \sigma_1) ( \sigma_1 \sigma_2) = (\sigma_2 \sigma_1) ( \sigma_1 \sigma_2) \, p_1$.} This relation is perfectly mapped in relations 7) of Theorem~\ref{thm:irredundant}.

\noindent \smallbreak
The second assertion of the Theorem follows by considerations analogoues to those adopted in Theorems~\ref{mainthm}, ~\ref{mainthm2}. Surjectivity follows as usual, since the restriction of $\Phi$ to braids is the identity. The map is not injective, as $b_1 \, \sigma_1^{-2} \in \text{ker} (\Phi)$. 
\end{proof}

\begin{remark} Like in Remark~\ref{relation4remark}, only one relation in the presentation of Proposition~\ref{thm:irredundant} is not in the image set of any relation in the source presentation by the restriction of $\Phi$, namely relation 4): $(\sigma_2 p_1 \sigma_2^{-1}) \sigma_1 (\sigma_2 p_1 \sigma_2^{-1}) \sigma_1  = \sigma_1 (\sigma_2 p_1 \sigma_2^{-1}) \sigma_1 (\sigma_2 p_1 \sigma_2^{-1})$. Again, we have a mixed relation in the tight bonded braid monoid which is mapped onto it: that is, $(\sigma_2 b_1 \sigma_2^{-1}) \sigma_1 (\sigma_2 p_1 \sigma_2^{-1}) \sigma_1 = \sigma_1 (\sigma_2 p_1 \sigma_2^{-1}) \sigma_1 (\sigma_2 b_1 \sigma_2^{-1})$. This relation for $B_n$ with one clasp generator was derived by relation 4) of the presentation for $B_n$ with $n-1$ clasp generators.
\end{remark}

\begin{remark} \label{Phik}
    Theorems~\ref{Tightphi} and ~\ref{Tight1phi} accept their own version of Theorem~\ref{mainthm2}: the restrictions of $\Phi_p$, $p \in \mathbb{Z}$, to elementary braidings and tight bonds (resp. the first bond) define surjective but not injective homomorphisms from the tight bonded braid monoid (resp. on one bond generator) to the braid group.
\end{remark}

\subsection{Homomorphisms of the singular braid monoid}

By virtue of the isomorphism $s$ from the singular braid monoid to the tight bonded braid monoid in Remark~\ref{singisotight}), we have completely analogous results for $SB_n$. Namely, by composing $\Phi \colon BB_n \to B_n$ with $s \colon SB_n \to BB_n$ we obtain the following result:

\begin{theorem} \label{mapsingular}
    The map $s\Phi \colon SB_n \to B_n$, where $\sigma_i^{\pm1} \stackrel{s\Phi}{\mapsto} \sigma_i^{\pm1}$ and $\tau_i \stackrel{s\Phi}{\mapsto} p_i$ defines a homomorphism of monoids from $SB_n$ to $B_n$, mapping the singular generators to the clasp generators. The homomorphism is surjective but not injective. 
    Moreover, $s\Phi$ can be restricted to the elementary braidings and the first singular crossing generator $\tau_1 \stackrel{s\Phi}{\mapsto} p_1$. This restriction defines a surjective but not injective homomorphism of monoids from $SB_n$ to $B_n$. 
    In addition, for every integer $p \in \mathbb{Z}$ we define a non-injective epimorphism of monoids $s\Phi_p \colon SB_n \to B_n$ by imposing $s\Phi_p$ to be the identity on elementary braidings and $s\Phi_p(\tau_i) = \sigma_i^{2p}$, which can restrict to only the first singular crossing $\tau_1$. 
\end{theorem}

\begin{proof}
    The statements follow by virtue of the isomorphism $s \colon SB_n \to BB_n$ in Remark~\ref{singisotight}, and Theorem~\ref{Tight1phi} and Remark~\ref{Phik}. 
\end{proof}

\begin{remark}
    The maps $s\Phi$, and in general $s\Phi_p$, are a new family of maps from the singular braid monoid to the braid group, whose distinctive feature is to not switch the strands involved.
    The path to singular braids is paved with good maps to classical crossings. A naive map consists in forgetting the singularity of the braid and replacing all singular crossings with classical crossings, either positive or negative. Albeit trivial, this is a well-defined homomorphism from singular braids to classical braids, since all relations in $SB_n$ carry through to $B_n$. 
    
    The crossing replacement, and in general the theory of singular knots/links, are related to the theory of Vassiliev invariants, cf. for example \cite{chm}. For a singular link $L$, one applies the \textit{Vassiliev skein relation}: $\mathcal{V}(L_{\times}) = \mathcal{V}(L_+) - \mathcal{V}(L_-)$, whereby a singular crossing in a diagram is replaced by the difference of a positive and a negative crossing. Set $\mathcal{L}$ as the set of all classical links and $\mathcal{S}$ as the set of all singular links: then, the map induced by the Vassiliev skein relation goes from $\mathcal{S}$ to the algebra $\mathbb{Z}\mathcal{L}$, and sends a singular link to its resolutions according to the Vassiliev skein relation. The map carries through analogously for singular braids, where $SB_n$ is mapped to $\mathbb{C}B_n$ via the Vassiliev skein relation.  Both maps in the examples replace a singular crossing with a parity-reversing (combination of) crossing(s). The maps defined in Theorem~\ref{mapsingular} consider instead a non-trivial replacement for the singular crossing which switches the parity.
\end{remark}

\section*{Acknowledgements}
The first author gratefully acknowledges hospitality and support of the Institut Henri Poincaré (UAR 839 CNRS-Sorbonne Université), and LabEx CARMIN (ANR-10-LABX-59-01). 
The second author wishes to express her gratitude towards Scuola Internazionale Superiore di Studi Avanzati (SISSA - Trieste) and the National Technical University of Athens (NTUA - Athens).

\section*{Use of AI and conflict of interest declaration}
The authors declare no use of AI and no potential conflicts of interest.



\end{document}